\documentclass[11pt]{article}

\usepackage{amsmath, amsthm, amssymb, amscd}
\usepackage[numbers]{natbib}
\usepackage{graphicx}
\usepackage{bm}
\usepackage{tikz}
\usepackage[title]{appendix}

\usepackage{hyperref}
\usepackage{wrapfig}
\newcommand{\argmin}{\mathop{\rm arg\min}}
\newcommand{\argmax}{\mathop{\rm arg\max}}

\newcommand{\R}{\mathbb{R}}

\numberwithin{equation}{section}
\newtheorem{thm}{Theorem}
\newtheorem{lem}{Lemma}
\newtheorem{rem}{Remark}
\newtheorem{prop}{Proposition}

\newcommand{\Var}{\operatorname{Var}}

\newcommand{\Std}{\operatorname{Std}}
\newcommand{\Bin}{\operatorname{Bin}}
\newcommand{\Ber}{\operatorname{Ber}}
\newcommand{\Poi}{\operatorname{Poi}}
\newcommand{\E}{\mathbb{E}}
\renewcommand{\P}{\mathbb{P}}

\newcommand{\TV}{\operatorname{TV}}

\newcommand{\KL}{\operatorname{KL}}

\newcommand{\mH}{\mathcal{H}}
\newcommand{\mC}{\mathcal{C}}

\newcommand{\mA}{\mathcal{A}}
\newcommand{\mF}{\mathcal{F}}

\newcommand{\mE}{\mathcal{E}}
\newcommand{\mG}{\mathcal{G}}
\newcommand{\mD}{\mathcal{D}}
\newcommand{\mP}{\mathcal{P}}

\begin{document}

% "Title of the paper"
\title{The Le Cam distance between density estimation, Poisson processes and Gaussian white noise}

%\author{JSH}
\author{Kolyan Ray\footnote{The research leading to these results has received funding from the European Research Council under ERC Grant Agreement 320637.\newline Email: \href{mailto:k.m.ray@math.leidenuniv.nl}{k.m.ray@math.leidenuniv.nl}, \href{mailto:schmidthieberaj@math.leidenuniv.nl}{schmidthieberaj@math.leidenuniv.nl}}\; and Johannes Schmidt-Hieber
 \vspace{0.1cm} \\
 {\em Leiden University} }

% \author{Johannes Schmidt-Hieber
% \vspace{0.1cm} \\
% {\em CREST-ENSAE,} \\ {\em 3, Avenue Pierre Larousse, 92240 Malakoff} \\ {\small {\em Email:} \texttt{Johannes.Schmidt.Hieber@ensae.fr}}}

% \footnotetext[1]{The research of Axel Munk and Johannes Schmidt-Hieber was supported by DFG Grant FOR 916 and GK 1023.}
% \footnotetext[2]{Author for correspondence, email:  \texttt{munk@math.uni-goettingen.de}}

\date{}
\maketitle

\begin{abstract}
\noindent It is well-known that density estimation on the unit interval is asymptotically equivalent to a Gaussian white noise experiment, provided the densities have H\"older smoothness larger than $1/2$ and are uniformly bounded away from zero. We derive matching lower and constructive upper bounds for the Le Cam deficiencies between these experiments, with explicit dependence on both the sample size and the size of the densities in the parameter space. As a consequence, we derive sharp conditions on how small the densities can be for asymptotic equivalence to hold. The related case of Poisson intensity estimation is also treated.
\end{abstract}

\paragraph{AMS 2010 Subject Classification:}
Primary 62B15; secondary 62G05, 62G07, 62G20.
% 62Gxx		Nonparametric inference
% %62G05   	Estimation
%	62G07  	Density estimation
% %62G08   	Nonparametric Regression
% %62G20   	Asymptotic properties
% 

\paragraph{Keywords:} Asymptotic equivalence; Le Cam distance; density estimation; Poisson intensity estimation; Gaussian shift experiments.

\section{Introduction}
\label{sec.intro}

In nonparametric density estimation on the unit interval, we observe $n$ i.i.d. random variables from an unknown probability density $f$ supported on $[0,1]$. This model is closely related to Poisson intensity estimation, where we observe a Poisson process on $[0,1]$ with unknown intensity function $nf$. The notion of ``closeness" between these problems can be made precise via the Le Cam deficiency $\delta$ and Le Cam (pseudo-)distance $\Delta,$ which we recall in Appendix \ref{sec.LeCam}. If the parameter space $\Theta$ consists of densities $f$ on $[0,1]$ that are uniformly bounded away from zero and have H\"older smoothness larger than $1/2$, then a seminal result of Nussbaum \cite{nussbaum1996} establishes that these models are asymptotically equivalent in the Le Cam sense to the Gaussian white noise model where we observe the Gaussian process $(Y_t)_{t\in [0,1]}$ such that 
\begin{align}
	dY_t = 2\sqrt{f(t)} dt + n^{-1/2} dW_t, \quad t\in [0,1], \ \ f\in \Theta,
	\label{eq.GWN_in_intro}
\end{align}
with $(W_t)_{t\in[0,1]}$ a Brownian motion. Brown and Zhang \cite{brown1998} constructed a parameter space with H\"older smoothness exactly $1/2$ such that asymptotic equivalence fails to hold, thereby establishing the sharpness of the smoothness constraint.

The main goal of this article is to sharply quantify the rate of the Le Cam distance between these three models with explicit dependence on both the smoothness of the underlying function class and the size of the functions contained therein. To this end, we derive matching upper and lower bounds for the rates of the various Le Cam deficiencies under general conditions. As a by-product, we characterize exactly how small densities can be for asymptotic equivalence to hold between these models. This is of particular interest in Poisson intensity estimation, where low count data is characteristic of many applied problems. Furthermore, since our upper bound is constructive and provably sharp, it provides a blueprint to transform Poisson data into Gaussian data in an optimal way with respect to the Le Cam distance.

We henceforth take the parameter space $\Theta=\Theta_n$ to be a sample size dependent subspace of $\beta$-smooth H\"older densities. Such a notion is widely used in high-dimensional statistics and turns out to be natural in our setting as well. Density estimation is a qualitatively different problem for densities taking values near zero, both in terms of estimation rates \cite{patschkowski2016,ray2016IP} and asymptotic equivalence, as we show below. Indeed, an $n$-dependent threshold turns out to be the correct notion to characterize ``small densities", much as in the case of high-dimensional statistics. We show that under general conditions, the squared Le Cam deficiencies between either the density estimation experiment or Poisson intensity experiment and the corresponding Gaussian white noise model are of the order 
\begin{align}
	1\wedge n^{\frac{1-2\beta}{2\beta+1}} \sup_{f\in \Theta}\int_0^1 f(x) ^{-\frac{2\beta+3}{2\beta+1}} dx,
	\label{eq.rate_in_intro}
\end{align}
where $\wedge$ denotes the minimum. Our main restriction is that for the upper bound we require smoothness $\beta\leq 1.$ Recall that two experiments are said to be asymptotically equivalent if both deficiencies tend to zero. In particular, if $f$ is uniformly bounded away from zero, we recover the rate $1\wedge n^{(1-2\beta)/(2\beta+1)}$ and so asymptotic equivalence holds if and only if $\beta>1/2.$ 

The Le Cam distance between two experiments controls the maximal difference in statistical risk of decision problems with loss function bounded by one, see Strasser [35] and Le Cam and Yang [18]. In particular, if one solves any such decision problem by transforming Poisson data into Gaussian data, which is a common approach as discussed below, then the rate (1.2) provides a bound on the contribution to the risk from the data transformation when using an optimal transformation, for instance the one considered in this article. The Le Cam distance thus provides a sharp description of the statistical cost associated to reducing one problem to another and allows one to characterize the optimal such reduction.

Whilst explicit formulas for the Le Cam deficiency are known for some parametric models (cf. Torgersen \cite{Torgersen1991}, Section 8.5-8.6), the existing theory for the Le Cam distance between nonparametric models focuses on necessary and sufficient conditions for asymptotic equivalence. Explicit upper bounds for the Le Cam distance are, however, sometimes available. For the models we consider, Carter \cite{carter2002} obtained suboptimal upper bounds using a multinomial approximation. In view of the lower bound we prove here, the approach of Brown et al. \cite{brown2004} yields the correct rate in terms of $n$, but not $\Theta$, even though their result is not stated in this form. Our upper bound construction is related to the quantile coupling employed in \cite{brown2004}, though obtaining the correct dependence on the density $f$ in \eqref{eq.rate_in_intro} imposes significant additional technical challenges. Explicit upper bounds have also been obtained for various regression models \cite{grama1998,rohde2004,reiss2008,schmidt-hieber2014}. Existing lower bound results have focused on proving asymptotic nonequivalence of models rather than lower bounding the rate of the Le Cam deficiency, see \cite{efromovich1996,brown1998,wang2002,ray2018lower}.

To understand the advantage of having rates for the Le Cam deficiency beyond simply asymptotic equivalence, one can make the analogy with consistency versus convergence rates for an estimator. Consistency specifies that an estimator will eventually be close to the true parameter, but this may occur only for extremely large sample sizes. In contrast, rates of convergence allow for a much finer understanding of the performance of estimators and provide a framework to compare different procedures. Asymptotic equivalence is a qualitative statement that the experiments will be close in the limit, while the rates at which the deficiencies tend to zero provide a quantitative insight into the speed of this convergence.

A major motivating application for this work is nonparametric Poisson intensity estimation, where there is a long list of techniques on transforming Poisson data into approximately Gaussian data. These methods typically use local binning together with variations of the parametric square root transform, see for instance \cite{fryzlewicz2004,brown2009,Makitalo2011} or the recent survey article \cite{hohage2016}. Given that there are multiple proposed transformations, one would like a theoretical concept to compare the quality of the different transformations, in particular against some information-theoretic optimal benchmark. With regards to a large class of decision procedures, such a benchmark is provided by the Le Cam distance.

More abstractly, given two sequences of statistical experiments $\mE_n(\Theta) = (\Omega_n, \mA_n, (P_\theta^n: \theta \in \Theta))$ and $\mF_n(\Theta) = (\Omega_n', \mA_n', (Q_\theta^n: \theta \in \Theta)),$ a (measurable) map $M$ that sends probability measures $P_\theta^n$ to probability measures on the measurable space $(\Omega_n', \mA_n')$ represents a method to transform data arising in $\mE_n(\Theta)$ into data comparable to that generated in $\mF_n(\Theta).$ In particular, one seeks a method to convert data arising from $P_\theta^n$ into a ``synthetic" observation that is a good approximation to true data generated from the corresponding $Q_\theta^n$, uniformly over $\theta \in \Theta$. The quality of such an approximation can be measured by the total variation distance $\sup_{\theta \in \Theta} \|MP_\theta^n - Q_\theta^n\|_{\TV}.$ If this converges to zero, then no statistical test can asymptotically tell whether given data are transformed data originating from $\mE_n(\Theta)$ or true data from $\mF_n(\Theta).$ The Le Cam deficiency therefore provides a benchmark for optimality in this regard and a rate-optimal approximation can be defined as one such that the corresponding map $M^*$ attains this lower bound (up to constants):
\begin{align*}
	\sup_{\theta\in \Theta} \|M^*P_\theta^n -Q_\theta^n\|_{\TV} 
	\asymp 
	\inf_{M} \sup_{\theta \in \Theta} \|MP_\theta^n -Q_\theta^n\|_{\TV}
	=:
	\delta(\mE_n(\Theta), \mF_n(\Theta)) .
\end{align*}
In particular, since our upper bound on the Le Cam deficiency is constructive, one can deduce from it an approximation of the Poisson model by the Gaussian white noise model \eqref{eq.GWN_in_intro} that is rate-optimal in this sense.

While the sharpness of the smoothness condition in Nussbaum's result has been established, the extent to which one can relax the assumption that the densities must be uniformly bounded away from zero has received little study. A notable exception is Mariucci \cite{mariucci2016b}, who studies densities of the form $f \cdot g$, where $g$ is known and possibly small and $f$ is unknown and uniformly bounded away from zero. From an applied perspective, a uniform lower bound on the density is artificial and one would like to weaken this condition. Low Poisson counts occur in applications, such as image denoising, and existing results can rely on Gaussian approximations \cite{Makitalo2011}. This regime is not well-understood and it would therefore be useful to understand how such a Gaussian approximation behaves for small densities.

The rate \eqref{eq.rate_in_intro} allows us to characterize exactly how small a density can be for asymptotic equivalence to hold between density estimation or Poisson intensity estimation and the Gaussian model \eqref{eq.GWN_in_intro}. For example, if $\inf_{f\in \Theta} \inf_x f(x) \gg n^{(1-2\beta)/(2\beta+3)},$ then asymptotic equivalence still holds. Since ``small" is defined in \eqref{eq.rate_in_intro} in an integrated sense, even weaker assumptions are required if the densities are small on a shrinking set: for example asymptotic equivalence still holds if $\Theta$ contains densities of the form $f(x)  \propto x^\beta + n^{-\beta/(\beta+1)}s_n$, where $s_n \rightarrow \infty.$ Densities can therefore come arbitrarily close to the threshold $n^{-\beta/(\beta+1)},$ which turns out to be the absolute lower limit since, under very weak assumptions, asymptotic equivalence fails if $\inf_{f\in \Theta} \inf_x f(x) \lesssim n^{-\beta/(\beta+1)}$, see Theorem 1 of \cite{ray2018lower}.

One might naturally wonder why the rate of the Le Cam deficiency becomes slower if the parameter space contains small densities. A possible explanation is that the information about $f$ contained in the data is not the same in the different models. If $f$ is small in some interval, then in density estimation we observe very few observations in this region, whereas in the Gaussian white noise model \eqref{eq.GWN_in_intro} the whole path $(Y_t)_{t\in [0,1]}$ is observed and the difficulty lies rather in separating small signal from noise. Due to the different structures of these estimation problems, it seems reasonable that they are further apart in the Le Cam distance when the densities are small.

By the localization principle, it suffices to consider a local parameter space for upper bounds on the Le Cam distance. Sharp estimation rates are therefore crucial, since they determine the size of the local parameter space. In both density estimation and the Gaussian white noise model \eqref{eq.GWN_in_intro}, small densities can be estimated with a faster pointwise rate of convergence recently derived in \cite{patschkowski2016} and \cite{ray2016IP}. If $f$ is $\beta$-smooth in an appropriate sense, then the pointwise estimation rate at any $x\in (0,1)$ is, up to $\log n$ factors,
\begin{align}
	n^{-\frac{\beta}{\beta+1}} + \Big(\frac {f(x)}n \Big)^{\frac{\beta}{2\beta+1}}.
	\label{eq.ptw_conv_rate}
\end{align}
If $f(x)$ is larger than $n^{-\beta/(\beta+1)}$ then the rate is of order $(f(x)/n)^{\frac{\beta}{2\beta+1}},$ while if $f(x)$ is very small, in the sense that $f(x)\leq n^{-\beta/(\beta+1)},$ then the convergence rate is $n^{-\frac{\beta}{\beta+1}}.$ Small densities can therefore be estimated with faster rates of convergence. Note that if $f$ is bounded from below, we recover the standard $n^{-\beta/(2\beta+1)}$-rate of convergence. We shall refer to $f(x)\geq n^{-\beta/(\beta+1)}$ as the {\it regular regime} and to $f(x)\leq n^{-\beta/(\beta+1)}$ as the {\it irregular regime.} While the faster convergence rate for small densities means we can localize better, this does not translate into better rates for the Le Cam distance, since for small densities the local approximations are much worse.

Assuming known smoothness $\beta$, one can use a density or Poisson intensity estimator $\widehat{f}_n$ to find a local parameter space $\Theta(\widehat f_n)$ containing the true density with high probability and whose size is determined by the estimation rate. It then suffices to restrict to this local parameter space and the rate of the Le Cam distance is of the possibly much faster order 
\begin{align*}
	1\wedge n^{\frac{1-2\beta}{2\beta+1}} \sup_{f\in \Theta(\widehat f_n)}\int_0^1 f(x) ^{-\frac{2\beta+3}{2\beta+1}} dx
	\asymp 1\wedge n^{\frac{1-2\beta}{2\beta+1}} \int_0^1 \widehat f_n(x) ^{-\frac{2\beta+3}{2\beta+1}} dx,
\end{align*}
provided $\widehat{f}_n$ achieves the pointwise estimation rate  \eqref{eq.ptw_conv_rate}. We may thus obtain faster rates for the \emph{local} asymptotic equivalence of these models compared with their \emph{global} asymptotic equivalence. Local asymptotic equivalence has been studied for example in \cite{grama2006,butucea2018}. Given this, $\widehat f_n$ can also be used to check whether the density lies in the regime where local asymptotic equivalence holds. Plugging $\widehat{f}_n$ into the rate \eqref{eq.rate_in_intro} yields the estimate $I_n(\widehat{f}_n) := 1 \wedge n^{(1-2\beta)/(2\beta+1)} \int_0^1 \widehat{f}_n(x) ^{-(2\beta+3)/(2\beta+1)} dx$, which with high probability gives the order of the Le Cam distance over the local parameter space $\Theta (\widehat{f}_n)$. In particular, if $f_0$ has points in the irregular regime, then $I_n (\widehat{f}_n)$ will typically be close to one. This provides a practical pre-test to verify, for example, if a Gaussian approximation is suitable for low-count Poisson data.

Although for small densities, density estimation and the Gaussian white noise model \eqref{eq.GWN_in_intro} are no longer asymptotically equivalent, many aspects of their statistical theory, such as consistent testing, remain the same, see \cite{ray2018lower} for further discussion. Indeed, the fact that many statistical decision problems have nearly the same asymptotic properties in these three models irrespective of the underlying density size makes it difficult to prove lower bounds for the Le Cam deficiencies. In the regular regime, that is if $\inf_{f\in \Theta} \inf_x f(x) \gg n^{-\beta/(\beta+1)},$ we bound the Le Cam deficiency from below by the difference of the  Bayes risks for a decision problem on a discrete parameter space equipped with a non-uniform prior. Considering non-uniform priors seems necessary here in order to achieve the correct rate. The construction of the lower bounds provides many insights regarding the sense in which these models differ. 

Mathematically, many of our techniques build on earlier works on asymptotic equivalence, in particular Nussbaum \cite{nussbaum1996}, Brown and Zhang \cite{brown1998}, Brown et al. \cite{brown2004} and Low and Zhou \cite{low2007}. While the upper bounds expand many existing techniques, the lower bounds require several new concepts. Other works on asymptotic equivalence include J\"ahnisch and Nussbaum \cite{jahnisch2003} for density estimation and Genon-Catalot et al. \cite{genon-catalot2002} and Meister and Rei\ss \ \cite{meister2013} for Poisson intensity estimation.

{\it Notation:} For two positive sequence $(a_n)_n$ and $(b_n)_n$, we write $a_n \lesssim b_n$ if there is exists a constant $C$ independent of $n$, such that $a_n \leq C b_n$ for all $n\geq n_0$ and some $n_0 \geq 1$. If $a_n \lesssim b_n$ and $b_n \lesssim a_n,$ we write $a_n \asymp b_n.$ Similarly, $a_n \ll b_n$ means $\lim_{n\rightarrow \infty} a_n/b_n = 0$. In some proofs, we additionally require that the constant does not depend on certain parameters and we always indicate this at the beginning of the proof. For two functions $f,g$ defined on the same domain, we write $f\leq g$ if $f(x) \leq g(x)$ for all $x.$ Let $\|\cdot\|_p$ denotes the usual $L^p$-norm. Given two probability measures $P,Q$ defined on the same measurable space, the total variation distance, Hellinger distance and Kullback-Leibler divergence are denoted by $\|P-Q\|_{\TV},$ $H(P,Q)$ and $\KL(P,Q)$ respectively.

\section{Main results}
\label{sec.main}

We now formally define the three statistical experiments considered in this article.

{\it Density estimation $\mE_n^D(\Theta)$:} We observe $n$ i.i.d. copies $X_1,\ldots,X_n$ of a random variable on $[0,1]$ with unknown Lebesgue density $f.$ The corresponding statistical experiment is $\mE_n^D(\Theta)=([0,1]^n ,\sigma([0,1]^n), (P_f^n : f\in \Theta))$ with $P_f^n$ the product probability measure of $X_1, \ldots,X_n.$

{\it Poisson intensity estimation $\mE_n^P(\Theta)$:} We observe a Poisson process on $[0,1]$ with intensity function $nf$ and unknown density $f\in \Theta$. We thus observe the point process $\sum_{i=1}^N \delta_{X_i}$, where $X_1,X_2,\dots$ are i.i.d. random variables with density $f$, $N$ is an independent Poisson($n$) random variable and $\delta_x$ is the Dirac measure at $x$. This is equivalent to observing $X_1,\dots,X_N$. Denoting the distribution of this point process by $\overline P_f^n$, we can write the corresponding statistical experiment as $\mE_n^P(\Theta)=(\mathbb{M},\mathcal{M},(\overline P_f^n : f\in \Theta)),$ where $\mathbb{M}$ is the space of point measures equipped with the appropriate $\sigma$-algebra $\mathcal{M}$, see Section 4 of \cite{nussbaum1996} for further details.

{\it Gaussian white noise experiment $\mE_n^G(\Theta)$:} We observe the Gaussian process $(Y_t)_{t\in [0,1]}$ given by $dY_t = 2\sqrt{f(t)} dt + n^{-1/2} dW_t,$ $t\in [0,1],$ where $f\in \Theta$ is unknown and $W$ is a Brownian motion. The Gaussian white noise experiment is $\mE_n^G(\Theta)=(\mC([0,1]) ,\sigma(\mC([0,1])), (Q_f^n : f\in \Theta))$ with $Q_f^n$ the distribution of $(Y_t)_{t\in [0,1]},$ $\mC([0,1])$ the space of continuous functions on $[0,1]$ and $\sigma(\mC([0,1]))$ the $\sigma$-algebra generated by the open sets with respect to the uniform norm.

{\it Function spaces:} Denote by $\lfloor \beta \rfloor$ the largest integer strictly smaller than $\beta.$ The H\"older semi-norm is given by $|f|_{\mC^\beta} := \sup_{x\neq y, x,y\in [0,1]} |f^{(\lfloor \beta \rfloor)}(x) - f^{(\lfloor \beta \rfloor)}(y)| /|x-y|^{\beta - \lfloor \beta \rfloor}$ and the H\"older norm is $\| f \|_{\mC^\beta} := \| f\|_\infty + \| f^{(\lfloor\beta\rfloor)} \|_\infty + |f|_{\mC^\beta}.$ Consider the space of $\beta$-smooth H\"older densities with H\"older norm bounded by $R,$
\begin{equation*}
\mC^\beta(R) := \big\{ f: [0,1] \rightarrow \R \  :  \  f\geq 0, \ \int_0^1 f(u) du =1, \  f^{(\lfloor \beta \rfloor)} \text{ exists}, \  \| f \|_{\mC^\beta} \leq R \big\}. 
\end{equation*}
If $f$ is allowed to depend on $n$ and $0<\beta \leq 2,$ the pointwise rate of estimation at any $x \in (0,1)$ over the parameter space $\mC^\beta(R)$ is given by \eqref{eq.ptw_conv_rate}, up to $\log n$-factors (see Theorems 3.1 and 3.3 of \cite{patschkowski2016}  and Theorems 1 and 2 of \cite{ray2016IP}). This rate of convergence does not extend beyond $\beta=2$ using the usual definition of H\"older smoothness (Theorem 3 of \cite{ray2016IP}). To take advantage of higher order smoothness, we must therefore modify our function class.

A natural way to extend such rates to smoothness $\beta>2$ is to impose a shape constraint. On $\mC^\beta$ define the flatness seminorm $| f |_{\mathcal{H}^\beta} = \max_{1 \leq j <\beta } \| |f^{(j)}|^\beta/|f|^{\beta-j} \|_\infty^{1/j},$ with $0/0$ defined as $0$ and $|f|_{\mH^\beta}=0$ for $\beta \leq 1.$ The quantity $| f |_{\mathcal{H}^\beta}$ measures the flatness of a function near zero in the sense that if $f(x)$ is small, then the derivatives of $f$ must also be small in a neighbourhood of $x$. Define $\| f\|_{\mH^\beta} := \|f\|_{\mC^\beta} + |f|_{\mH^\beta}$ and consider the space of densities
\begin{align*}
	\mH^\beta(R) := \{ f\in \mC^\beta(R) \ : \  \|f\|_{\mH^\beta}\leq R\}.
\end{align*}
Notice that $\mH^\beta(R)=\mC^\beta(R)$ for $\beta\leq 1.$ Properties of the function space $\mH^\beta(R)$ are studied in \cite{RaySchmidt-Hieber2015c}.

We are now ready to state the main results, beginning with the upper bound for Poissonization. The proof of the following theorem is given in Section \ref{sec.DE_PIE}.

\begin{thm}[Upper bound between density and Poisson intensity estimation]
\label{thm.ub_DE_PIE}
If $\Theta \subset \mH^\beta(R)$ for $\beta>0,$ then
\begin{align*}
	\Delta(\mE_n^D(\Theta), \mE_n^P(\Theta))^2 \lesssim n^{-\frac{2\beta}{2\beta+1}}\log^2 n \  \sup_{f\in \Theta}  \int_0^1 \Big( \frac 1{f(x)} \wedge n^{\frac{\beta}{\beta+1}}\Big)^{\frac{1}{2\beta+1}}  dx.
\end{align*}
\end{thm}

We deduce that the squared Le Cam distance is of order at most $n^{-\frac{\beta}{\beta+1}}\log^2 n$ and so asymptotic equivalence holds for any $\beta>0$ irrespective of the size of the densities in $\Theta$. If the densities are uniformly bounded away from zero then this rate improves to $n^{-\frac{2\beta}{2\beta+1}}\log^2 n.$ The $\log^2 n$ factor is an artifact of the proof.

Poisson intensity estimation is equivalent to observing $N \sim \Poi(n)$ i.i.d. observations from the density $f$. Since $N=n + O_P (\sqrt{n})$, one can compare this to the statistical information contained in $\sqrt{n}$ additional observations. Mammen \cite{mammen1986} showed that for smooth parametric i.i.d. models, adding $r_n$ observations changes the squared Le Cam distance by $O(r_n^2/n^2)$. Heuristically, the corresponding bound for a $d$-dimensional parameter with explicit dependence on $d$ is $O(dr_n^2 /n^2)$. The rate in Theorem \ref{thm.ub_DE_PIE} can be viewed as a nonparametric analogue. Indeed, we show in Section \ref{sec.heuristics} that there is an effective parameter dimension $m_n\rightarrow \infty$ such that the rate equals
\begin{align}
	\frac{m_n r_n^2}{n^2}
	\label{eq.rate_Poissonization_general}
\end{align}
with $r_n = \sqrt{n}.$ In the parametric case ``$\beta = \infty$", we recover the rate $O(r_n^2/n^2) = O(1/n).$

\begin{thm}[Upper bound between Poisson intensity estimation and Gaussian white noise]
\label{thm.ub}
Let $\tfrac 12 <\beta \leq 1.$ If $\Theta \subset \mH^\beta(R)$ and $ \inf_{f\in \Theta} \inf_x f(x) \gg n^{-\frac{\beta}{\beta+1}}\log^8 n,$ then
\begin{align*}
	\Delta(\mE_n^D(\Theta), \mE_n^G(\Theta))^2
	+\Delta(\mE_n^P(\Theta), \mE_n^G(\Theta))^2 \lesssim 1\wedge n^{\frac{1-2\beta}{2\beta+1}} \sup_{f\in \Theta}  \int_0^1 f(x)^{-\frac{2\beta+3}{2\beta+1}} dx.
\end{align*}
\end{thm}

The statement assumes smoothness $\beta>1/2$ since for $\beta \leq 1/2,$ asymptotic equivalence fails even if all densities are uniformly bounded away from zero \cite{brown1998}. The main restriction of this result is the assumption that $\beta \leq 1.$ As in \cite{brown2004}, our proof relies on a Haar wavelet decomposition and heavily exploits the fact that the Haar basis functions are locally constant and have disjoint support at a fixed resolution level, see Section \ref{sec.bds_for_couplings}. For tight upper bounds in the case $\beta>1,$  expansions with respect to more regular wavelets are required, but without the specific structure of the Haar wavelet the coupling of the empirical wavelet coefficients in our construction becomes infeasible. Since in dimension $d>1$ asymptotic equivalence is expected to hold for $\beta > d/2,$ the multivariate extension of our result requires different techniques. A heuristic discussion of the rate in Theorem \ref{thm.ub} is deferred to Section \ref{sec.heuristics}, since it relies on technical devices introduced in Section \ref{sec.PIE_GWN}.

The Le Cam distance $\Delta$ is a pseudo-metric on the class of statistical experiments with the same parameter space, see Appendix \ref{sec.LeCam}. To prove Theorem \ref{thm.ub}, it is therefore enough to establish the rate for $\Delta(\mE_n^P(\Theta), \mE_n^G(\Theta))^2$ since by Theorem \ref{thm.ub_DE_PIE},
\begin{align*}
	\Delta(\mE_n^D(\Theta), \mE_n^G(\Theta))^2
	&\leq 2\Delta(\mE_n^D(\Theta), \mE_n^P(\Theta))^2+ 2\Delta(\mE_n^P(\Theta), \mE_n^G(\Theta))^2  \\
	& = 2\Delta(\mE_n^P(\Theta), \mE_n^G(\Theta))^2 + o\Big( 1\wedge n^{\frac{1-2\beta}{2\beta+1}} \sup_{f\in \Theta}  \int_0^1 f(x)^{-\frac{2\beta+3}{2\beta+1}} dx \Big).
\end{align*}

For the lower bounds on the Le Cam deficiencies, we must take the supremum over densities which are not isolated in the parameter space and thus need to introduce a suitable notion of interior parameter space.  As a neighbourhood of a density $f^*$, consider the band$$\mathcal{U}(f^*):= \{ f\in \mH^\beta(R): \tfrac 12 f^* \leq f \leq 2f^*\}.$$ Given a parameter space $\Theta\subset \mH^\beta(R),$ let $R'<R$ be fixed. Define the interior parameter space $\Theta_0$ as the space of all $f\in \Theta \cap \mH^\beta(R')$ such that $\mathcal{U}(f) \subset \Theta.$ The dependence of $\Theta_0$ on $R'$ is omitted. For example, for an arbitrary sequence $(\delta_n)$ consider the parameter space $\Theta = \{ f\in \mH^\beta(R): f\geq \delta_n\}.$ The corresponding interior parameter space is then $\Theta_0 = \{ f\in \mH^\beta(R'): f\geq 2\delta_n\}.$

For the lower bounds, we distinguish between the regular and irregular regimes, that is whether $\inf_{f_0 \in \Theta_0} \inf_{x_0} f(x_0)$ is larger or smaller than $n^{-\beta/(\beta+1)}.$ In the irregular case, asymptotic equivalence always fails under very weak assumptions on the parameter space, see Theorem 1 of \cite{ray2018lower}. The level $n^{-\beta/(\beta+1)}$ is a fundamental threshold separating the ``small" and ``large" density regimes from a statistical perspective, as can be seen by the qualitatively different minimax estimation rates in \eqref{eq.ptw_conv_rate}. One way to view this is through the bias-variance tradeoff for estimation in the Gaussian model \eqref{eq.GWN_in_intro}. In the regular regime, one obtains the classical nonparametric bias-variance tradeoff, while in the irregular regime, the variance of an optimal estimator is strictly larger than its bias. Another perspective is the information geometry of the problem, measured through the Hellinger distance, which behaves differently in these two regimes. In the ``large regime", it behaves like the $L^2$-distance, thereby leading to the usual classical nonparametric behaviour, including the rate. As a density approaches zero however, the Hellinger distance behaves more like the $L^1$-distance, leading to the same rates occurring in irregular models, such as in nonparametric regression with one-sided errors \cite{jirak2014}. For further discussion see \cite{patschkowski2016,ray2016IP}.

\begin{thm}[Lower bound between Poisson intensity estimation and Gaussian white noise]
\label{thm.lb}
If $\Theta \subset \mH^\beta(R)$ for $\beta>0$ and $\inf_{f_0 \in \Theta_0} \inf_{x_0} f_0(x_0) \gg n^{-\beta/(\beta+1)},$ then there exists an integer $n_0$ such that for all $n\geq n_0,$
\begin{align*}
	\delta(\mE_n^P(\Theta), \mE_n^G(\Theta)) ^2 \wedge \delta( \mE_n^G(\Theta), \mE_n^P(\Theta)) ^2
	&\gtrsim
	 1 \wedge   n^{\frac{1-2\beta}{2\beta+1}} \sup_{f\in \Theta_0} \int_0^1 f(x)^{-\frac{2\beta+3}{2\beta+1}} dx.
\end{align*}
\end{thm}

For sufficiently large $n,$ the lower bound matches the rate obtained in Theorem \ref{thm.ub}, provided that the supremum over $f\in \Theta$ is of the same order as the supremum over $f\in \Theta_0.$ As in \cite{brown1998}, the proof is based on the construction of a decision problem and comparison of the Bayes risk in the two experiments, which yields a lower bound on the Le Cam deficiency. Since we are interested in the rates of the Le Cam deficiencies, the exact Bayes risks must be approximated up to second order. In fact, we explicitly construct a separate decision problem for every parameter $f\in \Theta_0,$ which quantifies how well we can separate $f$ from elements in the local neighbourhood $\mathcal{U}(f)$.

\begin{thm}[The Le Cam deficiencies between density estimation and Gaussian white noise]
\label{thm.lb2}
Let $\tfrac 12 <\beta \leq 1.$ If $\Theta \subset \mH^\beta(R)$, $ \inf_{f\in \Theta} \inf_x f(x) \gg n^{-\frac{\beta}{\beta+1}}\log^8 n$ and
\begin{align}
	1\wedge n^{\frac{1-2\beta}{2\beta+1}} \sup_{f\in \Theta_0 } \int_0^1 f(x)^{-\frac{2\beta+3}{2\beta+1}} dx
	\asymp 
	1\wedge n^{\frac{1-2\beta}{2\beta+1}}\sup_{f\in \Theta}\int_0^1 f(x)^{-\frac{2\beta+3}{2\beta+1}} dx,
	\label{eq.1st_assump_thmlb2}
\end{align}
then there exists an integer $n_0$ such that for all $n\geq n_0,$
\begin{align*}
	\delta(\mE_n^D(\Theta), \mE_n^G(\Theta)) ^2 \asymp \delta( \mE_n^G(\Theta), \mE_n^D(\Theta)) ^2
	\asymp 1\wedge n^{\frac{1-2\beta}{2\beta+1}} \sup_{f\in \Theta} \int_0^1 f(x)^{-\frac{2\beta+3}{2\beta+1}} dx.
\end{align*}
\end{thm}

The remaining sections are structured as follows. In Sections \ref{sec.DE_PIE} and \ref{sec.PIE_GWN}, we derive upper bounds for the Le Cam distance and prove Theorems \ref{thm.ub_DE_PIE} and \ref{thm.ub}. Some heuristics behind the rates for Poissonization and Gaussian approximation are presented in Section \ref{sec.heuristics}. Lower bounds can be found in Section \ref{sec.lbs_regular}, where we provide the proofs of Theorems \ref{thm.lb} and \ref{thm.lb2}. Technical results are deferred to the appendix, which also contains a brief summary of the Le Cam deficiency in Appendix \ref{sec.LeCam}.

\section{Asymptotic equivalence between density estimation and Poisson intensity estimation}
\label{sec.DE_PIE}

We now prove Theorem \ref{thm.ub_DE_PIE}, which states that if $\Theta \subset \mH^\beta(R)$ for some $\beta>0,$ then
\begin{align}
	\Delta(\mE_n^D(\Theta), \mE_n^P(\Theta))^2 
	&\lesssim n^{-\frac{2\beta}{2\beta+1}}\log^2 n \  \sup_{f\in \Theta}  \int_0^1 \Big( \frac 1{f(x)} \wedge n^{\frac{\beta}{\beta+1}}\Big)^{\frac{1}{2\beta+1}}  dx \notag \\
	&\leq n^{-\frac{\beta}{\beta+1}}\log^2 n \rightarrow 0.
	\label{eq.Poissonization_to_show}
\end{align}
The two experiments differ in the number of i.i.d. copies of $X\sim f$ which are observed. In the density estimation model, we observe $n$ copies and in the Poisson intensity model $N$ copies, where $N$ is drawn from a Poisson distribution with intensity $n.$ One strategy to bound the Le Cam distance is to `synchronize' the models in the sense that (pseudo)-observations are generated in the model with fewer observations. Proposition 4.1 in \cite{nussbaum1996} and \cite{LeCamY2000}, p.73 establish bounds based on this idea (see also the related earlier work of Le Cam \cite{LeCam1974} and Mammen \cite{mammen1986}). Asymptotic equivalence of the density and Poisson experiments then holds for H\"older balls whenever the H\"older index is larger than $1/2.$ A slightly different approach was employed by Low and Zhou \cite{low2007}, which gives asymptotic equivalence for all H\"older balls with positive smoothness index. Below, we show that combining this technique with the faster convergence rates for estimation of small signals yields the rate \eqref{eq.Poissonization_to_show}.

A key ingredient in the proof of Theorem \ref{thm.ub_DE_PIE} is the localization principle that we recall in Appendix \ref{sec.LeCam}. More precisely, we apply Lemma \ref{lem.localization_bd} to the local parameter space
\begin{align*}
	\Theta_1^\beta(f_0) := \Big\{ f\in \Theta  :  \big| f(x)-f_0(x) \big| \leq C \Big(\frac{\log n}{n}\Big)^{\frac{\beta}{\beta+1}}+ C \Big(\frac{\log n}{n}f_0(x)\Big)^{\frac{\beta}{2\beta+1}}, \ \forall x\in[0,1]\Big\}
\end{align*}
with $C$ some sufficiently large constant. The constants $R$ and $C$ are of no importance and therefore omitted in the notation. The right-hand side is the upper bound on the pointwise convergence rate given in \eqref{eq.ptw_conv_rate}, up to logarithmic factors. The next result establishes the rate of convergence for the Le Cam distance on the local parameter space $\Theta_1^\beta(f_0).$ The proof is given in Appendix \ref{eq.proof_sec_DE_PIE}.

\begin{thm}
\label{thm.Poissonization}
For any $\beta>0,$ 
\begin{align*}
	\Delta \big(\mE_n^D(\Theta_1^\beta(f_0)), \mE_n^P (\Theta_1^\beta(f_0)) \big)^2 \lesssim n^{-\frac{2\beta}{2\beta+1}}\log^2 n \ \int_0^1 \Big( \frac 1{f_0(x)} \wedge n^{\frac{\beta}{\beta+1}}\Big)^{\frac{1}{2\beta+1}}  dx.
\end{align*}
\end{thm}

\begin{thm}
\label{thm.globalization_Poissonization}
Let $\beta>0$ and $\Theta \subset \mH^\beta(R).$ In the nonparametric density estimation experiment $\mE_n^D(\Theta),$ there exists an estimator $\widehat f_n$ taking values in a finite subset of $\Theta$ which satisfies
\begin{align*}
	\inf_{f_0 \in \Theta} P_{f_0}^n\Big( f_0 \in \Theta_1^\beta \big(\widehat f_n \big)\Big) =1 -O(n^{-1}),
\end{align*}
provided the constant $C$ in the definition of $\Theta_1^\beta(f_0)$ is chosen large enough. Moreover, there exists an estimator in $\mE_n^P(\Theta)$ with the same properties.
\end{thm}

The proof can be found in Appendix \ref{sec.global}. The rate \eqref{eq.Poissonization_to_show} is now a direct consequence of Lemmas \ref{lem.localization_bd} and \ref{lem.sample_splitting}, which allow one to piece together a global Markov kernel using the estimator from Theorem \ref{thm.globalization_Poissonization} and local Markov kernels from Theorem \ref{thm.Poissonization}. This completes the proof of Theorem \ref{thm.ub_DE_PIE}.

\section{Asymptotic equivalence between Poisson intensity estimation and Gaussian white noise}
\label{sec.PIE_GWN}

To establish the rate of the Le Cam distance between the Poisson intensity estimation experiment and the Gaussian white noise experiment, Section \ref{sec.Haar_wavelets} introduces a suitable local parameter space together with an orthonormal basis of $L^2[0,1]$ which depends on this space. The Poisson process is expanded with respect to this basis in Section \ref{sec.Poisson_count_model}. The same is done for the Gaussian white noise model in Section \ref{sec.Gaussian_seq_model}. It then remains to couple the empirical basis coefficients in the Gaussian and Poisson models. In Section \ref{sec.bds_for_couplings} we discuss general bounds on the Hellinger distance, which are then applied to the specific problem in Section \ref{sec.LeCam_for_seq_space_models}. The proof is completed in Section \ref{sec.compl_proof_ub}.

\subsection{Localization and basis expansion}
\label{sec.Haar_wavelets}

As in the proof of Theorem \ref{thm.ub_DE_PIE}, we apply the localization principle (see Section \ref{sec.LeCam}) and consider for any $f_0\in \Theta\subset \mH^\beta(R)$ the local parameter space
\begin{align*}
	&\Theta^\beta(f_0) =\Theta_{C,R}^\beta(f_0) \\
	&:= \Big\{ f\in \Theta: \ \frac{1}{32}f_0 \leq f \leq 32 f_0\ \text{and} \ n \int_0^1 \frac{(f(x) -f_0(x))^4}{ f_0(x)^3} dx
	\leq C n^{\frac{1-2\beta}{2\beta+1}} \int_0^1 f_0(x)^{-\frac{2\beta+3}{2\beta+1}}dx \Big\},
\end{align*}
for a sufficiently large constant $C$, depending only on $R$ and $\beta$. By \eqref{eq.ptw_conv_rate}, the convergence rate for estimation of $f(x)$ in the regular regime is $(f_0(x)/n)^{\beta/(2\beta+1)}$ up to $\log n$ factors. Replacing $f(x) -f_0(x)$ by $C^{1/4}(f_0(x)/n)^{\beta/(2\beta+1)}$ in the definition of $\Theta^\beta(f_0)$ then yields equality. The localization constraint is written via integrals rather than pointwise to prevent unnecessary $\log n$ factors in the rate of the Le Cam distance. Localization using integral constraints was also used in Section 2.2 of Dalalyan and Rei\ss \ \cite{dalalyan2006}.

From now on let us work on $\Theta^\beta(f_0).$ A common approach in asymptotic equivalence is to further split the localized experiment into so-called \textit{doubly local} experiments (cf. Grama and Nussbaum \cite{grama1998}), such that on each of these single subexperiments, the unknown parameter can be estimated at the localization rate in the definition of $\Theta^\beta(f_0).$ Since $f_0$ is known in the local experiment, we may use it to define a partition of $[0,1],$ which provides the appropriate shrinking intervals generating the doubly local experiments. Define $z_0:=0$ and $z_{i+1} := z_i + (f_0(z_i)/n)^{1/(2\beta+1)}.$ Let $m$ be the index of the largest $z_i$ smaller than $1.$ Define the boundary corrected version $(x_i)_{i=0,\ldots,m}$ as
\begin{align}
	x_i:=z_i \quad \text{for} \ i<m  \quad \text{and} \ x_m:=1.
	\label{eq.xi_def}
\end{align}
Further write 
\begin{align}
	\Delta_i := x_i-x_{i-1} = \Big(\frac{f_0(x_{i-1})}n\Big)^{1/(2\beta+1)} + (1-z_m) \mathbf{1}(i=m).
	\label{eq.Deltai_def}
\end{align}	
By assumption $\inf _{f_0 \in \Theta}\inf_{x} f_0(x) \gg n^{-\frac{\beta}{\beta+1}}$ and so, for any positive constant $c$ and sufficiently large $n,$ $(f_0(x)/n)^{1/(2\beta+1)} \leq c (f_0(x)/R)^{1/\beta}$ for all $x.$ Applying Lemma \ref{lem.fx_local_bd} gives $\tfrac{1}{2} f_0(z_{j-1}) \leq f_0(x) \leq 2f_0(z_{j-1})$ for all $x\in [z_{j-1},z_j]$ and all $j=1,\ldots,m.$ Since $1-z_m \leq (f_0(z_m)/n)^{1/(2\beta+1)},$ we obtain for the remainder term
\begin{align}
	(f_0(x_{m-1})/n)^{1/(2\beta+1)} \leq \Delta_m \leq 3(f_0(x_{m-1})/n)^{1/(2\beta+1)}.
	\label{eq.Deltam_ineqs}
\end{align}
This also shows that for any fixed positive constant $c,$ $\Delta_m \leq c (f_0(x_{m-1})/R)^{1/\beta}$ provided $n$ is sufficiently large. Applying  Lemma \ref{lem.fx_local_bd} and $z_j=x_j$ for $j=1,\ldots,m-1$ yields
\begin{align}
	\frac{1}{2} f_0(x_{j-1}) \leq f_0(x) \leq 2f_0(x_{j-1}) \quad \text{for all} \ x\in [x_{j-1},x_j] \ \text{ and } \ j=1,\ldots,m.
	\label{eq.loc_comp_of_fcts_in_lb}
\end{align}
We thus obtain a second localization by further restricting the data in the local experiment with parameter space $\Theta^\beta(f_0)$ to the intervals $[x_{j-1},x_j]$. We motivate the specific choice of this decomposition by a heuristic argument showing it is natural in terms of double localization. The local parameter space $\Theta^\beta(f_0)$ is defined via an integral rather than pointwise constraint to avoid unnecessary $\log n$ factors in the rate, so for simplicity consider instead $\Theta_1^\beta(f_0)$ from Section \ref{sec.DE_PIE}. Since $\inf_x f_0(x) \gg n^{-\beta/(\beta+1)},$ this localization constraint essentially means that the density is known pointwise up to an error of order $(f_0(x)/n)^{\beta/(2\beta+1)}.$ To show that decomposing $[0,1]$ into the intervals $[x_{j-1}, x_j]$ is correct in the sense of double localization, we therefore have to show that on each interval $[x_{j-1}, x_j]$, the density can be estimated at the rate $(f_0(x)/n)^{\beta/(2\beta+1)}.$ In the Poisson experiment, the number of observations in each interval is 
\begin{align*}
	\# \{i : X_i \in [x_{j-1}, x_j]\} = n \int_{x_{j-1}}^{x_j} f(x) dx + O_P\Big( \sqrt{ n \int_{x_{j-1}}^{x_j} f(x) dx }\Big)
\end{align*}
and the estimator $\widehat f(x) := \# \{i : X_i \in [x_{j-1}, x_j]\}/(n \Delta_j) = f(x) + O_P((f_0(x)/n)^{\beta/(2\beta+1)}),$ $x\in [x_{j-1},x_j],$ thus has the correct rate. A similar result holds in the Gaussian white noise model, which completes the argument. We also note that for rate-optimal estimation of $f$ with smoothness $\beta \leq 1$, in both experiments it suffices to approximate $f$ by a function that is constant on each such interval, see the proof of Theorem \ref{thm.glob2}. The total number of such intervals $m_n$ can thus be viewed as the effective parameter dimension.

We define an orthonormal basis of $L^2[0,1]$ by decomposing $[0,1]$ into the intervals $[x_{i-1},x_i].$ Let $\psi = \mathbf{1}(\cdot \in [0,1/2))-\mathbf{1}(\cdot \in [1/2,1])$ be the Haar mother wavelet and set $\psi_{j,k} := 2^{j/2}\psi(2^j \cdot -k)$ as usual. Then $\{1\}\cup \{ \psi_{j,k} : j=0, 1, \ldots ; k=0, 1, \ldots, 2^j-1\}$ forms an orthonormal basis of $L^2[0,1].$ For the sequence $(x_i)_{i=1,\ldots,m}$ defined above, identify $L^2[0,1]$ with $ \bigotimes_{i=1}^m L^2[x_{i-1}, x_i]$ and consider the Haar basis on each of the intervals $[x_{i-1}, x_i],$ that is  $\phi_i := \Delta_i^{-1/2}\mathbf{1}(\cdot \in (x_{i-1}, x_i])$ and $\psi_{i,j,k} := \Delta_i^{-1/2} \psi_{j,k} ( \Delta_i^{-1}(\cdot-x_{i-1}) ).$ The support of $\psi_{i, j,k}$ is $I_{i,j,k}:=[x_{i-1}+\Delta_i k/2^{j} , x_{i-1}+\Delta_i (k+1)/2^{j}]$ and $\psi_{i,j,k}$ is positive on $I_{i,j,k}^+=I_{i,j+1,2k}$ and negative on $I_{i,j,k}^-:=I_{i,j+1,2k+1}.$ For any $i,$ $\{\phi_i \}\cup \{ \psi_{i, j,k} : j=0, 1, \ldots ; k=0, 1, \ldots, 2^j-1\}$ is an orthonormal basis of $L^2[x_{i-1},x_i].$ For $f\in L^2[0,1]$ write $c_i :=\int f(u) \phi_i(u) du=\Delta_i^{-1/2} \int_{x_{i-1}}^{x_i} f(u) du$ for the approximation coefficients and $d_{i,j,k} := \int f(u) \psi_{i,j,k}(u) du$ for the wavelet coefficients. With $$\Lambda :=\{(i,j,k): i=1,\ldots,m , j=-1,0,1,\ldots, k=0, \ldots, 0\vee (2^j-1)\},$$ $d_{i,-1,0}:=c_i,$ and $\psi_{i,-1,0}:=\phi_i,$  any $f\in L^2[0,1]$ can be decomposed as
\begin{align*}
	f = \sum_{i=1}^m c_i  \phi_i
	+ \sum_{i=1}^m \sum_{j=0}^\infty \sum_{k=0}^{2^{j-1}}  d_{i,j,k} \  \psi_{i,j,k}
	= \sum_{(i,j,k) \in \Lambda}  d_{i,j,k} \  \psi_{i,j,k}
\end{align*}
with convergence in $L^2[0,1]$.

\begin{lem}
\label{lem.wav_decay}
If $f\in \mH^\beta(R)$ with $0<\beta \leq 1,$ then for $j\geq 0,$ $|d_{i,j,k}| \leq R(2^{-j}\Delta_i)^{\beta+1/2}.$
\end{lem}

\begin{proof}
With $a_{i,j,k}:=x_{i-1}+\Delta_i k/2^j,$ $$d_{i,j,k} = (\Delta_i2^{-j})^{-1/2} \int_{a_{i,j,k}}^{a_{i,j,k}+\Delta_i/2^{j+1}} f(u) -f(u+\Delta_i/2^{j+1}) du.$$
Taking absolute values and using the H\"older continuity of $f$ yields the result.
\end{proof}

\subsection{Rewriting Poisson intensity estimation as a Poisson count model}
\label{sec.Poisson_count_model}

We now decompose the Poisson intensity experiment with respect to the basis from the previous section. For that define a new statistical experiment as follows. Let $(X_1,\ldots,X_N)$ be the jump times of a Poisson process on $[0,1]$ with time-varying intensity $x\mapsto nf(x).$ Define the counts  
\begin{align*}
	N_{i,j,k} &:= \#\{X_\ell \in I_{i,j,k}: \ell =1, \ldots, N\}, \quad (i,j,k) \in \Lambda, \ 0\leq j\leq \overline J+1,
\end{align*}
where $\overline{J}$ is the smallest integer larger than $3\log_2(n)$ and $I_{i,j,k}$ is the support of $\psi_{i,j,k}$ defined in the previous section. We thus have $N_{i,j,k} \sim \Poi(n\int_{I_{i,j,k}} f(u) du),$ and the counts $N_{i,j,k}$ and $N_{i',j',k'}$ are independent whenever $I_{i,j,k}$ and $I_{i',j',k'}$ are disjoint. Denote by $\overline P_{1,f}^n$ the distribution of the vector $(N_{i,j,k})_{(i,j,k) \in \Lambda, \ 0\leq j\leq \overline J+1}$ and by $s_n$ its length. With $\mP(\mathbb{N}^{s_n})$ the power set of $\mathbb{N}^{s_n},$ the Poisson count experiment $\mE_{1,n}^P(\Theta)$ is then defined as 
\begin{align*}
	\mE_{1,n}^P(\Theta) := \big( \mathbb{N}^{s_n}, \mP(\mathbb{N}^{s_n}), \big( \overline P_{1,f}^n: f\in \Theta\big) \big).
\end{align*}
On the local parameter space this experiment is close to $\mE_n^P(\Theta).$

\begin{prop}
\label{prop.mEP_mE1P}
Under the assumptions of Theorem \ref{thm.ub}, it holds that
\begin{align*}
	\Delta\big(\mE_{1,n}^P\big(\Theta^\beta(f_0)\big), \mE_n^P\big(\Theta^\beta(f_0)\big)\big)^2 = o(n^{-1}).
\end{align*}
\end{prop}

\begin{proof}
The experiment $\mE_n^P(\Theta^\beta(f_0))$ is by construction more informative than $\mE_{1,n}^P(\Theta^\beta(f_0)).$ It is thus enough to prove that the original Poisson intensity can be nearly reconstructed from the counts $(N_{i,j,k})_{(i,j,k) \in \Lambda, \ 0\leq j\leq \overline J+1}.$

Consider a Poisson process on $[0,1]$ with intensity $nf_n,$ where $f_n = \sum_{(i,j,k) \in \Lambda, j\leq \overline J}  d_{i,j,k} \  \psi_{i,j,k}.$ By construction, $\psi_{i,j,k}$ is constant on $I_{i,j,k}^+=I_{i,j+1,2k}$ and $I_{i,j,k}^-:=I_{i,j+1,2k+1}.$ Thus, $f_n$ is constant on the intervals $I_{i,\overline{J}+1, k}$ and therefore the counts on the highest resolution level $j= \overline J+1,$ that is $(N_{i,\overline J+1,k})_{i,k},$ form  a sufficient statistic for $f_n.$ Since counts on lower resolution levels can be constructed from $(N_{i,\overline J+1,k})_{i,k},$ we conclude that $(N_{i,\overline J+1,k})_{(i,j,k)\in \Lambda, 0\leq j\leq \overline{J}+1}$ is also a sufficient statistic for $f_n.$

By \eqref{eq.LeCam_ub_on_same_prob_space} it is enough to bound the squared Hellinger distance between a Poisson process with intensity $nf$ and a Poisson process with intensity $nf_n,$ uniformly over $f\in \Theta^\beta(f_0).$ Using Lemma \ref{lem.bds_of_info_distances}(i), the squared Hellinger distance is bounded from above by $n \int_0^1 \big(\sqrt{f(x)} - \sqrt{f_n(x)} \big)^2 dx.$ Together with Lemma \ref{lem.wav_decay} and $\inf_{f\in \Theta}\inf_x f(x) \geq n^{-1},$
\begin{align*}
	\Delta\big(\mE_{1,n}^P\big(\Theta^\beta(f_0)\big), \mE_n^P\big(\Theta^\beta(f_0)\big)\big)^2
	&\leq 
	\sup_{f\in \Theta^\beta(f_0)} n \int_0^1 \big(\sqrt{f(x)} - \sqrt{f_n(x)} \big)^2 dx\\
	&\leq \sup_{f\in \Theta^\beta(f_0)} n^2 \int_0^1 \big(f(x)-f_n(x)\big)^2 dx \\
	&=  \sup_{f\in \Theta^\beta(f_0)} n^2 \sum_{i=1}^m  \sum_{j>\overline J} \sum_{k=0}^{2^j-1} d_{i,j,k}^2\\
	&\leq R^2 n^2 \sum_{i=1}^m \Delta_i^{2\beta+1} 2^{-2\overline J \beta}
	= o(n^{-1}),
\end{align*}
since $\sum_{i=1}^m \Delta_i^{2\beta+1}\leq \sum_{i=1}^m \Delta_i =1,$ $\overline J >3\log_2(n)$ and $\beta>1/2.$
\end{proof}

\subsection{Sequence space representation of the Gaussian white noise experiment}
\label{sec.Gaussian_seq_model}

Given $f_0$ define the step function approximation $T_n f_0 = \sum_{i=1}^m f_0(x_{i-1})\mathbf{1}(\cdot \in [x_{i-1},x_i)) .$ On the local parameter space $\Theta^\beta(f_0),$ we introduce the statistical experiment $$\widetilde \mE_n^G(\Theta^\beta(f_0))=\big(\mC[0,1], \sigma(\mC[0,1]), (\widetilde Q_f^n: f\in \Theta^\beta(f_0))\big),$$ where $\widetilde Q_f^n$ is the distribution of the path $(\widetilde Y_t)_{t\in [0,1]}$ satisfying
\begin{align}
	d\widetilde Y_t = f(t) dt + n^{-1/2} \sqrt{T_n f_0(t)} dW_t, \quad t\in [0,1], \quad f\in \Theta(f_0).
	\label{eq.squared_mod}
\end{align}
The following proposition generalizes Theorem 2.7 in \cite{nussbaum1996} to small densities.

\begin{prop}
\label{prop.step_fct_in_variance}
Under the assumptions of Theorem \ref{thm.ub}, it holds that
$$\Delta\big(\mE_n^G(\Theta^\beta(f_0)), \widetilde\mE_n^G(\Theta^\beta(f_0)) \big)^2 \lesssim n^{\frac{1-2\beta}{2\beta+1}} \int_0^1 f_0(x)^{-\frac{2\beta+3}{2\beta+1}}dx.$$
\end{prop}

\begin{proof}
On $\Theta^\beta(f_0),$ the Gaussian white noise model is equivalent to observing $(U_t)_{t\in [0,1]}$ with $dU_t = 2(\sqrt{f(t)}-\sqrt{T_n f_0(t)}) dt + n^{-1/2} dW_t$ and observing $(\widetilde Y_t) _{t\in [0,1]}$ is equivalent to observing $(V_t)_{t\in [0,1]}$ with $dV_t =(f(t) - T_n f_0(t)) / \sqrt{T_n f_0(t)}dt + n^{-1/2} dW_t .$ Using \eqref{eq.LeCam_ub_on_same_prob_space}, Lemma \ref{lem.bds_of_info_distances}(ii), \eqref{eq.loc_comp_of_fcts_in_lb}, $f\in \Theta^\beta(f_0),$ $f_0\in \mH^\beta(R)$ and \eqref{eq.Deltam_ineqs}, we can bound the squared Le Cam distance $\Delta\big(\mE_n^G(\Theta^\beta(f_0)), \widetilde\mE_n^G(\Theta^\beta(f_0)) \big)^2$ by the supremum over $f\in \Theta^\beta(f_0)$ of
\begin{align*}
	&\frac{n}{2} \int_0^1 \Big(2(\sqrt{f(t)}-\sqrt{T_n f_0(t)}) - \frac{f(t)-T_n f_0(t)}{\sqrt{T_n f_0(t)}}\Big)^2 dt \\
	&= n \int_0^1 \frac{\big( \sqrt{f(t)} -\sqrt{T_n f_0(t)}\big)^4}{2T_n f_0(t)} dt \\
	&\leq 2^4 n \sum_{i=1}^m \int_{x_{i-1}}^{x_i}
	\frac{\big(f(t)-f_0(t)\big)^4+ \big(f_0(t)-f_0(x_{i-1})\big)^4}{f_0(x_{i-1})^3} dt\\
	&\leq 2^7 C  n^{\frac{1-2\beta}{2\beta+1}} \int_0^1 f_0(x)^{-\frac{2\beta+3}{2\beta+1}}dx 
	+ 2^4 R^4 n\sum_{i=1}^m \int_{x_{i-1}}^{x_i}
	\frac{\Delta_i^{4\beta}}{f_0(x_{i-1})^3} dt \\
	&\leq \big( 2^7 C + 2^73^{4\beta}R^4\big) n^{\frac{1-2\beta}{2\beta+1}} \int_0^1 f_0(x)^{-\frac{2\beta+3}{2\beta+1}}dx,
\end{align*}
which completes the proof.
\end{proof}
In the next step, we approximate \eqref{eq.squared_mod} by the following sequence space model. Denote by $Q_{1,f}^n$ the joint distribution of the (rescaled) empirical scaling and wavelet coefficients,
\begin{align*}
	Z_{i, -1,0}^* & := n\sqrt{\Delta_i} \int \phi_i(t) d\widetilde Y_t,  \quad \text{for} \ i=1,\ldots,m,\\
	Z_{i,j,k}^* &:= \sqrt{\frac{n}{f_0(x_{i-1})}}\int \psi_{i,j,k}(t) d\widetilde Y_t , \quad \text{for} \ (i,j,k)\in \Lambda, \ 0\leq j\leq \overline{J},
\end{align*}
where $\overline{J}$ is again the smallest integer larger than $3\log_2(n)$ (as in experiment $\mE_{1,n}^P(\Theta^\beta(f_0))$). Notice that the observations are independent and normally distributed with $$Z_{i, -1,0}^* \sim \mathcal N\Big(n \int_{x_{i-1}}^{x_i} f(t) dt, n \Delta_i f_0(x_{i-1})\Big) \ \text{ and } \ Z_{i,j,k}^* \sim \mathcal N\Big(\sqrt{\frac n{f_0(x_{i-1})}} d_{i,j,k}, 1\Big), \ \text{for} \  j\geq 0,$$where $d_{i,j,k} = \int f(t) \psi_{i,j,k}(t) dt.$ Write $s_n'$ for the total number of coefficients and define the experiment 
$$\mE_{1,n}^G(\Theta) :=\big(\mathbb{R}^{s_n'}, \sigma(\mathbb{R}^{s_n'}), \big(Q_{1,f}^n: f\in \Theta \big) \big).$$

\begin{prop}
Under the assumptions of Theorem \ref{thm.ub}, it holds that
\begin{align*}
	\Delta\big(\mE_{1,n}^G(\Theta^\beta(f_0)), \widetilde \mE_n^G(\Theta^\beta(f_0)) \big)^2  = o(n^{-1}).
\end{align*}
\end{prop}

\begin{proof}
Arguing as in the proof of Proposition \ref{prop.mEP_mE1P} using Lemma \ref{lem.bds_of_info_distances}(ii) instead of Lemma \ref{lem.bds_of_info_distances}(i) yields the result.
\end{proof}

\subsection{Information bounds for couplings}
\label{sec.bds_for_couplings}

At this point, we have transformed the Poisson intensity estimation and Gaussian experiments into sequence space experiments, where the empirical scaling and wavelet coefficients are observed. To relate these sequence models to each other, bounds on the information divergences between (transformed) Poisson and Gaussian random variables are discussed.

We firstly transform a Poisson random variable $N$ into a continuous random variable by adding an independent uniform variable $U$ on $[-\tfrac 12,\tfrac 12).$ From the sum $N+U$, we can recover $N$ by taking the nearest integer, which shows that this transformation is invertible. The sum can then be related to a normal random variable with the same mean and variance. To state the following result we write $H(X,Y):=H(P_X, P_Y)$ and $\KL(X,Y):=\KL(P_X,P_Y)$ if $X\sim P_X$ and $Y\sim P_Y.$

\begin{lem}
\label{lem.Hell_bd_scaling_coeffs}
Let $N\sim \Poi(\lambda)$ and $U$ be uniformly distributed on $[-\tfrac 12,\tfrac 12)$ and independent of $N.$ If $Z \sim \mathcal{N}(\lambda, \lambda),$ then
\begin{align*}
	KL(N+U, Z) = \frac{1}{8\lambda} (1+o(1)) \quad \text{as} \ \lambda \rightarrow \infty.
\end{align*}
Moreover, if $Z_0 \sim \mathcal{N}(\lambda, \lambda_0),$ then
\begin{align*}
	H^2(N+U, Z_0) \leq \frac{1}{4\lambda} (1+o(1)) + 4\Big(\frac{\lambda}{\lambda_0} - 1\Big)^2 \quad \text{as} \ \lambda \rightarrow \infty.
\end{align*}
\end{lem}

\begin{proof}
Denote the Lebesgue density of $N+U$ by $p$ and observe that on the interval $[k-\frac 12, k+\frac 12)$ this density equals $e^{-\lambda} \lambda^k/k!.$ Since $E[N+U] = \lambda,$ $\Var(N+U) =\Var(N)+\Var(U) = \lambda +\tfrac 1{12}$ and using the asymptotic expansion for the Poisson entropy (for instance Theorem 2 in \cite{Knessl1998}),
\begin{align*}
	KL(N+U, Z) 
	&= \sum_{k=0}^\infty\int_{k-1/2}^{k+1/2} \log\big( e^{-\lambda} \tfrac{\lambda^k}{k!} \sqrt{2\pi\lambda} e^{\frac 1{2\lambda} (x-\lambda)^2}  \big)p(x) dx \\
	&= \log (\sqrt{2\pi \lambda}) +\frac{1}{2\lambda} \big(\lambda + \frac 1{12}\big) + 
	\sum_{k=0}^\infty \log(e^{-\lambda} \lambda^k /k!) e^{-\lambda} \frac{\lambda^k}{k!} \\
	&=  \log (\sqrt{2\pi \lambda}) + \frac 1{2} +\frac 1{24\lambda} -\frac 12 \log(2\pi e\lambda) + \frac 1{12\lambda} +O\big(\lambda^{-2}) \\
	&= \frac{1}{8\lambda} (1+o(1))
\end{align*}
as $\lambda \rightarrow \infty.$ For the second statement, using that the Hellinger distance satisfies the triangle inequality and that the squared Hellinger distance is bounded by the Kullback-Leibler divergence (Lemma 2.4 of \cite{tsybakov2009}),
\begin{align*}
	H^2(N+U, Z_0) \leq 2 KL(N+U, Z) + 2 H^2(Z_0, Z) \leq \frac{1}{4\lambda} (1+o(1)) + 4 \Big(\frac{\lambda}{\lambda_0} - 1\Big)^2,
\end{align*}
where the bound for $H^2(Z_0, Z)$ follows from elementary computations.  
\end{proof}

If $N\sim \Poi(\lambda)$ and $N' \sim \Poi(\lambda')$ are independent, then $N | (N+N') \sim \Bin(N+N', \lambda/(\lambda+\lambda')),$ where $\Bin(m,p)$ denotes the binomial distribution with parameters $m$ and $0\leq p\leq 1.$ In experiment $\mE_{1,n}^P(\Theta^\beta(f_0)),$ the conditional distribution of the Poisson counts at resolution level $J+1$ given the Poisson counts at lower resolution levels $j\leq J$ is therefore 
\begin{align}
	N_{i,J+1,2k}| (N_{i,j,k})_{(i,j,k)\in \Lambda, 0\leq j\leq J} = N_{i,J+1,2k}|N_{i,J,k} \sim \Bin(N_{i,J,k}, p_{i,J,k})	
	\label{eq.NiJ+1k_cond}
\end{align}
with success probability
\begin{align}
	p_{i,J,k} := \frac{\int_{I_{i,J,k}^+} f(u) du}{\int_{I_{i,J,k}} f(u) du},
	\label{eq.pijk_id}
\end{align}
where $I_{i,j,k},$ $I_{i,j,k}^+$ are defined in Section \ref{sec.Haar_wavelets}. This property is tied to the Haar wavelet expansion and there is no natural extension to other wavelets or approximation schemes. In the corresponding Gaussian model $\mE_{1,n}^G(\Theta^\beta(f_0)),$ the observations are independent and normally distributed and therefore the conditional distributions are also normal. Working conditionally on lower resolution levels, we therefore need to couple binomial and Gaussian random variables. 

Notice that $p_{i,J,k} \approx 1/2$ with equality if $f$ is constant on $I_{i,J,k}.$ As in the Poisson case, we can make the distribution of $X_{m,p} \sim \Bin(m,p)$ continuous if we consider $X_{m,p}+U$ with $U$ uniform on $(-\tfrac 12, \frac 12]$ and independent of $X_{m,p}.$ Denote the c.d.f. of $X_{m,p}+U$ by $G_{m,p}$ and consider $\Phi^{-1} \circ G_{m,1/2}(X_{m,p}+U)$ with $\Phi^{-1}$ the quantile function of the standard normal distribution. The quantile transformation $\Phi^{-1} \circ G_{m,1/2}$ depends on $m$ but not on $p.$ Moreover, for $p=1/2,$ $\Phi^{-1} \circ G_{m,1/2}(X_{m,1/2}+U)\sim \mathcal N(0,1).$ For general $p$ this holds approximately and by Theorem 5 in \cite{brown2004}, 
\begin{align*}
	H^2\big(\mathcal{N}\big(\sqrt{m}(2p-1), 1\big), \Phi^{-1}\circ G_{m,1/2}(X_{m,p}+U) \big)
	\lesssim \big(p-\tfrac 12 \big)^2 + m\big(p-\tfrac 12\big)^4
\end{align*}
and the hidden constant does not depend on $m$ or $p.$ Using the triangle inequality and elementary computations, we obtain for any real number $\mu,$
\begin{align}
	H^2\big(\mathcal{N}\big(\mu, 1\big),  \Phi^{-1}\circ G_{m,1/2}(X_{m,p}+U) \big)
	\lesssim \big(\mu- \sqrt{m}(2p-1)\big)^2 + \big(p-\tfrac 12 \big)^2 + m\big(p-\tfrac 12\big)^4.
	\label{eq.Hell_bd_quantil_cpl}
\end{align}
Lemma \ref{lem.Hell_bd_scaling_coeffs} and \eqref{eq.Hell_bd_quantil_cpl} are used in the next section to bound the Le Cam distance between the sequence space experiments $\mE_{1,n}^P(\Theta^\beta(f_0))$ and $\mE_{1,n}^G(\Theta^\beta(f_0)).$

\subsection{Upper bound for the Le Cam distance between the Poisson and Gaussian sequence space experiments}
\label{sec.LeCam_for_seq_space_models}

In this section,  the proof of Theorem 3 in Brown et al. \cite{brown2004} is generalized to small densities. Recall that in experiment $\mE_{1,n}^P(\Theta^\beta(f_0))$ we observe the counts $(N_{i,j,k})_{(i,j,k) \in \Lambda, 0\leq j\leq \overline J+1}.$ Let $(U_{i,j,k})_{(i,j,k) \in \Lambda, 0\leq j\leq \overline J+1}$ be an i.i.d. sequence of uniform random variables on $(-\tfrac 12, \frac 12]$ which is independent of the Poisson counts.  Motivated by the previous section, define a new statistical experiment $\mE_{2,n}^P(\Theta^\beta(f_0)) = (\mathbb R^{s_n'}, \sigma(\mathbb R^{s_n'}), (\overline P_{2,f}^n: f\in \Theta^\beta(f_0)))$, where $\overline P_{2,f}^n$ is the distribution of the vector $(Z_{i,j,k})_{(i,j,k) \in \Lambda, 0\leq j\leq \overline J+1}$ with
\begin{align}
	Z_{i,-1,0} &:= N_{i,0,0}+U_{i,0,0}, \quad i=1,\ldots,m,  \notag \\
	Z_{i,j,k} & := \Phi^{-1}\circ G_{N_{i,j,k}, 1/2}(N_{i,j+1,2k}+U_{i,j,k}), \quad  (i,j,k) \in \Lambda, \ 0\leq j\leq \overline J.
	\label{eq.def_emp_wav_in_Poisson}
\end{align}
Since the function $\Phi^{-1} \circ G_{m, 1/2}$ is invertible, we can successively recover the Poisson counts $(N_{i,j,k})_{(i,j,k) \in \Lambda, 0\leq j\leq \overline J+1}$ from these observations and therefore
\begin{align*}
	\Delta \big( \mE_{1,n}^P(\Theta^\beta(f_0)), \mE_{2,n}^P(\Theta^\beta(f_0)) \big) =0.
\end{align*}
The experiment $\mE_{2,n}^P(\Theta^\beta(f_0))$ can now be compared to the Gaussian sequence experiment $\mE_{1,n}^G(\Theta^\beta(f_0)).$

\begin{prop}
\label{prop.seq_space_Le_Cam_bd}
Under the assumptions of Theorem \ref{thm.ub}, it holds that
\begin{align*}
	\Delta\big(\mE_{1,n}^G(\Theta^\beta(f_0)), \mE_{2,n}^P(\Theta^\beta(f_0))\big)^2 \lesssim n^{\frac{1-2\beta}{2\beta+1}}\  \int_0^1f_0(x)^{-\frac{2\beta+3}{2\beta+1}} dx.
\end{align*}
\end{prop}

\begin{proof} Let us begin with some notation. Write $p_{<J}$ and $p_{=J}$ for the joint density of $(Z_{i,j,k}^{*})_{(i,j,k)\in \Lambda, -1\leq j<J}$ and $(Z_{i,j,k}^{*})_{(i,j,k)\in \Lambda, j=J}$ respectively. Similarly, $q_{<J}$ denotes the joint density of $(Z_{i,j,k})_{(i,j,k)\in \Lambda, -1\leq j<J}$ and $q_{=J | <J}$ the density of the conditional distribution $(Z_{i,j,k})_{(i,j,k)\in \Lambda, j=J} | (Z_{i,j,k})_{(i,j,k)\in \Lambda, -1\leq j<J}.$

The random variables $(Z_{i,j,k}^{*})$ are independent and thus their joint densities factor into products. Expanding the squared Hellinger distance in a telescoping sum and then using this fact,
\begin{align}
	H^2\big(p_{<\overline J+1}, q_{<\overline J+1}\big)
	& = 2\big( 1 - \int \sqrt{p_{<0} q_{<0}}) + 2 \sum_{0\leq J \leq \overline J} \big( \int \sqrt{p_{<J} q_{<J}} - \int \sqrt{p_{<J+1} q_{<J+1}} \big) \notag \\
	& = H^2\big(p_{<0}, q_{<0}\big) + 2 \sum_{0\leq  J \leq \overline J} \int \sqrt{p_{<J} q_{<J}} \big(1- \int \sqrt{p_{=J} q_{=J|<J}}\big) \notag \\
	& = H^2\big(p_{<0}, q_{<0}\big) + \sum_{0\leq  J \leq \overline J}  \int \sqrt{p_{<J} q_{<J}}  H^2(p_{=J} ,q_{=J|<J}) .
	\label{eq.Hell_decomp}
\end{align}
On the lowest resolution level $j=-1,$ the Gaussian and Poisson random variables are independent and so $H^2\big(p_{<0}, q_{<0}\big) \leq  \sum_{i=1}^m H^2(Z_{i,-1,0}^{*}, Z_{i,-1,0})$ (\cite{str}, Lemma 2.17). Together with \eqref{eq.loc_comp_of_fcts_in_lb} and Lemma \ref{lem.Hell_bd_scaling_coeffs} applied to $\lambda = n\int_{x_{i-1}}^{x_i} f(u) du$ and $\lambda_0 =n\int_{x_{i-1}}^{x_i} f_0(u) du$ (noting that $\lambda, \lambda_0 \rightarrow \infty$ since $\inf_{f\in \Theta} \inf_x f(x) \gg n^{-\frac{\beta}{\beta+1}}$),
\begin{align}
	H^2\big(p_{<0}, q_{<0}\big)
	 \leq \sum_{i=1}^m \frac{1}{n \int_{x_{i-1}}^{x_i} f} 
	 + 16\sum_{i=1}^m  \Big( \frac{\int_{x_{i-1}}^{x_i} f(x)-f_0(x) dx}{\Delta_i f_0(x_{i-1}) } \Big)^2  =: (I) + (II),
	\label{eq.ub_small_res}
\end{align}
where $(I)$ and $(II)$ denote the first and second terms respectively. Since $f\in \Theta^\beta(f_0)$ and using \eqref{eq.Deltam_ineqs} and \eqref{eq.loc_comp_of_fcts_in_lb}, we find that  $\int_{x_{i-1}}^{x_i} f \geq 2^{-6}\Delta_i f_0(x_{i-1}) \geq 2^{-6} \Delta_i^{-1} n^{-\frac 2{2\beta+1}} f_0(x_{i-1})^{\frac{2\beta+3}{2\beta+1}}.$ Applying \eqref{eq.loc_comp_of_fcts_in_lb} again yields
\begin{align}
	\frac{1}{n \int_{x_{i-1}}^{x_i} f} \leq 2^6 n^{\frac{1-2\beta}{2\beta+1}} \Delta_i f_0(x_{i-1})^{-\frac{2\beta+3}{2\beta+1}}
	\leq 2^6 2^{\frac{2\beta+3}{2\beta+1}}n^{\frac{1-2\beta}{2\beta+1}} \int_{x_{i-1}}^{x_i}f_0(x)^{-\frac{2\beta+3}{2\beta+1}}dx
	\label{eq.heuristic_rate_computation}
\end{align}
and therefore $(I) \lesssim n^{\frac{1-2\beta}{2\beta+1}} \int_0^1f_0(x)^{-\frac{2\beta+3}{2\beta+1}}dx.$ In order to bound the term $(II)$ in \eqref{eq.ub_small_res}, we use Jensen's inequality, that $ab \leq a^2 + b^2$ for real numbers $a,b,$ and \eqref{eq.loc_comp_of_fcts_in_lb},
\begin{align*}
	(II)
	&\leq 16 \sum_{i=1}^m \frac1{\sqrt{\Delta_if_0(x_{i-1})}} \Big(\int_{x_{i-1}}^{x_i} \frac{(f(x)-f_0(x))^4}{ f_0(x_{i-1})^3} dx\Big)^{1/2}\\
	&\leq  2^7 n\int_0^1   \frac{(f(x)-f_0(x))^4}{ f_0(x)^3} dx  + 16\sum_{i=1}^m \frac1{n\Delta_if_0(x_{i-1})}.
\end{align*}
For the first term we use $f\in \Theta^\beta(f_0)$ and for the second term we can argue as for $(I)$ to obtain the upper bound
\begin{align}
	H^2\big(p_{<0}, q_{<0}\big) \leq (I)+(II) \lesssim n^{\frac{1-2\beta}{2\beta+1}} \int_0^1 f_0(x)^{-\frac{2\beta+3}{2\beta+1}} dx.
	\label{eq.Hell_bd_scaling_coeffs}
\end{align}
We next bound the Hellinger distance $H^2(p_{=J} ,q_{=J|<J}).$ For that we show that conditional on the observations at the lower resolution levels $(Z_{i,j,k})_{(i,j,k)\in \Lambda, -1\leq j<J},$ the random vector $(Z_{i,j,k})_{(i,j,k)\in \Lambda, j=J}$ has independent components. From the definition \eqref{eq.def_emp_wav_in_Poisson}, we conclude that conditioning on $(Z_{i,j,k})_{(i,j,k)\in \Lambda, -1\leq j<J}$ is the same as conditioning on $(U_{i,j,k})_{(i,j,k)\in \Lambda, j< J}$ and the counts $(N_{i,j,k})_{(i,j,k)\in \Lambda, j\leq J}.$  Since$$N_{i,J+1,2k}| (N_{i,j,k})_{(i,j,k)\in \Lambda, j\leq J} = N_{i,J+1,2k}|N_{i,J,k}, \quad k=0,\ldots, 2^J-1,$$are independent, $(Z_{i,j,k})_{(i,j,k)\in \Lambda, j=J}|(Z_{i,j,k})_{(i,j,k)\in \Lambda, -1\leq j<J}$ must also have independent components. This shows that $H^2(p_{=J} ,q_{=J|<J})\leq \sum_{k=0}^{2^J-1} H^2(Z_{i,J,k}^{*}, Z_{i,J,k}).$ Using moreover \eqref{eq.NiJ+1k_cond}, \eqref{eq.pijk_id} and \eqref{eq.Hell_bd_quantil_cpl}, we can bound $H^2(p_{=J} ,q_{=J|<J})$ by
\begin{align*}
		& \sum_{k=0}^{2^J-1} H^2(Z_{i,J,k}^{*}, Z_{i,J,k})  \\
		&\lesssim \sum_{k=0}^{2^J-1} \big(E[Z_{i,J,k}^{*}] - N_{i,J,k}^{1/2} (2p_{i,J,k}-1)\big)^2
		+\big(p_{i,J,k}-\tfrac 12\big)^2 + N_{i,J,k}\big(p_{i,J,k}-\tfrac 12\big)^4 \\
		&\lesssim \sum_{k=0}^{2^J-1} \big(E[Z_{i,J,k}^{*}] - [EN_{i,J,k}]^{1/2} (2p_{i,J,k}-1)\big)^2
		+\big(1 + \big(N_{i,J,k}^{1/2} - [EN_{i,J,k}]^{1/2} \big)^2 \big)\big(p_{i,J,k}-\tfrac 12\big)^2 \\
		&\quad\quad\quad\quad+ N_{i,J,k}\big(p_{i,J,k}-\tfrac 12\big)^4.
\end{align*}
With this inequality, we can now bound $\int \sqrt{p_{<J} q_{<J}}  H^2(p_{=J} ,q_{=J|<J}).$ By the Cauchy-Schwarz inequality, $\int \sqrt{p_{<J} q_{<J}} \leq 1,$ which yields a bound for the terms not depending on $N_{i,J,k}.$ For the terms depending on $N_{i,J,k}$ we use that $\int \sqrt{p_{<J}(x) q_{<J}(x)} h(x) dx \leq (\int h^2(x) q_{<J}(x) dx)^{1/2}$ for all integrable functions $h.$ By Lemma 3 in \cite{brown2004}, $E\big[\big(N_{i,J,k}^{1/2} - [EN_{i,J,k}]^{1/2} \big)^4\big]\leq 4$ and therefore,
\begin{align}
	\int \sqrt{p_{<J} q_{<J}}  H^2(p_{=J} ,q_{=J|<J}) 
	&\lesssim \sum_{i=1}^m\sum_{k=0}^{2^J-1} \big(EZ_{i,J,k}^{*} - [EN_{i,J,k}]^{1/2} (2p_{i,J,k}-1)\big)^2 \notag \\
	&\quad + \sum_{i=1}^m\sum_{k=0}^{2^J-1} (p_{i,J,k}-\tfrac 12)^2 +\sum_{i=1}^m\sum_{k=0}^{2^J-1} \sqrt{EN_{i,J,k}^2}\big(p_{i,J,k}-\tfrac 12\big)^4 \notag  \\
	& =: (i) +(ii) +(iii).
	\label{eq.high_freq_Hell_bd}
\end{align} 
We bound the three sums $(i)-(iii)$ separately. We will frequently use the fact that with $d_{i,J,k} = \int f(x) \psi_{i,J,k}(x) dx,$ \eqref{eq.pijk_id} can be rewritten as
\begin{align*}
	2p_{i,J,k} - 1  = \frac{\sqrt{\Delta_i}d_{i,J,k}}{2^{\frac J2} \int_{I_{i,J,k}} f(x) dx}.
\end{align*}

{\it (i):}  Observe that 
\begin{align*}
		\big(E[Z_{i,J,k}^{*}] - [EN_{i,J,k}]^{1/2} (2p_{i,J,k}-1)\big)^2
	&= n d_{i,J,k}^2\Big( \frac 1{\sqrt{f_0(x_{i-1})}} - \frac{\sqrt{\Delta_i} 2^{-J/2}}{\sqrt{\int_{I_{i,J,k}} f(x) dx}}\Big)^2.
\end{align*}
With $f\in \Theta^\beta(f_0)\subset \mH^\beta(R)$ for $\beta \leq 1,$ \eqref{eq.loc_comp_of_fcts_in_lb}, Jensen's inequality, $ab \leq a^2+b^2,$ and Lemma \ref{lem.wav_decay}, the right hand side of the last display can be bounded by
\begin{align*}
	& 2^6 n d_{i,J,k}^2 \frac{ \big(2^{J} \Delta_i^{-1}\int_{I_{i,J,k}} f(x) -f_0(x_{i-1}) dx \big)^2}{ f_0(x_{i-1})^3} \\
	&\leq 2^7 n d_{i,J,k}^2 \frac{ \big(2^{J} \Delta_i^{-1}\int_{I_{i,J,k}} f(x) -f_0(x) dx \big)^2+R^2 (2^{-J}\Delta_i)^{2\beta}}{ f_0(x_{i-1})^3} \\
	&\leq 2^7 n d_{i,J,k}^2 \frac{ \big(2^{J} \Delta_i^{-1}\int_{I_{i,J,k}} (f(x) -f_0(x))^4 dx \big)^{1/2}+R^2 (2^{-J}\Delta_i)^{2\beta}}{ f_0(x_{i-1})^3} \\
	&\leq 2^8R^4 n \frac{ 2^{-2J\beta}\Delta_i^{4\beta+1}}{f_0(x_{i-1})^3}
	+ 2^{10}2^{J-2J\beta} n \int_{I_{i,J,k}}\frac{(f(x)-f_0(x))^4}{f_0(x)^3} dx.
\end{align*}
Using that $f\in \Theta^\beta(f_0),$ \eqref{eq.Deltam_ineqs} and \eqref{eq.loc_comp_of_fcts_in_lb},
\begin{align*}
	(i) 
	&\lesssim 2^{J-2J\beta} n^{\frac{1-2\beta}{2\beta+1}} \int_0^1 f_0(x)^{-\frac{2\beta+3}{2\beta+1}} dx.
\end{align*}

{\it (ii):}  With $f\in \Theta^\beta(f_0),$  \eqref{eq.loc_comp_of_fcts_in_lb} Lemma \ref{lem.wav_decay}, and \eqref{eq.Deltam_ineqs},
\begin{align*}
	(2p_{i,J,k}-1)^2 \leq 2^{12} R^2 2^{-2J\beta} \frac{\Delta_i^{2\beta}}{f_0(x_{i-1})^2}
	\leq   2^{12} 3^{2\beta-1} R^2 2^{-2J\beta} n^{\frac{1-2\beta}{2\beta+1}}\Delta_i f_0(x_{i-1})^{-\frac{2\beta+3}{2\beta+1}} .
\end{align*}
Thus $(ii) \lesssim 2^{J-2J\beta} n^{\frac{1-2\beta}{2\beta+1}} \int_0^1 f_0(x)^{-\frac{2\beta+3}{2\beta+1}} dx.$ 

{\it (iii):} Since $N_{i,J,k}\sim \Poi(n\int_{I_{i,J,k}} f(u) du),$ we have $[EN_{i,J,k}^2]^{1/2} \leq 1+ n\int_{I_{i,J,k}} f(u) du.$ By definition $0 \leq p_{i,J,k}\leq 1$ and therefore  $(2p_{i,J,k}-1)^4\leq (2p_{i,J,k}-1)^2.$ Using \eqref{eq.loc_comp_of_fcts_in_lb} and the same bound as for $(ii),$
\begin{align*}
	[EN_{i,J,k}^2]^{1/2} \ (2p_{i,J,k}-1)^4 
	&\leq (2p_{i,J,k}-1)^2+2^{25} 3^{4\beta} R^4 2^{-(4\beta+1)J} n^{\frac{1-2\beta}{2\beta+1}} \Delta_i f_0(x_{i-1})^{-\frac{2\beta+3}{2\beta+1}} .
\end{align*}
Together with the bound for $(ii),$ this also shows that $(iii) \lesssim 2^{J-2J\beta} n^{\frac{1-2\beta}{2\beta+1}} \int_0^1 f_0(x)^{-\frac{2\beta+3}{2\beta+1}} dx.$

Combining the bounds for $(i)-(iii)$ gives for \eqref{eq.high_freq_Hell_bd}, $$\int \sqrt{p_{<J} q_{<J}}  H^2(p_{=J} ,q_{=J|<J}) \lesssim 2^{J-2J\beta} n^{\frac{1-2\beta}{2\beta+3}} \int_0^1 f_0(x)^{-\frac{2\beta+3}{2\beta+1}} dx.$$ Summing over $J$ and using that $\beta>1/2$ shows that with \eqref{eq.Hell_decomp} and \eqref{eq.Hell_bd_scaling_coeffs},
\begin{align*}
	H^2\big(p_{\overline J}, q_{\overline J}\big) \lesssim n^{\frac{1-2\beta}{2\beta+1}} \int_0^1 f_0(x)^{-\frac{2\beta+3}{2\beta+1}} dx,
\end{align*}
which proves the assertion.
\end{proof}

\subsection{Completion of the proof of Theorem \ref{thm.ub}}
\label{sec.compl_proof_ub}

From Propositions \ref{prop.mEP_mE1P}-\ref{prop.seq_space_Le_Cam_bd}, we deduce that under the assumptions of Theorem \ref{thm.ub}, 
\begin{align*}
	\sup_{f_0\in \Theta} \Delta \big( \mE_n^P(\Theta^\beta(f_0)), \mE_n^G(\Theta^\beta(f_0)) \big) 
	\lesssim n^{\frac{1-2\beta}{2\beta+1}} \sup_{f\in \Theta} \int_0^1 f(x)^{-\frac{2\beta+3}{2\beta+1}} dx.
\end{align*}
For the globalization step, the following result shows the existence of the required estimators satisfying the conditions of Lemma \ref{lem.localization_bd}.

\begin{thm}
\label{thm.glob2}
Under the assumptions of Theorem \ref{thm.ub}, there exists an estimator $\widehat f_n$ in $\mE_n^P(\Theta)$ taking values in a finite subset of $\Theta$ and satisfying
\begin{align*}
	\inf_{f_0\in \Theta} \overline P_{f_0}^n\big( f_0 \in \Theta^\beta(\widehat f_n) \big)  = 1 -O(n^{-1}).
\end{align*}
Moreover, there exists an estimator in $\mE_n^G(\Theta)$ with the same properties.
\end{thm}

Theorem \ref{thm.ub} then follows from Lemmas \ref{lem.localization_bd} and \ref{lem.sample_splitting}.

\section{Heuristics for the rates of the Le Cam deficiencies}
\label{sec.heuristics}

Most results on asymptotic equivalence require minimal smoothness assumptions, which are often difficult to explain heuristically. It is therefore unsurprising that the rates we obtain for the Le Cam deficiencies can also be difficult to interpret. Perhaps the best way to motivate these rates is to consider the doubly local decomposition of experiments explained in Section \ref{sec.Haar_wavelets}. Each doubly local experiment is similar to a parametric problem and the number $m_n$ of such experiments can be viewed as the effective dimension of the problem. For the following heuristic argument, one should think of the total variation distance as always being of the same order as the Hellinger distance, which is typically the case in our situation. If the double localization splits the model into (nearly) independent subproblems, then the overall squared Le Cam deficiency is simply the sum of the squared Le Cam deficiencies for each of the doubly local experiments. 

While we do not use a double localization for the Poissonization in the proof of Theorem \ref{thm.ub_DE_PIE}, it is still instructive to consider such an approach. For Poissonization, the Le Cam deficiencies for the doubly local experiments are all of the same order and the full Le Cam deficiency is therefore proportional to the effective dimension $m_n.$ For simplicity, we consider only the case $\inf_x f(x) \gg n^{-\beta/(\beta+1)}.$ Let $\Delta_j$ be as in \eqref{eq.Deltai_def}. By Lemma \ref{lem.fx_local_bd},
\begin{align*}
	m_n = \sum_{j=1}^{m_n} \Delta_j \Delta_j^{-1} \asymp n^{\frac 1{2\beta+1}} \int_0^1  f_0(x)^{-\frac 1{2\beta+1}} dx.
\end{align*}
If $\inf_x f(x) \gg n^{-\beta/(\beta+1)}$ then the squared rate in Theorem \ref{thm.Poissonization} can be written as $m_n/n,$ up to unnecessary $\log n$ terms. The squared rate can therefore also be written as $m_n r_n^2/n^2$ for $r_n = \sqrt{n}$, which motivates \eqref{eq.rate_Poissonization_general}. 

%\textit{A heuristic for the Gaussian approximation rate \eqref{eq.rate_in_intro}}: For simplicity, consider the deficiency $\delta (\mE_n^P(\Theta^\beta(f_0)),\mE_n^G(\Theta^\beta(f_0)))$ of the local experiments about $f_0$. As discussed in Section \ref{sec.Haar_wavelets}, we partition $[0,1]$ into the intervals $[x_{j-1},x_j]$, with these shrinking intervals generating appropriate doubly local subexperiments, say $\mE_{n,j}^P(\Theta^\beta(f_0))$ and $\mE_{n,j}^G(\Theta^\beta(f_0))$, $j=1,\dots,m_n$. By the independence structure of Poisson processes and the Gaussian process \eqref{eq.GWN_in_intro}, it suffices to compare each of these $m_n$ independent subexperiments and sum their contributions to the Le Cam distance. 

For the rate of the Le Cam deficiencies between Poisson intensity estimation and the Gaussian white noise model, recall that we partition $[0,1]$ into the intervals $[x_{j-1},x_j]$, with these shrinking intervals generating appropriate doubly local subexperiments. On each such independent subexperiment, we must couple a Poisson random variable with intensity parameter $\lambda_j = n\int_{x_{j-1}}^{x_j} f(u) du$ with a corresponding $\mathcal{N}(\lambda_j, \lambda_j)$ random variable. Since $\lambda_j \rightarrow \infty$ as $n\rightarrow \infty$, we may use Lemma \ref{lem.Hell_bd_scaling_coeffs} to couple a $\Poi (\lambda_j)$ random variable with a $\mathcal{N}(\lambda_j,\lambda_j)$ variable with a squared Hellinger error of size $1/(4\lambda_j) + o(1/\lambda_j)$. Using the independence structure of the subexperiments, these $m_n$ couplings yield a total squared Hellinger loss of order
\begin{align*}
	\sum_{j=1}^{m_n} \frac{1}{\lambda_j} \asymp 	1\wedge n^{\frac{1-2\beta}{2\beta+1}} \int_0^1 f_0(x) ^{-\frac{2\beta+3}{2\beta+1}} dx,
\end{align*}
where the $\lesssim $ direction follows from \eqref{eq.heuristic_rate_computation} and the $\gtrsim$ part can be similarly deduced. This motivates the rate \eqref{eq.rate_in_intro}.

%We thus see that the rate is driven by summing the errors from the $m_n$ independent and almost parametric subexperiments. This reinforces both the notion of $m_n$ as the effective dimension of the problem and the importance of sharply performing both localizations.

Of course, this decomposition into piecewise constant functions with $m_n$ pieces is too crude, and represents only the first resolution level of a much finer $L^2$-decomposition based on Haar wavelets, which is used to prove Theorem \ref{thm.ub} in Section \ref{sec.PIE_GWN}. However, it provides some insight into why the rate occurs already at low resolution levels without the full technical encumbrance of the higher order remainder terms, which are dealt with using quantile transformations in Section \ref{sec.LeCam_for_seq_space_models}.

\section{Lower bounds for Le Cam deficiencies in the regular regime}
\label{sec.lbs_regular}

In this section, we prove Theorems \ref{thm.lb} and \ref{thm.lb2}. The difference in the Bayes risk for an arbitrary prior and loss function bounded by one yields a lower bound for the Le Cam deficiency. Let  $\mE_1(\Theta)$ and $\mE_2(\Theta)$ be two experiments. If $ E_{\theta}^{(j)}[\ell(\widehat \theta_j, \theta)],$ $j=1,2,$ denotes the risk in experiment $\mE_j(\Theta)$ of the estimator $\widehat \theta_j$ with respect to the loss function $\ell,$ then
\begin{align*}
	\delta\big( \mE_1(\Theta),\mE_2(\Theta) \big)
	\geq
	&\inf_{\widehat \theta_1} \ \sup_{\widehat \theta_2}\ \sup_{\theta \in \Theta} \
	E_{\theta}^{(1)}[\ell(\widehat \theta_1, \theta)] - E_{\theta}^{(2)}[\ell(\widehat \theta_2, \theta)]
\end{align*}
provided the loss is bounded by one (see Definition 1 in \cite{LeCamY2000}, p.13). This immediately implies that for an arbitrary prior $\Pi$ on $\Theta,$
\begin{align}
	\delta\big( \mE_1(\Theta),\mE_2(\Theta) \big)
	\geq
	&\inf_{\widehat \theta_1} \ \sup_{\widehat \theta_2}\ \int_{\Theta}
	E_{\theta}^{(1)}[\ell(\widehat \theta_1, \theta)] d\Pi(\theta) -  \int_{\Theta}E_{\theta}^{(2)}[\ell(\widehat \theta_2, \theta)] d\Pi(\theta)
	\label{eq.deficiency_lb_by_diff_BR}
\end{align}
and the right hand side is just the difference of the Bayes risks (see also \cite{Torgersen1991}, Corollary 6.3.7). 

We argued in Section \ref{sec.Haar_wavelets} that the doubly local decomposition is intrinsic to this problem. For the lower bound, it is thus natural to again partition $[0,1]$ into the intervals $[x_{j-1},x_j].$ On each such doubly local experiment we construct a two hypothesis test, which are then combined into a global multiple testing problem. We compute the Bayes risk in both experiments, which, together with \eqref{eq.deficiency_lb_by_diff_BR}, provides a lower bound on the deficiencies.

\begin{proof}[Proof of Theorem \ref{thm.lb}] Throughout the proof, we write $a_n \lesssim b_n$ if $a_n \leq C b_n$ for all $n\geq n_0$ and a finite constant $C=C(\beta,R)$ which does not depend on $j$ and the parameter $\alpha$ defined below. In the same way we use $\gtrsim$ and the big-O notation.

Pick a sequence $(f_{0n})_n \subset \Theta_0$ such that $\int f_{0n}(x)^{-\frac{2\beta+3}{2\beta+1}} dx \geq \tfrac{1}{2} \sup_{f\in \Theta_0} \int f(x)^{-\frac{2\beta+3}{2\beta+1}} dx.$ For convenience we omit the dependence of $f_{0n}$ on $n,$ writing $f_0:=f_{0n}$ and $F_0 := \int_0^\cdot f_0(u) du.$

Set $\beta'=\beta \vee 2.$ Let $K:\mathbb{R}\rightarrow \mathbb{R}$ be a $\beta'$-smooth H\"older function with support on $[0,1]$ such that $\int_0^1 K(u) du=0,$ $\int K(u)^2 du =1$ and $\int K^3(u) du >0.$ Suppose additionally that $K'(u)=0$ for only finitely many $u \in [0,1]$. As an example of a kernel satisfying these conditions, consider the $L^2$-normalized version of $u\mapsto - \tfrac 43 h_{\beta'}(\tfrac 43 u) + 4 h_{\beta'}(4u-3),$ where $h_{\beta'}$ is the density of a Beta$(\beta'+1,\beta'+1)$ distribution. 

Let $(x_j)_{j=1,\ldots,m}$ be the sequence in \eqref{eq.xi_def} and define the functions
\begin{align}
	x \mapsto \psi_j(x) =  \frac{\alpha \gamma_j \Delta_j^\beta}{f_0(x_{j-1})}  K\Big(\frac{F_0(x)-F_0(x_{j-1})}{F_j}\Big), \quad j=1,\ldots,m,
	\label{eq.def_psi_lb}
\end{align}
where
\begin{align*}
	F_j:=F_0(x_j)-F_0(x_{j-1}), \quad     \gamma_j := \frac{f_0(x_{j-1})}{\sqrt{n \Delta_j^{2\beta} F_j}}
\end{align*}
and $0<\alpha \leq 1$ is a constant that will be chosen later to be small enough. The function $\psi_j$ has support $[x_{j-1}, x_j]$ and, since by assumption $\inf_{x_0} f_0(x_0) \gg n^{-\frac{\beta}{\beta+1}},$ we can apply \eqref{eq.loc_comp_of_fcts_in_lb} and \eqref{eq.Deltam_ineqs} to obtain 
\begin{align}
	\frac{1}{2}\Delta_j f_0(x_{j-1}) \leq F_j \leq 2\Delta_j f_0(x_{j-1})  \quad \text{and} \quad \frac 13 \leq \gamma_j^2 \leq 2 .
	\label{eq.Fj_gammaj_ineqs}
\end{align} 
Since $\inf_{x} f_0(x) \gg n^{-\frac{\beta}{\beta+1}},$ this also implies
\begin{align}
	\min_j nF_j \rightarrow \infty \ \ \text{and} \ \ \ 
	\max_{j=1,\ldots,m} \|\psi_j\|_{\infty} 
	\lesssim \max_{j=1,\ldots,m} \frac{\alpha \Delta_j^\beta}{f_0(x_{j-1})} 
	\lesssim  \max_{j=1,\ldots,m} \frac{\alpha }{\sqrt{nF_j}} 
	\ll \alpha.
	\label{eq.sup_norm_phi_j}
\end{align}
Define
\begin{align}
	\mu_{j,r} := \int \psi_j(x)^r f_0(x) dx
	\label{eq.mu_jr}
\end{align}
and observe that using the properties of $K$ as well as the definitions of $\Delta_j$ and $\gamma_j,$ $\mu_{j,1}=0,$ $\mu_{j,2}= \alpha^2n^{-1}$ and 
\begin{align}
	\mu_{j,3}
	= \frac{\alpha^3 \gamma_j^3 \Delta_j^{3\beta}}{f_0(x_{j-1})^3} F_j \int K^3(u) du
	\gtrsim \frac{\alpha^3}{n^{3/2}\sqrt{\Delta_j f_0(x_{j-1})}}.
	\label{eq.muj3_lb}
\end{align}
For higher moments, we frequently use the bound 
\begin{align}
	\mu_{j,r}\leq \|\psi_j\|_{\infty}^r F_j \lesssim \alpha^r F_j/(nF_j)^{r/2}.
	\label{eq.mujr_higher_moments}
\end{align}

We are now ready to define the test densities. For $\theta =(\theta_1, \ldots,\theta_m) \in \{-1, 1\}^m,$ consider
\begin{align*}
	x\mapsto f_\theta (x) = f_0 (x)
	\big( 1 +\sum_{j=1}^m \theta_j  \psi_j(x) \big).
%	\label{eq.ftheta_def}
\end{align*} 
From $\mu_{j,1}=0$ it follows that $\int f_\theta(x) dx =1$ and so $f_\theta$ are indeed probability densities. Observe also that $F_j = \int_{x_{j-1}}^{x_j} f_\theta (x) dx.$ With the sup-norm bound \eqref{eq.sup_norm_phi_j}, it follows immediately that for any $\theta  \in \{-1, 1\}^m,$ $f_\theta \in \mathcal{U}(f_0) \subset \Theta.$ By Lemma \ref{lem.test_fcts_in_Hoelder_ball}, we also know that $f_\theta \in \mathcal{H}^\beta(R)$ for all $\theta \in \{-1,1\}^m$ and $n$ large enough.

We now construct a prior on these densities. Renaming the parameters $f_\theta \leftrightarrow \theta,$ we can take $\{-1, 1\}^m$ as the parameter space and may also conveniently write $P_\theta^n = P_{f_\theta}^n$ and $Q_\theta^n = Q_{f_\theta}^n.$ We consider two priors called $\pi_+$ and $\pi_-,$ which are product priors on the parameter space $\{-1, 1\}^m,$ that is for each $\theta_0 =(\theta_1^0, \ldots,\theta_m^0) \in \{-1, 1\}^m,$
\begin{align*}
	\pi_\pm(\theta_0) = \prod_{j=1}^m \pi_\pm(\theta_j^0), \quad \text{with} \ \ \  \pi_\pm(\theta_j^0=1) = 1- \pi_\pm(\theta_j^0=-1) = e^{\pm 2\alpha}/(1+e^{\pm 2\alpha}).
\end{align*}
This prior is non-uniform. Indeed, $\pi_+$ assigns more weight to vectors which have more components being $+1$ than $-1.$ Both experiments behave very similarly under uniform priors and non-uniformity seems necessary here to obtain a rate-optimal separation of the experiments. The effect of $\alpha$ can best be seen in Proposition \ref{prop.pjqj_expansion} below. The priors $\pi_+$ and $\pi_-$ will lead to the lower bounds for the deficiencies  $\delta ( \mE_n^P(\Theta),\mE_n^G(\Theta))$ and $\delta ( \mE_n^G(\Theta),\mE_n^P(\Theta))$ respectively.
  
Next we construct the loss function. Observe that with \eqref{eq.Deltai_def}, \eqref{eq.loc_comp_of_fcts_in_lb}, \eqref{eq.Deltam_ineqs} and \eqref{eq.Fj_gammaj_ineqs},
\begin{align}
	\sum_{j=1}^m \frac 1{nF_j} = \sum_{j=1}^m \Delta_j \frac{1}{n \Delta_j F_j} \asymp n^{\frac{1-2\beta}{2\beta+1}}\int_0^1 f_0(x)^{-\frac{2\beta+3}{2\beta+1}} dx.
		\label{eq.nFj_int}
\end{align} 
Since $f_0$ is a density on $[0,1],$ $\{x : f_0(x) \geq 1\} \neq \varnothing.$ Let $[x_{j_{1n}},x_{j_{2n}}] \subset [0,1],$ $j_{1n}, j_{2n}\in \{1,\ldots,m\},$ be a sequence of intervals such that $[x_{j_{1n}},x_{j_{2n}}] \cap \{x : f_0(x) \geq 1\} \neq \varnothing$ for all $n$ and
\begin{align}
	n^{\frac{1-2\beta}{2\beta+1}}\int_{x_{j_{1n}}}^{x_{j_{2n}}} f_0(x)^{-\frac{2\beta+3}{2\beta+1}} dx \asymp 1 \wedge n^{\frac{1-2\beta}{2\beta+1}}\int_0^1 f_0(x)^{-\frac{2\beta+3}{2\beta+1}} dx.
		\label{eq.xj1xj2_def}
\end{align}
If the right-hand side is is smaller than one, set $[x_{j_{1n}},x_{j_{2n}}] = [x_0,x_m]=[0,1]$. If the right-hand side is exactly one, then arguing as in \eqref{eq.nFj_int} yields $n^{\frac{1-2\beta}{2\beta+1}}\int_{x_{j-1}}^{x_j} f_0(x)^{-\frac{2\beta+3}{2\beta+1}} dx \asymp 1/(nF_j)$ for all $j.$ By \eqref{eq.sup_norm_phi_j}, each interval $[x_{j-1},x_j]$ thus makes a vanishing contribution to the integral, which proves the existence of sequences satisfying \eqref{eq.xj1xj2_def}. Let
\begin{align}
	\rho(\theta, \theta') = \sum_{j=1}^m \rho_j \mathbf{1}(\theta_j \neq \theta_j') \quad \text{with}  \ \ \rho_j := \frac{1}{\sqrt{nF_j}} \mathbf{1}\big(j_{1n} < j\leq  j_{2n}\big)
	\label{eq.rhoj_def}
\end{align}
and for any $A>0,$ define the loss $\ell_A(\theta, \theta') = \mathbf{1}(\rho(\theta,\theta')\geq A).$ This loss is one if the weighted sum of the misclassified $\theta_j$'s exceeds the threshold $A$ and is zero otherwise. The reason for this particular weighting will become apparent later in the proof as a consequence of Proposition \ref{prop.pjqj_expansion} and Lemma \ref{lem.Bern_diff}. Arguing as for \eqref{eq.nFj_int},
\begin{align}
	\sum_{j=1}^m \rho_j^2 \asymp 1 \wedge n^{\frac{1-2\beta}{2\beta+1}}\int_0^1 f_0(x)^{-\frac{2\beta+3}{2\beta+1}} dx 
	\label{eq.rhojsq_int}
\end{align}
and similarly
\begin{align}
	\sum_{j=1}^m \rho_j^{3} =\sum_{j= j_{1n}+1}^{j_{2n}} \frac 1{(nF_j)^{3/2}}  \asymp n^{\frac{1-3\beta}{2\beta+1}}\int_{x_{j_{1n}}}^{x_{j_{2n}}} f_0(x)^{-\frac{3\beta+4}{2\beta+1}} dx .
	\label{eq.nFj_int32}
\end{align}
%which is the rate of the lower that we establish in this proof.

The key step is the following factorization of the likelihood ratio. In the Poisson experiment $\mE_n^P,$ define $N_j := \# \{ X_i: X_i \in (x_{j-1}, x_j]\}$ and write $X_1^{(j)}, \ldots, X_{N_j}^{(j)}$ for the observations in the interval $(x_{j-1}, x_j].$ Under $P_{\theta_0}^n,$ the counts $N_j$ are independent Poisson random variables with intensity parameters $ n\int_{x_{j-1}}^{x_j} f_{\theta_0}(x) dx = nF_j$ and the density of $X_i^{(j)}$ is $f_{\theta_0}(\cdot) \mathbf{1}(\cdot \in (x_{j-1},x_j])/ F_j.$ We can factorize
\begin{align}
	\frac{dP_\theta^n}{dP_{\theta_0}^n} 
	&= \prod_{i=1}^N  \frac{1 + \sum_{j=1}^m \theta_j \psi_j(X_i) }{1  +\sum_{j=1}^m \theta_j^0 \psi_j(X_i) } 
	= \prod_{j=1}^m \prod_{i=1}^{N_j}  \frac{1  + \theta_j \psi_j(X_i^{(j)}) }{1  + \theta_j^0 \psi_j(X_i^{(j)})} 
	 =: \prod_{j=1}^m P_j(\theta_j) 
	 \label{eq.LR_fact}
\end{align}
with $P_j(\theta_j)$ being independent random variables. Define the estimators $\widehat \theta_\pm^P=(\widehat \theta_{\pm,j}^P)_{j=1,\ldots,m}$ componentwise via $\widehat \theta_{\pm,j}^P \in \argmax_{\theta_j \in \{-1,1\}} P_j(\theta_j) \pi_{\pm}(\theta_j).$
Then $\widehat{\theta}_{\pm,j}^P \neq \theta_j^0$ iff $P_j(-\theta_j^0)\geq e^{\pm2\alpha \theta_j^0}.$ The random variables $\mathbf{1}(\widehat{\theta}_{\pm, j}^P \neq \theta_j^0)$ are therefore independent and Bernoulli distributed with success probabilities depending on the sign $\pm$ of the prior and $\theta_0,$
\begin{align*}
	p_{\pm, j}(\theta_0) := P_{\theta_0}^n(\widehat{\theta}_{\pm}^P \neq \theta_j^0)
	= P_{\theta_0}^n \big(P_j(-\theta_j^0)\geq e^{\pm2\alpha \theta_j^0} \big).
\end{align*}
We denote the Bernoulli distribution with parameter $p$ by $\Ber(p).$ For independent random variables $Z_j(a_j)\sim \Ber(a_j)$, the risk of $\theta_{\pm}^P$ under the loss function $\ell_A$ becomes
\begin{align}
	P_{\theta_0} \big(\rho(\widehat \theta_\pm^P, \theta_0) \geq A \big)
	= \P\big( \sum_{j=1}^m \rho_j Z_j\big(p_{\pm, j}(\theta_0)\big) > A\big).
	\label{eq.poisson_risk_lb_rewrite}
\end{align}

A similar factorization into independent products holds in the Gaussian white noise experiment since by Girsanov's formula,
\begin{align*}
	\frac{dQ_\theta^n}{dQ_{\theta_0}^n}
	&= \exp\Big( 2\sqrt{n} \int_0^1 (\sqrt{f_\theta(s)} - \sqrt{f_{\theta_0}(s)}) dW_s - 2n \big\| \sqrt{f_\theta}- \sqrt{f_{\theta_0}}\big\|_2^2\Big) \notag \\
	& = \prod_{j=1}^m \exp\Big( 2\sqrt{n} \int_{x_{j-1}}^{x_j} (\sqrt{f_\theta(s)} - \sqrt{f_{\theta_0}(s)}) dW_s - 2n \int_{x_{j-1}}^{x_j} \big(\sqrt{f_\theta(s)} - \sqrt{f_{\theta_0}(s)} \big)^2 ds \Big) \notag \\
	&=: \prod_{j=1}^m Q_j(\theta_j).
%	\label{eq.LR_GWN}
\end{align*}
In particular, $Q_j(\theta_j)$ are independent. In analogy with the Poisson model, define the estimators $\widehat \theta_\pm^G=(\widehat \theta_{\pm,j}^G)_{j=1,\ldots,m}$ componentwise via $\widehat \theta_{\pm,j}^G \in \argmax_{\theta_j \in \{-1,1\}} Q_j(\theta_j) \pi_{\pm}(\theta_j).$ Then $\widehat{\theta}_{\pm,j}^G \neq \theta_j^0$ iff $Q_j(-\theta_j^0)\geq e^{\pm2\alpha \theta_j^0}.$ With $ q_{\pm,j} (\theta_0):= Q_{\theta_0}^n(\widehat{\theta}_{\pm,j}^G \neq \theta_j^0) = Q_{\theta_0}^n(Q_j(-\theta_j^0)\geq e^{\pm2\alpha \theta_j^0})$ we find in the same way as for \eqref{eq.poisson_risk_lb_rewrite} that for independent $Z_j(q_{\pm, j} (\theta_0)) \sim \Ber(q_{\pm, j} (\theta_0))$, 
\begin{align*}
		Q_{\theta_0}^n \big(\rho(\widehat \theta_\pm^G, \theta_0) \geq A \big)
	= \P\big( \sum_{j=1}^m \rho_j Z_j\big(q_{\pm, j} (\theta_0)\big) > A\big).
\end{align*}

\begin{prop}
\label{prop.pjqj_expansion}
Let $\Phi$ be the c.d.f. of the standard normal distribution, $\phi=\Phi'$ be its density and $\mu_{j,r}$ be defined by \eqref{eq.mu_jr}. Then for sufficiently large $n,$ there exists a constant $C$ independent of $\alpha,n,j,$ such that 
\begin{align*}
	\big| q_{\pm, j}(\theta_0)  -\Phi (-\alpha \mp \theta_j^0) \big| \leq \frac{C\alpha^2}{nF_j}, \  \text{and} \ 
	\big| p_{\pm, j}(\theta_0)  -\Phi (-\alpha \mp \theta_j^0) \mp \frac{n\mu_{j,3}}{6\alpha^2}\phi(-\alpha \mp \theta_j^0) \big| \leq \frac{C\alpha^2}{\sqrt{nF_j}}.
\end{align*} 
\end{prop}

With \eqref{eq.muj3_lb}, we conclude that for sufficiently small $\alpha>0,$ the success probabilities differ by a term of order at least $\alpha/\sqrt{n F_j}.$ This is the key ingredient to show that there is a difference in the Bayes risks for the two experiments. Recall that $\alpha$ is the parameter modeling the non-uniformity of the prior and the size of the local alternatives $\psi_j.$ If the prior is uniform then $\alpha=0$, and a close inspection of the proof shows that the difference in the success probabilities is then of the smaller order $1/(nF_j)$, so that non-uniformity of the prior is crucial in this construction.

The following proposition shows that $\widehat \theta_{\pm}^P$ and $\widehat \theta_{\pm}^G$ are Bayes estimators and uses that the deficiency is lower bounded by the difference of the Bayes risks. A proof can be found in Appendix \ref{sec.add_proofs}.
\begin{prop}
\label{prop.MLE_is_mininmax}
Let $\Theta, \rho$ and $\pi_{\pm}$ be as above. Then $\widehat \theta_{\pm}^P$ and $\widehat \theta_{\pm}^G$ are Bayes estimators with respect to the priors $\pi_{\pm}$ in the Poisson intensity estimation and the Gaussian white noise experiments, respectively. 
\end{prop}

Together with \eqref{eq.deficiency_lb_by_diff_BR}, the previous proposition thus shows that for any $A>0,$
\begin{align}
	\delta \big( \mE_n^P(\Theta),\mE_n^G(\Theta) \big)
	&\geq 
	\sum_{\theta_0\in \Theta} 
	\Big(P_{\theta_0}^n \big(\rho(\widehat \theta_+^P, \theta_0) \geq A \big)  - Q_{\theta_0}^n \big(\rho(\widehat \theta_+^G, \theta_0) \geq A \big)\Big)
	\pi_+(\theta_0)  \notag \\
	&=
	\sum_{\theta_0\in \Theta} 
	\Big(  \P\big( \sum_{j=1}^m \rho_j Z_j\big(p_{+,j}(\theta_0)\big) > A\big) -\P\big( \sum_{j=1}^m \rho_j Z_j\big(q_{+,j} (\theta_0)\big) > A\big)\Big)
	\pi_+(\theta_0) 
	\label{eq.first_Bayes_difference}
\end{align}
and
\begin{align*}
	\delta \big( \mE_n^G(\Theta),\mE_n^P(\Theta) \big)
	&\geq 
	\sum_{\theta_0\in \Theta} 
	\Big(  \P\big( \sum_{j=1}^m \rho_j Z_j\big(q_{-,j}(\theta_0)\big) > A\big) -\P\big( \sum_{j=1}^m \rho_j Z_j\big(q_{-,j} (\theta_0)\big) > A\big)\Big)
	\pi_-(\theta_0). 
\end{align*}

We have therefore reduced lower bounding the Le Cam deficiency to computing probabilities connected to weighted sums of independent Bernoulli random variables. To finish the proof we need the following monotonicity property together with a change of measure type inequality which are established next and proved separately in Appendix \ref{sec.add_proofs}.

\begin{rem}
\label{rem.monotonicity_of_Bern_sum}
The probability $\P\big(\sum_{j=1}^m \rho_j Z_j(a_j)>A \big)$ is monotone increasing in the parameters $a_j.$ Indeed if $a_j'\geq a_j,$ then for $\eta \sim \Ber(a_j/a_j')$ independent, $Z_j(a_j') \geq \eta Z_j(a_j') \sim \Ber(a_j).$
\end{rem}

\begin{lem}
\label{lem.Bern_diff}
Suppose that $(p_j)_{j=1,\ldots,m},$ $(q_j)_{j=1,\ldots,m}$ and $(\beta_j)_{j=1,\ldots,m}$ are vectors with entries between zero and one such that for some $0\leq \omega \leq 1/2,$ $p_j \geq  q_j + q_j(1-q_j) \omega \beta_j$ for all $j=1,\ldots,m.$ If $(Z_j(p_j))_{j=1,\ldots,m}$ are independent $\Ber(p_j)$ random variables, then
\begin{align*}
	\P\big(\sum_{j=1}^m \beta_j Z_j(p_j)>A \big)
	\geq  \exp\big ( \omega A -\omega \sum_{j=1}^m \beta_j q_j - 2\omega^2 \sum_{j=1}^m \beta_j^2 \big) \P\big(\sum_{j=1}^m \beta_j Z_j(q_j)>A \big).
\end{align*}
\end{lem}

Recall that the difference of the success probabilities in Proposition \ref{prop.pjqj_expansion} is of the order at least $\alpha /\sqrt{nF_j}.$ Together with the change of measure formula in Lemma \ref{lem.Bern_diff}, this shows why the weights $\rho_j=1/\sqrt{nF_j}$ in the Hamming loss \eqref{eq.rhoj_def} are natural.
Let us only consider the case where $\theta_0$ is drawn from $\pi_+,$ that is the case \eqref{eq.first_Bayes_difference}. The other case can be proved analogously. By Proposition \ref{prop.pjqj_expansion}, $q_j :=q_{+,j}(\theta_0) = \Phi(-\alpha - \theta_j^0) + O(\alpha^2/(nF_j))$ and $p_j := p_{+,j}(\theta_0) = \Phi(-\alpha - \theta_j^0) + (n\mu_{j,3}/(6\alpha^2))\phi(-\alpha -\theta_j^0) +O(\alpha^2/\sqrt{nF_j}).$ Choosing the constant $\alpha$ small enough, $\Phi(-2) \leq q_j \leq \Phi(1)$ and moreover by \eqref{eq.muj3_lb} we can always find a positive constant $c>0$ such that $p_j \geq q_j +  c q_j(1-q_j)\alpha \rho_j,$ for all $j=1,\ldots,m.$ Denote the mean of $q_j=q_{+,j}(\theta_0)$ under $\pi_+$ by $\overline q_j,$ let $r_\alpha = \mathbb{E}_{\theta_j \sim \pi_+}[\Phi(-\alpha-\theta_j)]=  \Phi(-\alpha - 1) \pi_+(\theta_j=1)+ \Phi(-\alpha + 1) \pi_+(\theta_j=-1)$  and choose the constant in the loss $\ell_A$ as
\begin{align*}
	A= r_\alpha \sum_{j=1}^m \rho_j 
	+ 4\Big( \sum_{j=1}^m \rho_j^2\Big)^{1/2}.
\end{align*}
Throughout the remaining proof we make frequent use of the formula $\sum_{j=1}^m \rho_j^2 \lesssim 1,$ which follows immediately from \eqref{eq.rhojsq_int}. In particular, this allows us to conclude from $|\sum_{j=1}^m \rho_j(\overline q_j-r_\alpha)|\lesssim \alpha \sum_{j=1}^m \rho_j^2$ that for sufficiently small $\alpha$ and $n$ large enough, $|\sum_{j=1}^m \rho_j(\overline q_j-r_\alpha)| \leq (\sum_{j=1}^m \rho_j^2)^{1/2}.$ Define the set  
\begin{align*}
	\mathcal{D} := \Big\{\theta_0 \in \Theta : 
	 \big| \sum_{j=1}^m \rho_j \overline q_j - \sum_{j=1}^m \rho_j q_j \big|
	\leq \big(\sum_{j=1}^m \rho_j^2\big)^{1/2} \Big\}.
\end{align*}
Since $p_j \geq q_j,$ all the summands in \eqref{eq.first_Bayes_difference} are non-negative in view of Remark \ref{rem.monotonicity_of_Bern_sum}. By Lemma \ref{lem.Bern_diff} with $\omega =c\alpha,$ the definition of $A$ and $e^x\geq x+1,$ it follows that for sufficiently small $\alpha$ and $n$ large enough,
\begin{align}
	\delta \big( \mE_n^P(\Theta),\mE_n^G(\Theta) \big)
	\geq 
	&\sum_{\theta_0 \in \mathcal{D}} 
	 \Big[ \exp\Big( c\alpha A - c\alpha \sum_{j=1}^m \rho_j q_j - c^2\alpha^2\sum_{j=1}^m \rho_j^2  \Big) -1\Big]
	\P\big( \sum_{j=1}^m \rho_j Z_j(q_j) > A\big) \pi_+(\theta_0) \notag  \\
	\geq &
	c\alpha \big( \sum_{j=1}^m \rho_j^2 \big)^{1/2}
	\sum_{\theta_0 \in \mathcal{D}}
	\P\big( \sum_{j=1}^m \rho_j Z_j(q_j) > A\big) \pi_+(\theta_0).
	\label{eq.delta_lb}
\end{align}
Recall that the expectation and the variance of $ \sum_{j=1}^m \rho_j Z_j(q_j)$ are $ \sum_{j=1}^m \rho_j q_j$ and $ \sum_{j=1}^m \rho_j^2 q_j (1-q_j)$ respectively. Let $\xi$ be a Gaussian random variable with the same mean and variance. By Berry-Esseen's theorem there exists a universal constant $C_0$ such that for $\theta_0\in \mathcal{D},$
\begin{align*}
	\P\big( \sum_{j=1}^m \rho_j Z_j(q_j) > A\big)
	&\geq \P\big( \xi >A \big) 
	- C_0 \frac{ \sum_{j=1}^m \rho_j^3}{(\sum_{j=1}^m \rho_j^2)^{3/2}} \\
	&\geq 
	1- \Phi\Big(\frac{6}{\sqrt{\Phi(-2)(1-\Phi(1))}}\Big)- C_0 \frac{ \sum_{j=1}^m \rho_j^3}{(\sum_{j=1}^m \rho_j^2)^{3/2}},
\end{align*}
where we used that $q_j(1-q_j) \geq \Phi(-2)(1-\Phi(1)).$ From \eqref{eq.rhojsq_int}, \eqref{eq.nFj_int32} and Lemma \ref{lem.integral_comp}, it follows that $\sum_{j=1}^m \rho_j^3 \ll (\sum_{j=1}^m \rho_j^2)^{3/2}.$ For all sufficiently large $n$,
\begin{align*}
	\inf_{\theta_0 \in \mD}\P\big( \sum_{j=1}^m \rho_j Z_j(q_j) > A\big)
	\geq
	\frac 12\Big(1- \Phi\Big(\frac{6}{\sqrt{\Phi(-2)(1-\Phi(1))}}\Big)\Big)
\end{align*}
and the right-hand side is positive. Denote by $\Var_{\pi_+}$ the variance with respect to the prior $\pi_+.$ Since $0\leq q_j \leq \Phi(1),$ Chebychev's  inequality yields
\begin{align*}
	\pi_+(\mathcal{D})
	= 1- \pi_+(\Theta\setminus \mathcal{D})
	\geq 1 - \frac{\Var_{\pi_+}(\sum_{j=1}^m \rho_j q_j)}{\sum_{j=1}^m \rho_j^2}
	\geq 1 - \Phi(1)^2 >0.
\end{align*}
Together with \eqref{eq.delta_lb}, this shows that $\delta \big( \mE_n^P(\Theta),\mE_n^G(\Theta) \big) \geq $ const.$\times ( \sum_{j=1}^m \rho_j^2 )^{1/2}$ and with \eqref{eq.rhojsq_int} this completes the proof for the lower bound of $\delta \big( \mE_n^P(\Theta),\mE_n^G(\Theta) \big).$ A similar argument holds for the deficiency $\delta \big( \mE_n^G(\Theta),\mE_n^P(\Theta) \big),$ replacing the prior $\pi_+$ by $\pi_-.$
\end{proof}

\begin{proof}[Proof of Theorem \ref{thm.lb2}]
Recall that by assumption, $\inf_{f\in \Theta} \inf_x f(x) \gg n^{-\beta/(\beta+1)}.$ Since $f\in \Theta  \subset \mH^\beta(R),$ $f$ is also uniformly bounded and with Theorem \ref{thm.ub_DE_PIE},
\begin{align*}
	\Delta(\mE_n^D(\Theta), \mE_n^P(\Theta))^2 
	&\lesssim n^{-\frac{2\beta}{2\beta+1}}\log^2 n \  \sup_{f\in \Theta}  \int_0^1 \Big( \frac 1{f(x)} \wedge n^{\frac{\beta}{\beta+1}}\Big)^{\frac{1}{2\beta+1}}  dx \\
	& \ll 
	n^{\frac{1-2\beta}{2\beta+1}} \sup_{f\in \Theta}  \int_0^1 f(x)^{-\frac{2\beta+3}{2\beta+1}} dx.
\end{align*}
Using \eqref{eq.1st_assump_thmlb2} and that the Le Cam deficiency satisfies the triangle inequality, Theorem \ref{thm.lb} implies $$\delta(\mE_n^D(\Theta), \mE_n^G(\Theta)) \geq \delta(\mE_n^P(\Theta), \mE_n^G(\Theta)) - \Delta(\mE_n^D(\Theta), \mE_n^P(\Theta))\gtrsim \Big(n^{\frac{1-2\beta}{2\beta+1}}\sup_{f\in \Theta} \int f(x)^{-\frac{2\beta+3}{2\beta+1}}dx\Big)^{1/2}.$$
Similarly, we can obtain the same lower bound for the deficiency $\delta(\mE_n^G(\Theta), \mE_n^D(\Theta))$ and this completes the proof. 
\end{proof}

\textbf{Acknowledgements:} The authors would like to thank the Associate Editor and Referees for their valuable suggestions and comments. The second author would like to thank Marc Hoffmann for bringing this problem to his attention during his PhD defense.

\appendix

\section*{Appendix}

\section{Proofs for Section \ref{sec.DE_PIE}}
\label{eq.proof_sec_DE_PIE}

\begin{proof}
[Proof of Theorem \ref{thm.Poissonization}] We first construct a Markov kernel that maps density estimation to the Poisson intensity model up to an error
\begin{align}
	\delta\big(\mE_n^D(\Theta_1^\beta(f_0)), \mE_n^P(\Theta_1^\beta(f_0))\big)\lesssim n^{-\frac{2\beta}{2\beta+1}}\log^2 n \ \int_0^1 \Big( \frac 1{f_0(x)} \wedge n^{\frac{\beta}{\beta+1}}\Big)^{\frac{1}{2\beta+1}}  dx.
	\label{eq.to_show_Poissonization}
\end{align}
Throughout the proof, we always consider the parameter space $\Theta_1^\beta(f_0)$ and thus omit it in the notation, that is we write $\mE_n^D:=\mE_n^D(\Theta_1^\beta(f_0)), \mE_n^P:=\mE_n^P(\Theta_1^\beta(f_0)), \ldots$ For $\kappa_n := \sqrt{2n\log n},$ let $N \sim \Poi(n-\kappa_n)$ and define a new experiment $\mG_{n-\kappa_n}^P$ in which we observe $N\wedge n$ i.i.d. random variables $X_1, \ldots, X_{N\wedge n}$ with density $f.$ The Le Cam deficiency satisfies the triangle inequality and so
\begin{align*}
	\delta\big(\mE_n^D, \mE_n^P\big) 
	\leq \delta\big(\mE_n^D, \mG_{n-\kappa_n}^P\big) + \delta (\mG_{n-\kappa_n}^P, \mE_{n-\kappa_n}^P\big)+\delta\big(\mE_{n-\kappa_n}^P, \mE_n^P\big) =(I) +(II)+(III).
%	\label{eq.delta_mED_mEP_decomp}
\end{align*}

{\it (I):} Since $\mG_{n-\kappa_n}^P$ is not more informative than $\mE_n^D,$ $\delta\big(\mE_n^D, \mG_{n-\kappa_n}^P\big)=0.$ 

{\it(II):} Denote by $P_f^{X,N}$ the distribution of $(X_1, \ldots, X_{N\wedge n}, N)$ in experiment $\mG_{n-\kappa_n}^P.$ Similarly, write $Q_f^{X,N}$ and $Q_f^{X|N_n}$ for the distributions of $(X_1, \ldots, X_N, N)$ and $(X_1, \ldots,X_N)|N$ in experiment $\mE_{n-\kappa_n}^P.$ If $N\leq n,$ both experiments are equally informative. If $M$ denotes the Markov kernel adding $(N-n)\vee 0$ times the first observation, 
\begin{align*}
	\widetilde P_f^{X,N}=MP_f^{X,N} = (X_1, \ldots, X_{N\wedge n}, \underbrace{ X_1, \ldots,X_1}_{(N-n)\vee 0}, N).
\end{align*}
Writing $\widetilde P_f^{X|N}$ for the conditional distribution given $N,$
\begin{align*}
	\delta (\mG_{n-\kappa_n}^P, \mE_{n-\kappa_n}^P\big) 
	= \inf_M \sup_f \|M P_f^X - Q_f^X \|_{\TV}
	\leq \sup_f \E\big[ \|\widetilde P_f^{X|N} - Q_f^{X|N} \|_{\TV} | N \big]
	\leq \P(N>n).
\end{align*}
With Lemma \ref{lem.Poisson_moments}(iii), we can further bound the right-hand side by $4/n.$

{\it (III):} Let $L_n :=n^{-1}\log n$ and $c:=(4C) \vee (4C)^{(2\beta+1)/(\beta+1)}$ with $C$ the constant in the definition of $\Theta_1^\beta(f_0)$. Recall that $N\sim \Poi(n-\kappa_n).$ In experiment $\mE_{m}^P$ we observe a Poisson process on $[0,1]$ with intensity $mf.$ Adding an independent Poisson process with intensity $\kappa_n \widetilde f_0,$ where $\widetilde f_0=f_0 \mathbf{1}(f_0(x)\geq c L_n^{\beta/(\beta+1)}),$ we observe in experiment $\mE_{n-\kappa_n}^P$ a Poisson process with intensity $(n-\kappa_n)f+ \kappa_n \widetilde f_0.$ Due to the choice of the constant $c,$ we have $|f(x)-f_0(x)|\leq CL_n^{\beta/(\beta+1)}+C(L_nf_0(x))^{\beta/(2\beta+1)}\leq \tfrac 14 f_0(x) + \tfrac 14 f_0(x) \leq \tfrac 12 f_0(x)$ whenever $f_0(x)\geq c L_n^{\beta/(\beta+1)}$ and $f \in \Theta_1^\beta(f_0).$ This implies in particular that under these conditions $f(x)\geq \tfrac 12 f_0(x).$ Using the Hellinger bound for two Poisson processes in Lemma  \ref{lem.bds_of_info_distances}(i), uniformly over $f\in \Theta_1^\beta(f_0),$
\begin{align*}
	\delta\big(\mE_{n-\kappa_n}^P, \mE_n^P\big)^2
	&\leq \int (\sqrt{(n-\kappa_n)f(x)+ \kappa_n \widetilde f_0(x)}- \sqrt{n f(x)})^2 dx \\
	&\leq  \frac{\kappa_n^2}{n} \int f(x) \mathbf{1}\big(f_0(x)< cL_n^{\frac{\beta}{\beta+1}} \big) + \frac{(f_0(x)-f(x))^2}{f(x)}\mathbf{1}\big(f_0(x)\geq cL_n^{\frac{\beta}{\beta+1}} \big) dx \\
	&\lesssim \log n \int_0^1 L_n^{\frac{\beta}{\beta+1}}\mathbf{1}\big(f_0(x)< cL_n^{\frac{\beta}{\beta+1}} \big)
	+ \frac{(L_nf_0(x))^{\frac{2\beta}{2\beta+1}}}{f_0(x)} \mathbf{1}\big(f_0(x)\geq cL_n^{\frac{\beta}{\beta+1}} \big) dx \\
	&\lesssim \log n \int_0^1 L_n^{\frac{\beta}{\beta+1}} \wedge L_n^{\frac{2\beta}{2\beta+1}} f_0(x) ^{-\frac{1}{2\beta+1}} dx \\
	&\leq n^{-\frac{2\beta}{2\beta+1}}\log^2 n \int_0^1 \Big(\frac{1}{f_0(x)} \wedge n^{\frac{\beta}{\beta+1}}\Big)^{\frac{1}{2\beta+1}} dx.
\end{align*}
The upper bounds derived in $(I)-(III)$ imply \eqref{eq.to_show_Poissonization}. Estimating $\delta\big( \mE_n^P, \mE_n^D\big)$ from above can be done using the same arguments and leads to exactly the same rate in the upper bound. Since $\Delta( \mE_n^D, \mE_n^P\big)= \delta\big( \mE_n^D, \mE_n^P\big) \vee \delta\big( \mE_n^P, \mE_n^D\big),$ the proof is complete.
\end{proof}

\section{Additional proofs for Theorem \ref{thm.lb}}
\label{sec.add_proofs}

In this section, we provide proofs for the propositions occurring in the proof of Theorem \ref{thm.lb}.

\begin{lem}
\label{lem.test_fcts_in_Hoelder_ball}
Suppose that $f_0 \in \mathcal{H}^\beta(R')$ and let $f_\theta=  f_0+ f_0\sum_{j=1}^m \theta_j \psi_j$ with $\psi_j$ as defined in \eqref{eq.def_psi_lb}. Assume that $\inf_x f_0(x) \gg n^{-\frac{\beta}{\beta+1}}.$ For any $R>R',$ there exist $\alpha_0>0$ and $n_0$ such that for any $n\geq n_0,$ whenever $\alpha$ in the definition of $\psi_j$ in \eqref{eq.def_psi_lb} is smaller than  $\alpha_0,$  $$f_\theta \in \mathcal{H}^\beta(R), \quad \text{for all} \ \theta  \in \{-1, 1\}^m.$$
\end{lem}

\begin{proof}%[Proof of Lemma \ref{lem.test_fcts_in_Hoelder_ball}]
The $\lesssim$ symbol is used as in Theorem \ref{thm.lb}. Throughout the proof all statements are considered to hold for sufficiently large $n.$

Let $\delta>0$ be arbitrary. In $(i)$ we check that for sufficiently large $n,$ $\|f_\theta\|_{\infty}+|f_\theta |_{\mC^\beta}\leq \|f_0\|_{\infty}+|f_0|_{\mC^\beta}+2\delta$ and in $(ii)$ we verify that for sufficiently large $n,$ $|f_\theta|_{\mH^\beta}\leq |f_0|_{\mH^\beta} +\delta$ and $\|f_\theta^{(\lfloor \beta \rfloor)}\|_\infty \leq \|f_0^{(\lfloor \beta \rfloor)}\|_\infty+ \delta.$ Putting all the bounds together, we find that for sufficiently large $n,$ $(i)$ and $(ii)$ imply $\|f_\theta\|_{\mH^\beta}\leq \|f_0\|_{\mH^\beta}+ 4\delta.$ Since $\delta>0$ was arbitrary, this then gives the result. 

Throughout the proof of $(i)$ and $(ii),$ we use freely the inequalities \eqref{eq.Fj_gammaj_ineqs} and $\max_{j=1, \ldots,m} \Delta_j^\beta/f_0(x_{j-1}) \rightarrow 0,$ which is a consequence of $\inf_{x} f_0(x) \gg n^{-\beta/(\beta+1)}.$

{\it (i):} Recall that $\|f\|_{\mC^\beta}= \|f\|_\infty+\|f^{(\lfloor \beta \rfloor)}\|_\infty+|f|_{\mC^\beta}.$ Since $\|f_\theta\|_{\mC^\beta}\leq \|f_0\|_{\mC^\beta}+ \| f_0\sum_{j=1}^m \theta_j \psi_j\|_{\mC^\beta},$ it remains to show that $ \| f_0\sum_{j=1}^m \theta_j \psi_j\|_{\mC^\beta}\leq 3\delta.$ By \eqref{eq.sup_norm_phi_j} and due to the disjoint support of $\psi_{j}$ for different $j,$ $\|f_0\sum_{j=1}^m\theta_j \psi_j \|_{\infty} \leq \delta.$ In the next step we show that $|f_0\sum_{j=1}^m\theta_j  \psi_j|_{\mC^\beta}\leq \delta.$ By definition, the derivatives of the kernel function $K$ in the definition of $\psi_j$ in \eqref{eq.def_psi_lb} vanish on the boundary points $u\in \{0,1\}$ and so $(f_0\sum_{j=1}^m \theta_j \psi_j)^{(\lfloor \beta \rfloor )}(x)=0,$ whenever $x=x_j$ with $j=0,1,\ldots,m.$ Thus, if $x\in [x_{j-1},x_j]$ and $y \in [x_{j'-1},x_{j'}]$ with $j<j',$
\begin{align*}
	&\big | (f_0\sum_{j=1}^m \theta_j \psi_j)^{(\lfloor \beta \rfloor )}(x)
	- (f_0\sum_{j=1}^m \theta_j \psi_j)^{(\lfloor \beta \rfloor )}(y)\big| \\
	&\leq \big | (f_0 \psi_j)^{(\lfloor \beta \rfloor )}(x) - (f_0 \psi_j)^{(\lfloor \beta \rfloor )}(x_j)\big|
	+\big | (f_0 \psi_{j'})^{(\lfloor \beta \rfloor )}(x_{j'-1}) - (f_0 \psi_{j'})^{(\lfloor \beta \rfloor )}(y)\big|.
\end{align*}
Together with the inequality $x^\gamma + y^\gamma \leq 2^{1-\gamma}(x + y)^\gamma$ for $0<\gamma \leq 1,$ which is a consequence of the concavity of $x\mapsto x^\gamma,$ $0<\gamma \leq 1,$ it follows that if the H\"older seminorm on each interval $[x_{j-1}, x_j]$ is bounded by $\delta/2,$ then the global H\"older seminorm is less than $\delta.$ It is thus enough to show $|f_0\psi_j|_{\mC^\beta}\leq \delta/2.$

For $\beta\leq 1,$ with \eqref{eq.loc_comp_of_fcts_in_lb} and \eqref{eq.sup_norm_phi_j}, $|f_0 \psi_j|_{\mC^\beta}\leq 2 f_0(x_{j-1})|\psi_j|_{\mC^\beta}+ |f_0 |_{\mC^\beta} \|\psi_j\|_\infty \lesssim \alpha \leq \alpha_0.$ Choosing $\alpha_0$ small gives $|f_0 \psi_j|_{\mC^\beta} \leq \delta/2.$ Now suppose $\beta>1.$ The proof that $|f_0 \psi_j|_{\mC^\beta} \leq \delta/2$ follows along the lines of the proof of Lemma 2 in \cite{RaySchmidt-Hieber2015c}. For the convenience of the reader, we nevertheless give the full proof here and only refer to \cite{RaySchmidt-Hieber2015c} for a more detailed exposition. With $v_j(x) := (F_0(x)-F_0(x_{j-1}))/F_j,$ we can rewrite $f_0(x) \psi_j(x) = \alpha \gamma_j \Delta_j^{\beta} f_0(x_{j-1})^{-1}f_0(x) (K \circ v_j)(x).$ For two $r$-times differentiable functions $g, h,$ $(gh)^{(r)} = \sum_{q=0}^r \binom{r}{q}g^{(q)}h^{(r-q)}.$ Moreover, by Fa\`a di Bruno's formula, we have for the $q$-th derivative of $K\circ v_j,$ 
\begin{align*}
	\big( K\circ v_j \big)^{(q)} = \sum c_{m_1, \ldots, m_q}  (K^{(M_q)}\circ v_j) \prod_{s=1}^q \big(v_j^{(s)}\big)^{m_s} 
	= \sum c_{m_1, \ldots, m_q}   \frac{K^{(M_q)}\circ v_j}{F_j^{M_q}} \prod_{s=1}^q \big(f_0^{(s-1)}\big)^{m_s},
\end{align*}
where the sum is over all non-negative integers $m_1,\ldots,m_q$ with $m_1+2m_2+\ldots+q m_q =q,$ $M_q:=\sum_{\ell=1}^q m_\ell,$ and $c_{m_1, \ldots, m_q}$ are suitable coefficients. The $r$-th derivative of $f_0\psi_j$ can thus be rewritten as
\begin{align}
	\frac{\alpha \gamma_j \Delta_j^\beta}{f_0(x_{j-1})} \Big((K\circ v_j) f_0^{(r)}+ \sum_{q=1}^r \sum \binom{r}{q}  c_{m_1, \ldots, m_q} \frac{K^{(M_q)}\circ v_j}{F_j^{M_q}} f_0^{(r-q)} \prod_{s=1}^q \big(f_0^{(s-1)}\big)^{m_s} \Big),
	\label{eq.f0phij_expansion}
\end{align}
where the second sum is over the same set of integers as above.

If $x, y \in [x_{j-1}, x_j],$ then by \eqref{eq.Fj_gammaj_ineqs}, $|K^{(q)}\big(v_j(x)\big)-K^{(q)}\big(v_j(y)\big) |\lesssim (\Delta_j^{-1}|x-y|)^{\beta -r}$ for any $q=0, \ldots, r.$ By definition, $f_0\in \mathcal{H}^\beta(R')$ implies that $|f_0^{(r)}(x)|\leq R^{\frac{r}{\beta}} |f_0(x)|^{\frac{\beta-r}{\beta}}$ for all $r=1, \ldots, \lfloor \beta \rfloor$ and all $x\in [0,1].$ Without loss of generality, we may assume that $x<y.$ Using Lemma \ref{lem.fx_local_bd} and the mean value theorem, we can argue as for Equation (3.5) in \cite{RaySchmidt-Hieber2015c} and find for $s\leq \lfloor \beta \rfloor-1$ and some $\xi \in [x,y],$ $|f_0^{(s)}(x)^{m_s}-f_0^{(s)}(y)^{m_s}|\leq m_s |f_0^{(s+1)}(\xi)f_0^{(s)}(\xi)^{m_s-1}| |x-y| \lesssim R^{\frac{sm_s+1}{\beta}} f_0(x_{j-1})^{-\frac{1}{\beta}+\frac{\beta-s}{\beta}m_s}\Delta_j^{1-(\beta-r)}|x-y|^{\beta-r}$ and $|f_0^{(\lfloor \beta \rfloor)}(x)-f_0^{(\lfloor \beta \rfloor)}(y)|\lesssim R |x-y|^{\beta-r}.$ In order to control $|(f_0\psi_j)^{(r)}(x)-(f_0\psi_j)^{(r)}(y)|,$ we rewrite this expression using  \eqref{eq.f0phij_expansion} with $r=\lfloor \beta \rfloor$ and control each factor separately, applying the inequality $|ab -a'b'|\leq |a-a'| |b| + |a'| |b-b'|$ which holds for any $a,a',b,b' \in \mathbb{R}.$ This gives
\begin{align*}
	|f_0 \psi_j|_{\mC^\beta} \lesssim \alpha \sum_q \Big(\frac{\Delta_j^\beta}{f_0(x_{j-1})} \Big)^{\frac{r-M_q}{\beta}}
	\lesssim \alpha,
\end{align*} %(The largest seminorm is induced by the term with $q=r, m_1=r, M_r=r$).
where for the second step we used $\max_j \Delta_j^\beta/ f_0(x_{j-1})\rightarrow 0.$ Thus, $|f_0\psi_j|_{\mC^\beta}\leq \delta/2$ for $\alpha$ small and all sufficiently large $n.$

{\it (ii):} We first show that $|f_\theta|_{\mH^\beta}\leq |f_0|_{\mH^\beta}+\delta.$ Equation \eqref{eq.sup_norm_phi_j} implies $|f_\theta(x)/f_0(x)| =|1+ \sum_{j=1}^m \theta_j\psi_j(x)| =1+o(1),$ uniformly over $x.$ It is  thus enough to prove $|f_\theta^{(r)}(x)| \leq  (|f_0|_{\mH^\beta}+\delta/2)^{\frac{r}{\beta}} |f_0(x)|^{\frac{\beta-r}{\beta}}$ for $r=1,\ldots, \lfloor \beta \rfloor.$ If for any $r=1,\ldots,\lfloor \beta \rfloor,$ 
\begin{align}
	\big |(f_0\psi_j)^{(r)}(x) \big| \leq [(R+\delta/2)^{\frac{r}{\beta}}-R^{r/\beta}] |f_0(x)|^{\frac{\beta-r}{\beta}}, \ \text{for all} \ x\in [x_{j-1},x_j],  \ j=1,\ldots,m,
	\label{eq.toshow_fct_in_Hoeld_space}
\end{align}
then, since $x\mapsto (x+b)^\alpha-x^\alpha$ for $b, x>0$ and $0<\alpha \leq 1$ is  monotone decreasing and $| f_0^{(r)}(x)|\leq |f_0|_{\mH^\beta}^{\frac{r}{\beta}} |f_0(x)|^{\frac{\beta-r}{\beta}}$ by assumption, $$|f_\theta^{(r)}(x)|\leq |f_0^{(r)}(x)| + |(f_0\psi_j)^{(r)}(x)| \leq (|f_0|_{\mH^\beta}+\delta/2)^{\frac{r}{\beta}} |f_0(x)|^{\frac{\beta-r}{\beta}}.$$
It thus remains to show \eqref{eq.toshow_fct_in_Hoeld_space}. To see this, use \eqref{eq.f0phij_expansion} and $f_0\in \mH^\beta(R').$ This yields $|(f_0\psi_j)^{(r)}| \lesssim \sum_q (\Delta_j f_0^{-\frac{1}{\beta}})^{\beta-M_q}f_0^{\frac{\beta-r}{\beta}},$ which implies \eqref{eq.toshow_fct_in_Hoeld_space} for sufficiently large $n$ since $M_q \leq \lfloor \beta \rfloor < \beta$ and $\max_j \Delta_j^\beta/ f_0(x_{j-1}) \rightarrow 0.$ The previous step also shows that $\|f_\theta^{(\lfloor \beta \rfloor)}\|_\infty \leq \|f_0^{(\lfloor \beta \rfloor)}\|_\infty+ \delta.$ 
\end{proof}

\subsubsection*{Proof of Proposition \ref{prop.pjqj_expansion}}

We use $\lesssim,$ $\gtrsim$ and the big-O notation in the same way as in Theorem \ref{thm.lb}.

{\it Expansion of $q_{\pm, j}(\theta_j^0)$:} Recall that $f_\theta = f_0(1+\sum_{j=1}^m \theta_j\psi_j)$ and that the $\psi_j$ have disjoint support. Using the identity $\sqrt{z}-1 = \tfrac 12 (z-1) - \tfrac 18 (z-1)^2 + \tfrac 18 (z-1)^3(3+\sqrt{z})/(\sqrt{z}+1)^3$ for $z=1+\theta_j\psi_j(x)$ and $z=1 +\theta_j^0 \psi_j(x),$ together with $\mu_{j,2} =\alpha^2/n$ and \eqref{eq.sup_norm_phi_j}, we find for $\theta_j\neq \theta_j^0,$
\begin{align*}
	D_j:=n\int_{x_{j-1}}^{x_j}\big(\sqrt{f_\theta(x)}-\sqrt{f_{\theta_0}(x)}\big)^2 dx
	= \alpha^2 + O\big( n \mu_{j,4}\big)
%	\label{eq.GWN_int_approx}
\end{align*} 
and in particular, $D_j \geq \alpha^2/2$ for all $j$ if $n$ is large enough. Therefore, by Taylor expansion and straightforward computations,
\begin{align*}
	q_{\pm, j}(\theta_j^0) = Q_{\theta_0}(Q_j(-\theta_j^0)\geq e^{\pm 2\alpha \theta_j^0}) = \Phi\big(-D_j^{1/2} \mp \alpha \theta_j^0 D_j^{-1/2}\big)
	= \Phi(-\alpha \mp \theta_j^0)+ O\Big(\frac{\alpha^2}{nF_j}\Big),
\end{align*}
which proves the first part of the proposition.

{\it Expansion of $p_{\pm, j}(\theta_j^0)$:} Throughout this part of the proof we make freely use of the inequalities \eqref{eq.Fj_gammaj_ineqs} and \eqref{eq.sup_norm_phi_j}. For a real number $b$ with $1-|b|>0,$ consider the difference $\log(1+b)-\log(1-b).$ By a fourth order Taylor expansion of both log terms around one, we find
\begin{align*}
	&\big| \log(1+b)-\log(1-b) -2b
	\big| \leq \frac 23 |b|^3 + \frac{b^4}{2(1-|b|)^4}.
\end{align*}
Recall the definition of $P_j(\theta_j)$ in \eqref{eq.LR_fact}. With $b=\theta_j \psi_j(X_i^{(j)})$, the likelihood ratio for $\theta_j$ in the Poisson experiment $\mE_n^P$ is
\begin{align}
	P_j(\theta_j)
	&= \exp\Big( r_{j,n}+   (\theta_j -\theta_j^0)  \sum_{i=1}^{N_j} \psi_j(X_i^{(j)}) \Big),
	\label{eq.LR_approx_dens_est}
\end{align}
for a suitable remainder term $r_{j,n}$ satisfying $|r_{j,n}| \lesssim N_j \|\psi_j\|_\infty^3.$ Due to \eqref{eq.sup_norm_phi_j}, there is a constant $c_r$ such that 
\begin{align*}
	|r_{j,n}| \leq 2c_r N_j \alpha^3 (nF_j)^{-3/2} 
%	\label{eq.rjn_bd}
\end{align*}
(the factor $2$ allows us to simplify expressions later). Define $E_j:=E_{\theta_0}[\psi_j(X_1^{(j)})]$ and $s_j:= \Std_{\theta_0}(\psi_j(X_1^{(j)})).$  Let
\begin{align*}
	\xi_j 
	= \sqrt{N_j} \frac{\frac 1{N_j} \sum_{i=1}^{N_j}\psi_j(X_i^{(j)}) - E_j}{s_j}
%	\label{eq.xi_j_def}
\end{align*}
and observe that
\begin{align*}
	E_j
	&= \frac 1{F_j} \int \psi_j(x) f_0(x) \big(1 + \sum_{j=1}^m \theta_j^0  \psi_j(x) \big) dx
	=  \frac{\alpha^2}{nF_j} \theta_j^0 
\end{align*}
and
\begin{align*}
	s_j^2
	&= F_j^{-1} \int \psi_j^2(x)  f_0(x) \big(1 + \sum_{j=1}^m \theta_j^0  \psi_j(x) \big) dx  - E_j^2 
	= \frac{\alpha^2}{nF_j} 	+ \frac{\mu_{j,3}}{F_j}\theta_j^0 - \frac{\alpha^4}{(nF_j)^2} ,
%	\label{eq.explicit_var_for_lb}
\end{align*}
implying for sufficiently large $n,$
\begin{align}
		\frac{\alpha}{2\sqrt{nF_j}}\leq s_j \leq \frac{2\alpha}{\sqrt{nF_j}} \quad \text{for all} \ j=1,\ldots,m.
		\label{eq.sj_asymp}
\end{align}
Since $(1+x)^{-1/2}=1-x/2+O(x^2)$ for $|x|\leq 1/2$, we also have
\begin{align}
	\frac{\alpha}{ \sqrt{nF_j}s_j}
	= \big(1+ n \mu_{j,3}\alpha^{-2}\theta_j^0 -\alpha^2/nF_j\big)^{-1/2}
	= 1- \frac{n\mu_{j,3}\theta_j^0}{2\alpha^2} +O\Big(\frac{\alpha^2}{nF_j}\Big).
	\label{eq.sj_asymp2}
\end{align}
The $r$-th central moment of $\theta_j^0\psi_j(X_1^{(j)})$ will be denoted by  $\mathfrak{m}_{j,r}.$ With \eqref{eq.sj_asymp2},
\begin{align}
	\frac{\mathfrak{m}_{j,3}}{s_j^3}
	&=
	\theta_j^0  \frac{E_{\theta_0}[\psi_j(X_1^{(j)})^3] -3E_{\theta_0}[\psi_j(X_1^{(j)})^2]E_j + 2E_j^3}{s_j^3}  \notag \\
	&=
	\theta_j^0\frac{\mu_{j,3}}{ F_j s_j^3}+  O\big(\alpha/(nF_j)^{1/2}\big)  \notag \\
	&=\theta_j^0 \mu_{j,3}n^{3/2}\sqrt{F_j}\alpha^{-3}+  O\big(\alpha/(nF_j)^{1/2}\big) 
	\label{eq.cumulant_expansions}
\end{align}
and with \eqref{eq.mujr_higher_moments}, $\max_j \mathfrak{m}_{j,r}/s_j^r\lesssim\max_j  E_{\theta_0}[\psi_j(X_1^{(j)})^r]/s_j^r\lesssim 1.$ We can further rewrite \eqref{eq.LR_approx_dens_est} as
\begin{align}
	P_j(-\theta_j^0)= \exp\Big(r_{j,n} -2\sqrt{N_j} s_j \theta_j^0\xi_j -  N_j\frac{2\alpha^2}{nF_j}  \Big).
	\label{eq.LR_approx_dens_est2}
\end{align}
For $\ell=1,2,$ let
\begin{align*}
	B_{j,n}^{(\ell)}:=-\frac{\alpha^2 \sqrt{N_j}} {nF_j s_j} + (-1)^\ell \frac{c_r\alpha^3 \sqrt{N_j}}{(nF_j)^{3/2}s_j} \mp \frac{\alpha\theta_j^0}{\sqrt{N_j}s_j}
\end{align*}
and observe that the only randomness in $B_{j,n}^{(\ell)}$ comes from $N_j.$ Recall that $\widehat{\theta}_{\pm,j}^P \neq \theta_j^0$ iff $P_j(-\theta_j^0) \geq e^{\pm 2\alpha \theta_j^0}.$ Due to \eqref{eq.LR_approx_dens_est2}, we therefore have $\widehat{\theta}_{\pm,j}^P \neq \theta_j^0$ iff $ r_{j,n} -2\theta_j^0 \sqrt{N_j} s_j \xi_j -  4\alpha^2 N_j (nF_j(f_{\theta_0}))^{-1}  > \pm 2\alpha \theta_j^0$ and thus $$P_{\theta_0}(\theta_j^0\xi_j \leq  B_{j,n}^{(1)}) \leq  p_{j, \pm}(\theta_0)\leq P_{\theta_0}(\theta_j^0\xi_j \leq  B_{j,n}^{(2)}).$$ In the next step, we show that for $\ell=1,2,$ $ P_{\theta_0}(\theta_j^0\xi_j \leq  B_{j,n}^{(\ell)})=\Phi(-\alpha  \mp \theta_j^0)\pm n\mu_{j,3}/(6\alpha^2)\phi(-\alpha \mp \theta_j^0)+O(\alpha^2/\sqrt{nF_j}).$ To do that we need the following Edgeworth expansion, which is a simplification of Petrov \cite{petrov1975sums}, p.159 with $k=3.$
\begin{thm}
\label{thm.edgeworth}
Let $(Y_i)_{i=1,\ldots,M}$ be i.i.d. random variables with $EY_1=0,$ $\sigma :=\Std(Y_1)$ and $E[Y_1^4]<\infty.$ Let $v(t)=Ee^{itY_1}$ and denote by $G_M$ the c.d.f. of $\xi = M^{-1/2}\sum_i Y_i/\sigma.$ There exists an absolute constant $C$ such that for any $t\in \mathbb{R},$
\begin{align*}
	\Big|G_M(t) -\Phi(t) - \frac{1}{\sqrt{M}} \frac{E[Y_1^3]}{6\sigma^3} (1-t^2) \phi (t) 
	\Big|
	\leq C\frac {E[Y_1^4]}{\sigma^{4}M}  + C\Big( \sup_{|u|\geq \sigma^2/(12E|Y_1|^3)} |v(u)| + \frac{1}{2M}\Big)^M M^6.
\end{align*}
\end{thm}
To compute $P_{\theta_0}(\theta_j^0\xi_j \leq  B_{j,n}^{(\ell)}) ,$ we first condition on $N_j.$ The bounds below are only useful if $N_j>0$ and we will later see that this is enough. Using Theorem \ref{thm.edgeworth}, there exists a constant $C'$ such that
\begin{align*}
	\big|P_{\theta_0}\big(\theta_j^0\xi_j \leq  y \big| N_j \big) 
	- \Phi\big(y\big)
	-\frac{\mathfrak{m}_{j,3}}{6\sqrt{N_j} s_j^3}(1-y^2)\phi(y) \big|\leq
	 \frac{C'}{N_j} + C'\Big( \sup_{|t|\geq \delta_j } |v_j(t)| + \frac{1}{2N_j}\Big)^{N_j} N_j^6
\end{align*}
with $|v_j(t)| = | E_{\theta_0}\exp(it [\psi_j(X_1^{(j)})-E_j]) | =| E_{\theta_0}\exp(it \psi_j(X_1^{(j)}))|$ and $\delta_j = s_j^2/(12 \sqrt{\mathfrak{m}_{j,6}}).$ 
\begin{lem}
For $n$ sufficiently large, there exists a constant $L<1$ such that $\max_j\sup_{|t|\geq \delta_j} |v_j(t)| \leq L <1.$
\end{lem}

\begin{proof}
To simplify the proof, write $\kappa_j = \alpha \gamma_j \Delta_j^\beta/f_0(x_{j-1})$ and observe that with \eqref{eq.sj_asymp}, $\kappa_j \asymp s_j.$ Let $W$ be a random variable with Lebesgue density $f_W$ and $V=g(W)$ for a continuously differentiable function $g.$ Let $v$ be such that for all $w\in g^{-1}(v)$ the derivative $g'(\omega)$ is non-zero. For such a $v$, the density $f_V$ of $V$ is given by
\begin{align*}
	f_V(v) = \sum_{w\in g^{-1}(v)} \frac{f_W(w)}{|g'(w)|}.
\end{align*}
Since $K$ is by assumption continuously differentiable and $K'(u)=0$ for only finitely many different values of $u\in [0,1],$ the density of $\psi_j(X_1^{(j)})$ with $X_1^{(j)}$ generated from $P_{\theta_0}$ is contained in the support $[\kappa_j \inf K, \kappa_j \sup K]$ and almost everywhere bounded from below by
\begin{align*}
	\inf_{x\in [x_{j-1},x_j]} \frac{f_{\theta_0}(x)}{\kappa_j \|K'\|_{\infty}f_0(x)}.
\end{align*}
By \eqref{eq.sup_norm_phi_j}, we have that for sufficiently large $n$ this is lower bound by $1/(2\kappa_j \|K'\|_\infty).$ Subtracting and adding  $1/(2\kappa_j \|K'\|_\infty)$ to the density, we obtain for the characteristic function,
\begin{align*}
	|v_j(t)| &\leq 1 - \frac{\sup K - \inf K}{2\|K'\|_{\infty}}
	+ \frac{1}{2\kappa_j \|K'\|_\infty}\Big|\int_{\kappa_j \inf K}^{\kappa_j \sup K} e^{itu} du \Big| \\
	&=1 - \frac{\sup K - \inf K}{2\|K'\|_{\infty}} 
	+ 
	\Big|\frac{\sin(t\kappa_j(\sup K - \inf K)/2)}{t\kappa_j \|K'\|_\infty}\Big|.	
\end{align*}
Observe that $\delta_j = 1/(12 s_j\sqrt{\mathfrak{m}_{j,6}/s_j^6}) \gtrsim 1/s_j\gtrsim 1/\kappa_j$ and therefore there exits a positive constant that does not depend on $j$ such that $\sup_{|t|\geq \delta_j } |v_j(t)|\leq \sup_{t\kappa_j \geq c>0} |v_j(t)|.$ Since the sinc-function $\sin(x)/x$ is smaller than one whenever $x$ is bounded away from zero, this implies $\max_j \sup_{t\kappa_j \geq c>0} |v_j(t)|\leq L <1.$
\end{proof}

As a consequence of the previous lemma, we obtain
\begin{align*}
	\big|P_{\theta_0}\big(\theta_j^0\xi_j \leq  y \big| N_j \big) 
	- \Phi\big(y\big)
	-\frac{\mathfrak{m}_{j,3}}{6\sqrt{N_j} s_j^3}(1-y^2)\phi(y) \big|\lesssim 
	 \frac{1}{N_j}.
\end{align*}
For any real numbers $y,z$, there exist $\eta,\eta', \eta'' \in \mathbb{R}$ such that by Taylor expansion $\Phi(y) = \Phi(z) + (y -z) \phi(z) +\frac 12(y-z)^2 \phi'(\eta)$ as well as $\phi(y)  = \phi(z) + (y-z) \phi'(\eta')$ and $y^2\phi(y)  =z^2 \phi(z) + (y-z) [2\eta'' \phi(\eta'')+(\eta'')^2\phi'(\eta'')].$ Together with $\max_j \mathfrak{m}_{j,3}/s_j^3 \lesssim 1$ this yields
\begin{align}
	\Big|
	&P_{\theta_0}\big(\theta_j^0\xi_j \leq y \big| N_j \big) 
	- \Phi(z) -(y-z) \phi(z)
	-\frac{\mathfrak{m}_{j,3}}{6\sqrt{N_j} s_j^3}(1-z^2)\phi(z) \Big|\lesssim
	 \frac{1}{N_j} +(y-z)^2.
	 \label{eq.gen_statement_for_pj_cond}
\end{align}
In the next step, we show that 
\begin{align}
	&\big| B_{j,n}^{(\ell)} + \alpha  \pm \theta_j^0 \mp  \frac{ n \mu_{j,3}}{2\alpha^2} + \frac{\alpha^2\mp \alpha \theta_j^0}{ \sqrt{nF_j}  s_j}   \frac{N_j-nF_j}{2nF_j} \big|
	\lesssim \frac{\sqrt{N_j} \alpha^2}{ nF_j} + \frac{|N_j-nF_j|^2}{(nF_j)^2} \Big(1+\frac{\sqrt{nF_j}}{\sqrt{N_j}}\Big) + \frac{\alpha^2}{\sqrt{nF_j}}.
	\label{eq.Bjn_alpha_est}
\end{align}
For that, decompose $B_{j,n}^{(\ell)} + \alpha  \pm \theta_j^0$ into
\begin{align}
	(-1)^\ell \frac{c_r\alpha^3\sqrt{N_j}}{(nF_j)^{3/2}s_j}
	-
	\frac{\alpha^2 }{ \sqrt{nF_j}  s_j} \Big( \frac{\sqrt{N_j} }{\sqrt{n F_j}}  -1 \Big)
	+
	(\alpha \pm \theta_j^0) \Big(  1-\frac{\alpha }{ \sqrt{nF_j}  s_j} \Big)
	\pm \frac{\alpha\theta_j^0}{\sqrt{nF_j}s_j}\Big( 1- \frac{\sqrt{nF_j}}{\sqrt{N_j}} \Big).
	\label{eq.Bjn_alpha_decomp}
\end{align}
Using \eqref{eq.sj_asymp}, the first term is of order $\sqrt{N_j} \alpha^2 /(nF_j).$ Applying the identity $\sqrt{z}-1=\tfrac 12 (z-1)-\tfrac 12 (z-1)^2/(\sqrt{z}+1)^2$ to $z= N_j/(n F_j),$
\begin{align}
 \frac{\sqrt{N_j} }{\sqrt{n F_j}}  -1 
 	= \frac{N_j-nF_j}{2nF_j} +O\Big(\frac{(N_j-nF_j)^2}{(nF_j)^2} \Big),
 	\label{eq.Nj_ratio_expansion}
\end{align}
which controls the second term in \eqref{eq.Bjn_alpha_decomp}. For the last term, using $1-z^{-1/2} = \sqrt{z}-1 -(\sqrt{z}-1)^2/\sqrt{z}$ together with \eqref{eq.Nj_ratio_expansion} gives
\begin{align*}
	1- \frac{\sqrt{n F_j}}{\sqrt{N_j} }
	 = 
	 \frac{N_j-nF_j}{2nF_j} +O\Big(\frac{(N_j-nF_j)^2}{(nF_j)^2} \Big(1+\frac{\sqrt{nF_j}}{\sqrt{N_j}}\Big)\Big).
\end{align*} 
Finally, the third term of \eqref{eq.Bjn_alpha_decomp} can be controlled with \eqref{eq.sj_asymp2} and this proves \eqref{eq.Bjn_alpha_est}.

Using \eqref{eq.sup_norm_phi_j}, $P(N_j=0)=\exp(-nF_j)$ decreases faster to zero than any power of $1/(nF_j).$ Considering each term in \eqref{eq.Bjn_alpha_est} individually using Lemma \ref{lem.Poisson_moments}(ii), that $EN_j^{1/2} \leq [EN_j]^{1/2}$ and the Cauchy-Schwarz inequality gives
\begin{align*}
	E_{\theta_0}[B_{j,n}^{(\ell)}\mathbf{1}(N_j>0)] = -\alpha  \mp \theta_j^0 \pm \frac{n\mu_{j,3}}{2\alpha^2} + O\Big(\frac{\alpha^2}{\sqrt{nF_j}}\Big)
\end{align*}
and $E_{\theta_0}[(B_{j,n}^{(\ell)}+\alpha \pm \theta_j^0)^2\mathbf{1}(N_j>0)]\lesssim 1/(nF_j).$ Applying this to \eqref{eq.gen_statement_for_pj_cond} with $y=B_{j,n}^{(\ell)}$ and $z=-\alpha \mp \theta_j^0,$ using \eqref{eq.sup_norm_phi_j}, \eqref{eq.sj_asymp} and the expression for the standardized cumulant $\mathfrak{m}_{j,3}/s_j^3$ in \eqref{eq.cumulant_expansions} gives
\begin{align*}
	E_{\theta_0}[P(\theta_j^0\xi_j \leq  B_{j,n}^{(\ell)} \big| N_j )] 
	&= E_{\theta_0}[P(\theta_j^0\xi_j \leq  B_{j,n}^{(\ell)} \big| N_j )\mathbf{1}(N_j>0)]+ O(e^{-(nF_j)}) \\
	&=
	\Phi(-\alpha\mp \theta_j^0) \pm \frac{n\mu_{j,3}}{6\alpha^2}\phi(-\alpha \mp \theta_j^0) +O(\alpha^2/\sqrt{nF_j}).
\end{align*}
This finally yields
\begin{align*}
	p_{\pm, j}(\theta_j^0)
	&=
	\Phi(-\alpha\mp \theta_j^0) \pm \frac{n\mu_{j,3}}{6\alpha^2}\phi(-\alpha \mp \theta_j^0) +O(\alpha^2/\sqrt{nF_j}),
%	\label{eq.pj_final}
\end{align*}
which completes the proof of the second assertion of the proposition.\qed

\subsubsection*{Remaining proofs}

\begin{proof}[Proof of Proposition \ref{prop.MLE_is_mininmax}]
We first prove that $\widehat{\theta}_{\pm}^P$ is a Bayes estimator in the Poisson model. Denote by $p_\theta$ the density of $P_\theta^n$ with respect to some dominating measure $\mu.$ In step $(i),$ we prove that any estimator
\begin{align}
	\widetilde \theta \in \argmax_{\theta \in \Theta} \sum_{\theta': \rho(\theta, \theta') \leq A} p_{\theta'} \pi_{\pm}(\theta')
	\label{eq.argmin_over_neighbh}
\end{align}
is a Bayes estimator. In step $(ii),$ we show that $\widehat \theta_{\pm}^P$ is always contained in the argmax.

{\it (i):} Observe that
\begin{align*}
	\inf_{\widehat{\theta}} \sum_{\theta_0\in \Theta} P_{\theta_0}^n \big(\rho(\widehat \theta, \theta_0) \geq A \big) \pi_{\pm}(\theta_0)
	& = 1 -  \sup_{\widehat{\theta}} \int \sum_{\theta_0 \in \Theta}  \mathbf{1}(\rho(\widehat \theta, \theta_0) < A) p_{\theta_0}\pi_{\pm}(\theta_0)  d\mu.
\end{align*}
Now $\sum_{\theta_0 \in \Theta}  \mathbf{1}(\rho(\widehat \theta, \theta_0) < A) p_{\theta_0}\pi_{\pm}(\theta_0) \leq \sup_\theta \sum_{\theta_0 \in \Theta}  \mathbf{1}(\rho(\theta, \theta_0) < A) p_{\theta_0}\pi_{\pm}(\theta_0),$ which does not depend on $\widehat \theta$ anymore. The upper bound is attained by any estimator $\widetilde \theta$ satisfying \eqref{eq.argmin_over_neighbh}.

{\it (ii):} Let $\widehat{\theta}$ be an arbitrary estimator. If $L=\sum_{j=1}^m \mathbf{1}(\widehat \theta_{\pm,j}^P \neq \widehat \theta_j)$ is positive, we can find a sequence of estimators $\widehat \theta_0:= \widehat \theta, \widehat \theta_1, \ldots, \widehat\theta_{L-1}, \widehat\theta_L :=\widehat \theta_{\pm}^P$ such that for any $r=1, \ldots, L,$ $\widehat \theta_r$ and $\widehat \theta_{r-1}$ differ in exactly one entry. Write $U_r = \{\theta: \rho(\widehat \theta_r, \theta) \leq A\}.$ It is enough to prove that the sequence
\begin{align}
	\sum_{\theta\in U_r} \pi_{\pm}(\theta)p_{\theta}, \quad r=0,\ldots, L 
	\label{eq.to_prove_mon_incr}
\end{align}
is monotone increasing in $r.$ Let $\theta=(\theta_1, \ldots, \theta_m)$ and observe that by \eqref{eq.LR_fact} the densities $p_\theta$ and the priors $\pi_{\pm}$ factorize with respect to the components $\theta_j,$ that is $p_{\theta} = \prod_{j=1}^m p_{\theta_j}$ and $\pi_{\pm}(\theta)=\prod_{j=1}^m  \pi_{\pm}(\theta_j).$ Going from $\widehat \theta_r$ to $\widehat \theta_{r+1}$ we increase one of the factors, say the first one. It thus remains to show that
\begin{align*}
	\sum_{\theta\in U_r} \pi_{\pm}(\theta_1)p_{\theta_1} \pi_{\pm}(\theta_2)p_{\theta_2}\cdot \ldots \cdot \pi_{\pm}(\theta_m)p_{\theta_m}
	&\leq \sum_{\theta\in U_{r+1}} \pi_{\pm}(\theta_1)p_{\theta_1} \pi_{\pm}(\theta_2)p_{\theta_2}\cdot \ldots \cdot \pi_{\pm}(\theta_m)p_{\theta_m} \\
	&= \sum_{\theta\in U_r} \pi_{\pm}(-\theta_1)p_{-\theta_1} \pi_{\pm}(\theta_2)p_{\theta_2}\cdot \ldots \cdot \pi_{\pm}(\theta_m)p_{\theta_m}.
\end{align*}
If $(\theta_1,\theta_2,\ldots, \theta_m )$ and $(-\theta_1,\theta_2,\ldots, \theta_m )$ are both elements of $U_r,$ the respective terms cancel in both sums. We are thus left with the case that $(\theta_1,\theta_2,\ldots, \theta_m ) \in U_r$ and $(-\theta_1,\theta_2,\ldots, \theta_m ) \not \in U_r.$ In this case, we must have $\sum_{j=1}^m \rho_j |\widehat \theta_j^r -\theta_j|\leq 2A$ and $\rho_1|\widehat \theta_1^r +\theta_1| + \sum_{j=2}^m \rho_j |\widehat \theta_j^r -\theta_j|> 2A,$ implying $\theta_1=\widehat{\theta}_1^r.$ Since by construction $\pi_{\pm}(\widehat{\theta}_1^r)p_{\widehat{\theta}_1^r} \leq \pi_{\pm}(\widehat{\theta}_1^{r+1})p_{\widehat{\theta}_1^{r+1}}=\pi_{\pm}(-\widehat{\theta}_1^{r}) p_{-\widehat{\theta}_1^r},$ we finally see that \eqref{eq.to_prove_mon_incr} is monotone increasing in $r$ and this completes the proof of $(ii).$

The same arguments hold for the Gaussian experiment, proving that $\widehat{\theta}_{\pm}^G$ are Bayes estimators as well.
\end{proof}

\begin{proof}[Proof of Lemma \ref{lem.Bern_diff}]
By Remark \ref{rem.monotonicity_of_Bern_sum}, it is enough to prove the result for $p_j=q_j + q_j(1-q_j)\omega  \beta_j.$

Define the set $\mathcal{V}:=\{ I\subset \{1,\ldots,m\}: \sum_{j=1}^m \beta_j>A\}$ and notice that
\begin{align*}
	\P\big(\sum_{j=1}^m \beta_j Z_j(p_j)>A \big)
	=\sum_{V\in \mathcal{V}} \prod_{j\in V} p_j\prod_{j\in V^c} (1-p_j)
	\geq  \P\big(\sum_{j=1}^m \beta_j Z_j(q_j)>A \big) \inf_{V\in \mathcal{V}} \prod_{j\in V} \frac{p_j}{q_j}\prod_{j\in V^c} \frac{1-p_j}{1-q_j}.
\end{align*}
Moreover, for any $V\in \mathcal{V},$
\begin{align*}
	 R(V):=\log \prod_{j\in V}\frac{p_j}{q_j}\prod_{j\in V^c} \frac{1-p_j}{1-q_j}
	=\sum_{j\in V} \log\big(1+(1-q_j)\omega \beta_j \big)+ \sum_{j\in V^c} \log\big( 1- q_j\omega  \beta_j\big).
\end{align*}
For $0\leq x\leq1/2,$ $\log(1+x)\geq x-x^2/2$ and $\log(1-x)\geq-x-2x^2.$ Since $\omega \leq 1/2,$
\begin{align*}
	R(V) \geq \omega \sum_{j\in V} \beta_j -\omega \sum_{j=1}^m \beta_j q_j - 2\omega^2 \sum_{j=1}^m \beta_j^2
	\geq \omega A -\omega \sum_{j=1}^m \beta_j q_j -2\omega^2  \sum_{j=1}^m \beta_j^2.
\end{align*}
\end{proof}

\section{Results for globalization}
\label{sec.global}

We now derive estimators for the globalization step of the proofs. Denote by $\Theta(f)$ the local parameter space about a point $f$. We must show that if $f_0$ is the true parameter, there exists an estimator $\widehat f_n$ such that $f_0 \in \Theta(\widehat f_n)$ with high probability. To avoid measurability issues, we restrict $\widehat f_n$ to take values in a finite subset $\Theta' \subset \Theta,$ whose cardinality may depend on $n.$ 

The construction of such estimators is similar in all the cases. In a first step, we split the sample and use the first part for a preliminary kernel density estimator of $f_0.$ The second part of the sample is then used for another estimator $\widehat f_{2n}$ of $f_0$, whose bandwidth depends locally on the first estimator. This estimator is then shown to satisfy $f_0 \in \Theta(\widehat f_{2n})$ with high probability. Finally, we construct from $\widehat f_{2n}$ an estimator  $\widehat f_n$ with values in a finite subset of $\Theta.$ By the Arzel\` a -Ascoli theorem, the H\"older ball $\mC^\beta(R)$ is compact with respect to the uniform topology. For any decreasing positive sequence $(\delta_n),$ the parameter space $\Theta \subset \mH^\beta(R) \subset \mC^\beta(R)$ can therefore be covered with respect to the uniform norm by finitely many $\delta_n$-balls with centers in $\Theta.$ The set of centers $\Theta'$ form a finite subset of $\Theta.$ Define the estimator $\widehat f_n$ as any element of $\Theta'$ (i.e. center of a ball) that lies in $\Theta(\widehat f_{2n}).$ We next show that if $f_0 \in \Theta(\widehat f_{2n}),$ then the center of the ball covering $f_0$ also lies in $\Theta(\widehat f_{2n}),$ provided that $\delta_n$ is chosen small enough. This shows that with high probability $\widehat f_n \in \Theta' \subset \Theta.$ We finally show that this also implies the assertion that $f_0 \in \Theta(\widehat f_n)$ with high probability.

We begin with a preliminary result on kernel density estimators. For the definition and construction of an $\ell$-th order kernel see for instance \cite{tsybakov2009}, Definition 1.3 and Section 1.2.2.

\begin{thm}
\label{thm.dens_estimation}
Work in the density estimation experiment $\mE_n^D(\Theta).$ Consider a kernel density estimator $\widehat f_{nh_x}=(nh_x)^{-1}\sum_{i=1}^n K((X_i-\cdot )/h_x)$ for a positive bandwidth function $h_x> 0$ and some $\lfloor \beta \rfloor$-th order kernel $K$ with support on $[-1,1].$ Let $a=a(\beta)$ be the constant from Lemma \ref{lem.fx_local_bd}. If $f\in \mH^\beta(R),$ then with probability at least $1-2n^{1-\gamma},$
\begin{align*}
\big| \widehat f_{nh_x}(x) - f(x)\big| 
&\leq R\Big(\|K\|_{\infty} + \frac{1}{a^\beta}\Big) h_x^{\beta}  +
	2\gamma (\|K\|_{\infty}+\|K\|_2^2) \frac{\log n}{nh_x} + \|K\|_2\sqrt{8\gamma f(x) \frac{\log n}{nh_x}} \\
	&\leq R\Big(\|K\|_{\infty} + \frac{1}{a^\beta}\Big) h_x^{\beta}  +
	2\gamma (\|K\|_{\infty}+5\|K\|_2^2) \frac{\log n}{nh_x} +\frac 12 f(x)
\end{align*}
for all $x \in \{1/n, 2/n, \ldots, 1\}.$
\end{thm}

\begin{proof}
Using Proposition 1.2 in \cite{tsybakov2009}, we can bound the bias by $|E[\widehat f_{nh_x}(x)] - f(x)|\leq \frac{Rh_x^{\beta}}{\lfloor \beta \rfloor !} \int |u^\beta K(u) | du\leq 2R \|K\|_{\infty}h_x^{\beta}.$ Recall Bernstein's inequality: if $Z_1, \ldots,Z_n$ is a sequence of i.i.d. centered, real-valued random variables such that $|Z_i|\leq 1$ a.s., then for any $t>0,$
\begin{align*}
	P\big( \big|\sum_{i=1}^n Z_i\big| >t\big) \leq 2 \exp\Big(-\frac{\tfrac 12 t^2}{nE[Z_1^2]+t/3}\Big).
\end{align*}
Defining $G_h f(x) := \sup_{z\in [x-h_x,x+h_x]} f(z),$ this shows that
\begin{align*}
	P_f^n\Big( \Big| \sum_{i=1}^n K\Big( \frac{X_i-x}{h_x}\Big)- E\Big[K\Big( \frac{X_i-x}{h_x}\Big)\Big] \Big|
	\geq 2\gamma \|K\|_{\infty} \log n + 2\|K\|_2\sqrt{\gamma G_hf(x) nh_x \log n}\Big) 
	\leq 2n^{-\gamma}.
\end{align*}
Together with a union bound and the bound for the bias, this proves that with probability at least $1-2n^{1-\gamma},$
\begin{align*}
\big| \widehat f_{nh_x}(x) - f(x)\big| \leq 2R\|K\|_{\infty}h_x^{\beta}  +
	2\gamma \|K\|_{\infty} \frac{\log n}{nh_x} + 2\|K\|_2\sqrt{\frac{\gamma G_hf(x) \log n}{nh_x}}
\end{align*}
for all $x \in \{1/n, 2/n, \ldots, 1\}.$ Let $a=a(\beta)$ be the constant from Lemma \ref{lem.fx_local_bd}. This implies that $G_hf(x) \leq 2 f(x)$ whenever $a^{-\beta}R h_x^\beta \leq G_hf(x).$ If this does not hold, we simply use $G_hf(x) \leq a^{-\beta}R h_x^\beta$ so that $G_hf(x) \leq 2 f(x)+ a^{-\beta}R h_x^\beta$ for all $x.$ Using that for positive numbers $\sqrt{a+b}\leq \sqrt{a}+\sqrt{b}$ and $2\sqrt{uv} \leq u +2v,$ this finally gives that with probability at least $1-2n^{1-\gamma},$
\begin{align*}
\big| \widehat f_{nh_x}(x) - f(x)\big| \leq R\Big(\|K\|_{\infty} + \frac{1}{a^\beta}\Big) h_x^{\beta}  +
	2\gamma (\|K\|_{\infty}+\|K\|_2^2) \frac{\log n}{nh_x} + \|K\|_2\sqrt{8\gamma f(x) \frac{\log n}{nh_x}}
\end{align*}
for all $x \in \{1/n, 2/n, \ldots, 1\}.$ This proves the first inequality. For the second inequality, use $2\sqrt{uv} \leq u +2v$ again.
\end{proof}

\begin{proof}[Proof of Theorem \ref{thm.globalization_Poissonization}] 
In the Poisson intensity estimation experiment, we observe $X_1, \ldots, X_N$ with $N\sim \Poi(n).$ By Lemma \ref{lem.Poisson_moments}(iii), $P(N\geq n/2) \geq 1-2e^{-n/16}.$ Thus, on an event with probability $1- o(1/n),$ we can recover the density estimation model with sample size $\lfloor n/2\rfloor.$ It is therefore enough to prove the result for density estimation.  

Throughout the following let $K$ be an $\lfloor \beta \rfloor$-th order kernel with support on $[-1,1]$ and let $n_* :=\lfloor n/2 \rfloor \asymp n$ and $L_{n_*} := (\log n_*)/n_*.$ In the density estimation experiment, we can split the sample in two independent samples of size $n_*$  and use the first part of the sample to define the estimator $\widehat f_{1n_*}=(n_*h_{1n})^{-1}\sum_{i=1}^{n_*} K((X_i-\cdot )/h_{1n})$ with $h_{1n} = L_{n_*}^{1/(\beta+1)}.$ The second part of the sample is then used for the estimator $\widehat f_{2n_*}=(n_*\hat h_n)^{-1}\sum_{i=n^*+1}^{2n_*} K((X_i-\cdot )/\hat h_n)$ with $\hat h_n = L_{n_*}^{1/(\beta+1)} \vee (L_{n_*}\widehat f_{1n_*}(x) )^{1/(2\beta+1)}.$ By the compactness argument given at the beginning of Section \ref{sec.global}, $\Theta$ can be covered by finitely many $L^\infty$-balls of radius $L_{n^*}^{\beta/(\beta+1)}$ having centers in $\Theta.$ Let us define an estimator $\widehat f_n$ as any of the centers of the covering balls in the set
\begin{align*}
	\big\{ f\in \Theta : \big| \widehat{f}_{2n_*}\big(\tfrac in \big) - f\big(\tfrac in\big)\big| 
	\leq (C+1) L_{n_*}^{\beta/(\beta+1)} +C (L_{n_*}\widehat f_{2n_*}(x) )^{\beta/(2\beta+1)}, \ \  i=1,\ldots,n\big\}.
\end{align*}
If none of the centers are in this set then set $\widehat f_n :=f^*$ for some fixed parameter $f^* \in \Theta.$

Applying Theorem \ref{thm.dens_estimation} with $\gamma=2,$ there is a constant $C_1$ such that $|\widehat f_{1n_*}(i/n)-f_0(i/n)|\leq C_1 L_{n_*}^{\beta/(\beta+1)}+ f_0(i/n)/2$ for all $i=1,\ldots,n$ with probability at least $1-2n_*^{-1}.$ In particular, if $f_0(i/n) \geq 4C_1L_{n_*}^{\beta/(\beta+1)},$ then $\tfrac 14 f_0(i/n) \leq \widehat f_{1n_*}(i/n)\leq \tfrac 74 f_0(i/n).$ Applying Theorem \ref{thm.dens_estimation} with $\gamma =2$ to $\widehat f_{2n_*}$ conditionally on $X_1, \ldots X_{n_*},$ and treating the cases $f_0(i/n) \gtrless 4C_1L_{n_*}^{\beta/(\beta+1)}$ separately, gives for some constant $C_3,$
\begin{align*}
	\big | \widehat f_{2n_*}\big(\tfrac in \big) - f_0\big(\tfrac in \big) \big | 
	\leq C_3 L_{n_*}^{\beta/(\beta+1)} + C_3 \big( f_0\big(\tfrac in \big) L_{n_*}\big)^{\beta/(2\beta+1)} \ \ \text{for all} \ i=1,\ldots,n,
\end{align*}
with probability at least $1-4n_*^{-1}\geq 1-8/(n-1).$ From now on, let us work on the event where the previous inequalities hold. The switching relation in Lemma \ref{lem.switch_relation} shows that we can exchange $f_0$ by $\widehat f_{2n_*}$ on the right-hand side and therefore, for a constant $C_4,$
\begin{align*}
	\big | \widehat f_{2n_*}\big(\tfrac in \big) - f_0\big(\tfrac in \big) \big | 
	\leq C_4 L_{n_*}^{\beta/(\beta+1)} + C_4 \big( \widehat f_{2n_*}\big(\tfrac in \big) L_{n_*}\big)^{\beta/(2\beta+1)} \ \ \text{for all} \ i=1,\ldots,n.
\end{align*}
By construction, we can then conclude that if the constant $C$ in the definition of $\widehat{f}_n$ is taken to be larger than $C_4,$ $\widehat{f}_n$ must be a center of a ball from the covering and $| \widehat{f}_{2n_*}(i/n \big) - \widehat f_n(i/n)| \leq (C_4+1) L_{n_*}^{\beta/(\beta+1)} +C_4 (\widehat f_{2n_*}(i/n) L_{n_*} )^{\beta/(2\beta+1)}$ for all $i=1,\ldots,n.$ With Lemma \ref{lem.switch_relation}, we can replace  $\widehat f_{2n_*}(i/n)$ by $\widehat f_{n}(i/n)$ and this shows that for some constants $C_5, C_6$ and any $i=1,\ldots,n,$
\begin{align*}
	\big | f_0\big(\tfrac in \big) - \widehat f_n\big(\tfrac in \big) \big | 
	&\leq C_5 L_{n_*}^{\beta/(\beta+1)}
	+ C_5 \Big( \max\big(f_0\big(\tfrac in \big), \widehat f_n\big(\tfrac in \big)\big) L_{n_*}\Big)^{\beta/(2\beta+1)} \\
	&\leq 
	C_6 L_{n_*}^{\beta/(\beta+1)}
	+ C_6 \Big( \widehat f_n\big(\tfrac in \big) L_{n_*}\Big)^{\beta/(2\beta+1)},
\end{align*}
where the last step follows from Lemma \ref{lem.switch_relation} applied to $a_n = \max(f_0(i/n), \widehat f_n(i/n))$ and $b_n = \min (f_0(i/n), \widehat f_n(i/n)).$ Finally, let $x\in [0,1]$ be arbitrary and define $i_x := \argmin_i |x-\tfrac in|.$ Since $f_0, \widehat f_n \in \mH^\beta(R)$ and $n^{-(1\wedge \beta)} \leq L_{n_*}^{\beta/(\beta+1)},$ the triangle inequality gives
\begin{align*}
	\big | f_0(x) - \widehat f_n(x) \big | 
	&\leq 2Rn^{-(1\wedge \beta)}  + \big | f_0\big(\tfrac {i_x}n\big) - \widehat f_n\big(\tfrac {i_x}n\big)\big | \\
	&\leq 2Rn^{-(1\wedge \beta)} + C_6 L_{n_*}^{\beta/(\beta+1)}
	+ C_6 \Big( \big(\widehat f_n(x)+Rn^{-(1\wedge \beta)} \big) L_{n_*}\Big)^{\beta/(2\beta+1)} \\
	&\leq \big(2R +C_6(1+R^{\beta/(2\beta+1)})\big) L_{n_*}^{\beta/(\beta+1)}
	+ C_6 \Big(\widehat f_n(x) L_{n_*}\Big)^{\beta/(2\beta+1)}.
\end{align*}
Since $x$ was arbitrary, this shows that $f_0 \in \Theta_1^\beta (\widehat f_n )$ provided that the constant $C$ in the definition of $\Theta_1^\beta (\widehat f_n )$ is taken large enough.
\end{proof}

\begin{proof}[Proof of Theorem \ref{thm.glob2}]
The arguments in the proof always hold for sufficiently large $n$ although this is not always explicitly mentioned. Let $f^* \in \Theta$ be an arbitrary fixed parameter. In $(I)$ we prove the result for the Poisson intensity estimation experiment and in $(II)$ the result is extended to the Gaussian white noise experiment $\mE_n^G(\Theta).$

{\it (I):} We first construct two preliminary estimators $\widehat f_{1n}$ and $\widehat f_{2n}.$ Given $N \sim \Poi(n),$ let $N_1 \sim \Bin(N,1/2).$ Then $(X_1, \ldots, X_{N_1})$ and $(X_{N_1+1}, \ldots, X_N)$ are two independent samples from the same Poisson intensity estimation experiment with $n$ replaced by $n/2.$ If $N_1>n/4,$ construct the estimator satisfying the conclusions of Theorem \ref{thm.globalization_Poissonization} based on the subsample $(X_1, \ldots, X_{\lfloor n/4\rfloor})$ and denote this estimator by $\widehat f_{1n}.$ If $N_1\leq n/4,$ set $\widehat f_{1n}=f^*$. Let $L_n = n^{-1}\log n.$ By the conclusion of Theorem \ref{thm.globalization_Poissonization} and Lemma \ref{lem.switch_relation}, it follows for that some sufficiently large constant $C,$ the event
\begin{align*}
	\Omega:= \Big\{ |\widehat{f}_{1n}(x) -f_0(x)|\leq CL_n^{\beta/(\beta+1)}+C(f_0(x)L_n)^{\beta/(2\beta+1)} \ \text{for all}\ x\in [0,1]\Big\}
\end{align*}
has $\overline P_{f_0}^n$-probability $1-O(n^{-1}).$ Since by assumption $\inf_{f_0 \in \Theta} \inf_x f_0(x) \gg L_n^{\beta/(\beta+1)}$, it follows that $\tfrac 12 f_0 \leq  \widehat{f}_{1n} \leq 2f_0$ on $\Omega.$ Based on $\widehat f_{1n},$ we estimate the sequence \eqref{eq.xi_def}. Let $\widehat z_0:=0$ and $\widehat z_{i+1} := \widehat z_i + (\widehat f_{1n}(\widehat z_i)/n)^{1/(2\beta+1)}.$ Denote by $\widehat m$ the index of the largest $\widehat z_i$ smaller than $1$ and define $(\widehat x_i)_{i=0,\ldots,m}$ as $\widehat x_i:= \widehat z_i$ for  $i<\widehat m$ and $\widehat x_{\widehat m}:=1.$ In analogy with \eqref{eq.Deltai_def}, write $\widehat \Delta_i := \widehat x_i- \widehat x_{i-1} =  (\widehat f_{1n}(\widehat x_{i-1})/n)^{1/(2\beta+1)}+(1-\widehat z_{\widehat m})\mathbf{1}(i=\widehat m).$ Using the same arguments as for \eqref{eq.Deltam_ineqs} and \eqref{eq.loc_comp_of_fcts_in_lb}, we obtain that on $\Omega$ and for sufficiently large $n,$
\begin{align}
	(f_0(\widehat x_{j-1})/n)^{1/(2\beta+1)} \leq \widehat \Delta_j \leq 3(f_0(\widehat x_{j-1})/n)^{1/(2\beta+1)}
	\label{eq.hatDeltam_ineqs}
\end{align}
and
\begin{align}
	\frac{1}{2} f_0(\widehat x_{j-1}) \leq f_0(x) \leq 2f_0(\widehat x_{j-1}), \quad \text{for all} \ x\in [\widehat x_{j-1}, \widehat x_j]
	\label{eq.hatloc_comp_of_fcts_in_lb}
\end{align}
for all $j=1,\ldots, \widehat m.$

Let $N_i' := \# \{j\in \{N_1+1, \ldots , N\}: X_j\in [\widehat x_{i-1},\widehat x_i)\}$ be the number of counts in the interval $[\widehat x_{i-1},\widehat x_i)$ based on the second part of the sample. Thus, conditionally on $X_1,\ldots,X_{N_1},$ $N_i'$ follows a Poisson distribution with intensity $E[N_i' |X_1,\ldots,X_{N_1}]=\tfrac n2 \int_{\widehat x_{i-1}}^{\widehat x_i} f_0(u) du.$ Define the estimator
\begin{align}
	\widetilde{f}_{2n} = \sum_{i=1}^{\widehat m} \frac{2N_i'}{n\widehat \Delta_i} \mathbf{1}\big( \cdot \in [\widehat x_{i-1}, \widehat x_i) \big)
	\label{eq.def_widetildef2n}
\end{align}
and denote by $\widehat{f}_{2n}$ the projection of $\widetilde{f}_{2n}$ on $[\tfrac 12 \widehat f_{1n}(x), 2\widehat f_{1n}(x)]$, that is
\begin{align}
	\widehat{f}_{2n}(x)
	= 
	\big( \widetilde{f}_{2n}(x) \wedge 2\widehat f_{1n}(x) \big) \vee \frac{\widehat f_{1n}(x)}{2}.
	\label{eq.def_widehatf2n}
\end{align}
On $\Omega,$ $\tfrac 12 \widehat f_{1n} \leq f_0\leq  2\widehat f_{1n}$ and thus $\tfrac 14 f_0 \leq \widehat{f}_{2n}\leq 4f_0$ as well as $|\widehat{f}_{2n}(x) - f_0(x)|\leq |\widetilde{f}_{2n}(x)- f_0(x)|$ for all $x\in [0,1].$

We next show that on an event $\Omega_1$ with probability $P(\Omega_1)= 1 -O(n^{-1}),$ the estimator $\widehat f_{2n}(x)$ satisfies 
\begin{align}
	&n \int_0^1 \frac{(f_0(x)-\widehat f_{2n}(x))^4}{\widehat f_{2n}(x)^3} dx
	\leq C_2 n^{\frac{1-2\beta}{2\beta+1}} \int_0^1 \widehat f_{2n}(x)^{-\frac{2\beta+3}{2\beta+1}} dx
	\label{eq.widehatf2n_toshow}
\end{align}
for some constant $C_2$ which depends only on $R$ and $\beta.$ Let $\lambda_i := \tfrac n2 \int_{\widehat x_{i-1}}^{\widehat x_i} f_0(u) du,$ $\omega_i :=1/(n \widehat \Delta_i f_0(\widehat x_{i-1}))$ and $\eta_i := (N_i'-\lambda_i)/\sqrt{\lambda_i}.$  On $\Omega,$ using $f_0 \in \mH^\beta(R),$ \eqref{eq.hatDeltam_ineqs} and \eqref{eq.hatloc_comp_of_fcts_in_lb},
\begin{align}
	&n \int_0^1 \frac{(f_0(x) - \widehat f_{2n}(x))^4}{\widehat f_{2n}(x)^3} dx \notag \\
	&\leq 2^9 n \int_0^1 \frac{(f_0(x) - E[\widetilde f_{2n}(x)|X_1, \ldots,X_{N_1}])^4+(E[\widetilde f_{2n}(x)|X_1, \ldots,X_{N_1}] - \widetilde f_{2n}(x))^4}{f_0(x)^3} dx \notag \\
	&\leq 2^{12} n \sum_{i=1}^{\widehat m}  \frac{R^4\widehat \Delta_i^{1+4\beta}}{f_0(\widehat x_{i-1})^3}
	+ 2^{16} \sum_{i=1}^{\widehat m} \frac{\lambda_i^2\eta_i^4}{n^3 \widehat \Delta_i^3f_0(\widehat x_{i-1})^3}  \label{eq.glob_decomp_int} \\
	&\leq 3^{4\beta} 2^{12} R^4 n^{\frac{1-2\beta}{2\beta+1}}
	\sum_{i=1}^{\widehat m} 
	\widehat \Delta_i f_0(\widehat x_{i-1})^{-\frac{2\beta+3}{2\beta+1}}
	+  2^{16} \sum_{i=1}^{\widehat m} \omega_i \eta_i^4 \notag \\
	&\leq 3^{4\beta} 2^{15} R^4 n^{\frac{1-2\beta}{2\beta+1}}
	\int_0^1 f_0(x)^{-\frac{2\beta+3}{2\beta+1}} dx
	+  2^{16} \sum_{i=1}^{\widehat m} \omega_i \eta_i^4. \notag
\end{align}
Due to
\begin{align}
	\sum_{i=1}^{\widehat m } \omega_i
	\leq 
	n^{\frac{1-2\beta}{2\beta+1}}\sum_{i=1}^{\widehat m} \frac{\widehat \Delta_i}{f_0(\widehat x_{i-1})^{\frac{2\beta+3}{2\beta+1}}}
	\leq 
	8 n^{\frac{1-2\beta}{2\beta+1}} \int_0^1 f_0(x)^{-\frac{2\beta+3}{2\beta+1}} dx,
	\label{eq.sum_omega_i}
\end{align}
$\min_i \lambda_i \geq \min_i \tfrac 14 n \widehat \Delta_i f_0(\widehat x_{i-1})\geq \tfrac 14 n^{2\beta/(2\beta+1)} \inf_{f_0\in \Theta}\inf_x f_0(x)^{(2\beta+2)/(2\beta+1)}\rightarrow \infty$ and Lemma \ref{lem.Poisson_moments}(i), we find for some sufficiently large constant $C_1,$
\begin{align}
	n \int_0^1 \frac{(f_0(x) - \widehat f_{2n}(x))^4}{\widehat f_{2n}(x)^3} dx
	&\leq C_1 n^{\frac{1-2\beta}{2\beta+1}} \int_0^1 \widehat f_{2n}(x)^{-\frac{2\beta+3}{2\beta+1}} dx
	+ 2^{16} \sum_{i=1}^{\widehat m} \omega_i \big(\eta_i^4-E\big[\eta_i^4\big]\big).
	\label{eq.widetildef2n_ubd}
\end{align}
For the second term, we apply the exponential inequality in Lemma \ref{lem.exp_ineq_glob}. For that we firstly verify that $\| \omega\|_\infty \log^5 n \lesssim  \sum_i \omega_i.$ Set $f_* := \inf_x f_0(x)$ and $x_*\in \argmin_x f_0(x).$ For $K\in \{2,4\},$ denote by $I_K$ the largest interval such that $x_*\in I_K$ and $I_K \subset \{x:f_*\leq f_0(x) \leq Kf_*\}.$ Let us derive a lower bound for the cardinality of $\{i: \widehat x_{i-1}\in I_4\}.$ If $[\widehat x_{i-1}, \widehat x_i) \cap I_2 \neq \varnothing,$ then by \eqref{eq.hatloc_comp_of_fcts_in_lb}, $f(\widehat x_{i-1})\leq 4f_*$ for sufficiently large $n$ and so $\widehat x_{i-1} \in I_4.$ The cardinality of $\{i: \widehat x_{i-1}\in I_4\}$ can therefore be lower bounded by the cardinality of $\{i:[\widehat x_{i-1}, \widehat x_i) \cap I_2 \neq \varnothing \}.$ If $\widehat x_{i-1}\in I_4$ then by \eqref{eq.hatDeltam_ineqs}, $\widehat \Delta_i \leq 3(4f_*/n)^{1/(2\beta+1)}.$ Moreover by Lemma \ref{lem.fx_local_bd}, the Lebesgue measure of the set $I_2$ is at least $a(f_*/R)^{1/\beta}$ with $a$ the constant in Lemma \ref{lem.fx_local_bd}. This means that the cardinality of $\{i: \widehat x_{i-1}\in I_4\}$ is at least 
\begin{align*}
	\frac{a(f_*/R)^{\frac 1{\beta}}}{3 (4f_*/n)^{\frac 1{2\beta+1}}} = \frac{ a}{ 3 R^{\frac 1{\beta}}4^{\frac 1{2\beta+1}}} f_*^{\frac{\beta+1}{\beta(2\beta+1)}}n^{\frac 1{2\beta+1}} \gtrsim  \log^5 n,
\end{align*}
where for the last step we used that $\beta\mapsto (\beta+1)/(\beta(2\beta+1))$ is monotone decreasing for $\beta>0$ and that $\inf_{f\in \Theta} \inf_x f(x) \gg n^{-\beta/(\beta+1)}\log^8 n$ by assumption. Recall the definition of $\omega_i$ and observe that if $i \in I_4,$ the ratio $\omega_i/\|\omega\|_\infty$ is bounded from below by a constant. Consequently, $\| \omega\|_\infty\log^5 n \lesssim \sum_{i: \widehat x_{i-1}\in I_4} \omega_i \leq \sum_{i=1}^{\widehat m} \omega_i$ and the right-hand side can be further bounded using \eqref{eq.sum_omega_i}. By \eqref{eq.hatDeltam_ineqs}, \eqref{eq.hatloc_comp_of_fcts_in_lb} and Lemma \ref{lem.integral_comp}(ii), 
\begin{align*}
	\log n \|\omega\|_2  \lesssim \log n \Big(n^{\frac{1-4\beta}{2\beta+1}} \int_0^1 f_0(x)^{-\frac{4\beta+5}{2\beta+1}} dx \Big)^{1/2} \lesssim    n^{\frac{1-2\beta}{2\beta+1}}\int_0^1  f_0(x)^{-\frac{2\beta+3}{2\beta+1}} dx. 
\end{align*}
Since $\inf_{f_0 \in \Theta}\inf_x f_0(x) \gg n^{-\beta/(\beta+1)},$ we have $\widehat m = \sum_{i=1}^{\widehat m} \widehat \Delta_i / \widehat \Delta_i \leq n^{1/(\beta+1)}\sum_{i=1}^{\widehat m} \widehat \Delta_i = n^{1/(\beta+1)}$ for all sufficiently large $n.$ Thus, using Lemma \ref{lem.Poisson_moments}(i) and $\min_i \lambda_i \rightarrow \infty,$ we can apply the exponential inequality in Lemma \ref{lem.exp_ineq_glob} with $p=4$ and $t=2\log n$ to obtain
\begin{align*}
	\sum_{i=1}^{\widehat m} \omega_i \big(\eta_i^4-E\big[\eta_i^4\big]\big)
	\lesssim n^{\frac{1-2\beta}{2\beta+1}}\int_0^1  f_0(x)^{-\frac{2\beta+3}{2\beta+1}} dx
\end{align*}
with probability $\geq 1- \widehat m e^2/n^2 \geq 1-e^2/n.$ Together with \eqref{eq.widetildef2n_ubd}, this shows that there is a constant $C_2$ depending only on $\beta$ and $R$, such that 
\begin{align}
	n \int_0^1 \frac{(f_0(x) - \widehat f_{2n}(x))^4}{\widehat f_{2n}(x)^3} dx
	\leq C_2 n^{\frac{1-2\beta}{2\beta+1}} \int_0^1 \widehat f_{2n}(x)^{-\frac{2\beta+3}{2\beta+1}} dx
	\label{eq.glob2_f0_in_ball}
\end{align}
on an event $\Omega_1$ with probability $P(\Omega_1)\geq 1-e^2/n - P(\Omega^c) = 1 -O(n^{-1}).$ This proves \eqref{eq.widehatf2n_toshow}. 

As in the proof of Theorem \ref{thm.globalization_Poissonization}, we cover $\Theta \subset \mH^\beta(R)$ with finitely many balls of sup-norm radius $n^{-2}$ and centers in $\Theta.$ The estimator $\widehat f_n$ is then defined as any of the centers of the covering balls in the set
\begin{align}
	\Big\{ f\in \Theta : 
	&\frac 18 f \leq \widehat f_{2n} \leq 8f ,   \text{ and} \  n \int_0^1 \frac{(f(x) - \widehat f_{2n}(x))^4}{\widehat f_{2n}(x)^3} dx 
	\leq 8(C_2+2) n^{\frac{1-2\beta}{2\beta+1}} \int_0^1 \widehat f_{2n}(x)^{-\frac{2\beta+3}{2\beta+1}} dx \Big\}.
	\label{eq.set_final_est_glob2}
\end{align}
If none of the centers are in this set then set $\widehat f_n :=f^*.$

By construction, the estimator $\widehat f_n$ can take only finitely many values in the parameter space $\Theta.$ We now show that on the event $\Omega_1,$ $\widehat f_n$ lies in the set \eqref{eq.set_final_est_glob2}. By construction of the covering, it is enough to prove that on $\Omega_1,$ any $\tilde f\in \Theta$ with $\|\tilde f -f_0\|_\infty \leq n^{-2}$ is in the set \eqref{eq.set_final_est_glob2}. Let us work on $\Omega_1.$ Since $\inf_{f_0\in \Theta} \inf_x f_0(x) \gg 4n^{-\beta/(\beta+1)}\geq 4n^{-1}$ and $\tfrac 14 f_0\leq  \widehat f_{2n}\leq 4f_0,$ it follows that $\widehat f_{2n} \geq 1/n$ and $\tfrac 18 \widetilde f \leq \widehat f_{2n} \leq 8 \widetilde f.$ Observe that $(\widetilde f (x)- \widehat f_{2n}(x))^4\leq 8(\widetilde f(x)-f_0(x))^4 + 8(f_0(x) - \widehat f_{2n}(x))^4 \leq 8n^{-8} + 8(f_0(x) - \widehat f_{2n}(x))^4.$ Using \eqref{eq.glob2_f0_in_ball} and that $\|\hat{f}_{2n}\|_{L^\infty}\leq 4R$,
\begin{align}
	n \int_0^1 \frac{(\widetilde f(x) -\widehat f_{2n}(x))^4}{\widehat f_{2n}(x)^3} dx  
	&\leq 8n^{-4}+8C_2 n^{\frac{1-2\beta}{2\beta+1}} \int_0^1 \widehat f_{2n}(x)^{-\frac{2\beta+3}{2\beta+1}} dx \notag \\
	&\leq 8(C_2+o(1)) n^{\frac{1-2\beta}{2\beta+1}} \int_0^1 \widehat f_{2n}(x)^{-\frac{2\beta+3}{2\beta+1}} dx
	\label{eq.widetildefn_in_discr_set}
\end{align}
for sufficiently large $n.$ Thus on $\Omega_1,$ $\widehat f_n$ is in the set \eqref{eq.set_final_est_glob2}. We also know that $\tfrac 18 \widehat f_n \leq \widehat f_{2n} \leq 8\widehat f_n,$ which together with $\tfrac 14 f_0 \leq  \widehat f_{2n} \leq 4f_0$ gives $2^{-5}\widehat f_n \leq f_0  \leq 2^5 \widehat f_n.$ By the triangle inequality $|f_0(x) -\widehat f_n(x)|\leq |f_0(x) - \widehat f_{2n}(x)|+|\widehat f_{2n}(x)- \widehat f_n(x)|$ and using \eqref{eq.set_final_est_glob2} and \eqref{eq.widetildefn_in_discr_set},
\begin{align*}
	n \int_0^1 \frac{(f_0(x)-\widehat f_{n}(x))^4}{\widehat f_{n}(x)^3} dx
	\leq Cn^{\frac{1-2\beta}{2\beta+1}} \int_0^1 \widehat f_{n}(x)^{-\frac{2\beta+3}{2\beta+1}} dx
\end{align*}
for some sufficiently large constant $C,$ which proves that $f_0 \in \Theta^\beta(\widehat f_n).$

{\it (II):} By the same argument as in the proof of Lemma \ref{lem.sample_splitting}, we know that observing $(Y_t)_{t\in [0,1]}$ with $dY_t = 2\sqrt{f(t)} dt + n^{-1/2} dW_t,$ $t\in [0,1],$ is equivalent to observing two independent processes $(Y_{i,t})_{t\in [0,1]},$ $i=1,2,$ with $dY_{i,t}= \sqrt{f(t)} dt + n^{-1/2} dW_{i,t},$ $t\in [0,1],$ and $W_{i,t}$ independent Brownian motions. Instead of observing one process with noise level $n^{-1/2},$ we can thus rewrite the experiment such that we observe two independent processes with $n$ replaced by $n/2.$ By Theorem 1 in \cite{ray2016IP}, there exists an estimator $\widehat f_{3,n}$ based on $(Y_{1,t})_{t\in [0,1]}$ and a constant $C_3$ depending only on $\beta$ and $R$, such that $\inf_{f_0 \in \Theta} Q_{f_0}^n (\widetilde \Omega)= 1-o(n^{-1})$ with
\begin{align*}
	\widetilde \Omega := \Big\{\big|\widehat f_{3,n}(x) -f_0(x)\big|\leq C_3 L_n^{\frac{\beta}{\beta+1}}+C_3 \big(f_0(x)L_n\big)^{\frac{\beta}{2\beta+1}}, \ \text{for all} \ x\in [0,1]\Big\}.
\end{align*}
Throughout the remaining proof, we work on the event $\widetilde \Omega.$ Replace $\widehat f_{1n}$ by $\widehat f_{3n}$ in the construction of the sequence $(\widehat x_{i})_{i=0,\ldots \widehat m}$ in part $(I),$ labelling the new sequence $(\widetilde x_i)_{i=0,\ldots, \widetilde m}.$ Define also $\widetilde \Delta_i = \widetilde x_i-\widetilde x_{i-1}.$ These sequences satisfy in particular the relations \eqref{eq.hatDeltam_ineqs} and \eqref{eq.hatloc_comp_of_fcts_in_lb} on $\widetilde \Omega,$ with $\widetilde x_i$ and $\widetilde \Delta_i$ replacing $\widehat x_i$ and $\widehat \Delta_i.$ Similarly to \eqref{eq.def_widetildef2n} and \eqref{eq.def_widehatf2n}, we define the estimators
\begin{align*}
	\widehat f_{4n} = \sum_{i=1}^{\widetilde m} \Big(\frac{Y_{2, \widetilde x_i}-Y_{2, \widetilde x_{i-1}}}{\widetilde \Delta_i}\Big)^2 
	\mathbf{1}\big( \cdot \in [\widetilde x_{i-1}, \widetilde x_i)\big)
\end{align*}
and $\widehat{f}_{4n}(x)= ( \widetilde{f}_{4n}(x) \wedge 2\widehat f_{3n}(x) ) \vee \tfrac 12 \widehat f_{3n}(x).$ Thus on $\widetilde \Omega,$ $\tfrac 14 f_0 \leq \widehat f_{4n} \leq 4f_0$ and $|\widehat{f}_{4n}(x) - f_0(x)|\leq |\widetilde{f}_{4n}(x) - f_0(x)|$ for all $x\in [0,1].$ The next step is then to show that \eqref{eq.widehatf2n_toshow} holds with probability $1-O(1/n)$ and $\widehat f_{2n}$ replaced by $\widehat f_{4n}.$ To show this notice that for $x\in [\widetilde x_{i-1},\widetilde x_i],$
\begin{align*}
	\widetilde f_{4n}(x)|(Y_{1,t})_t \,{\buildrel d \over =}\, \Big(\frac 1{\widetilde \Delta_i}  \int_{\widetilde x_{i-1}}^{\widetilde x_i} \sqrt{f_0(u)} du\Big)^2
	+ \frac 2{\sqrt{n} \widetilde \Delta_i^{3/2}}  \int_{\widetilde x_{i-1}}^{\widetilde x_i} \sqrt{f_0(u)} du \ \xi_i
	+ \frac{1}{n\widetilde \Delta_i} \xi_i^2,
\end{align*}
where $\xi_i \sim \mathcal{N}(0,1)$ are i.i.d. for $i=1, \ldots, \widetilde m$ and $\,{\buildrel d \over =}\,$ means equal in distribution. Using \ref{eq.hatloc_comp_of_fcts_in_lb} and the formula for the difference of two squares, the first term can be approximated by
\begin{align*}
	\Big| f_0(x) - \Big(\frac 1{\widetilde \Delta_i}  \int_{\widetilde x_{i-1}}^{\widetilde x_i} \sqrt{f_0(u)} du\Big)^2 \Big|
	&\leq \frac{1}{\widetilde \Delta_i} \int_{\widetilde x_{i-1}}^{\widetilde x_i} \frac{|f_0(x)-f_0(u)|}{\sqrt{f_0(x)}} du
	\Big( \sqrt{f_0(x)} + \frac 1{\widetilde \Delta_i}  \int_{\widetilde x_{i-1}}^{\widetilde x_i} \sqrt{f_0(u)} du \Big)\\
	&\leq 3R \widetilde \Delta_i^\beta.
\end{align*}
With the expression for $\widetilde f_{4n}(x)|(Y_{1,t})_t,$ the previous inequality and $\widetilde \omega_i :=1/(n \widetilde \Delta_i f_0(\widetilde x_{i-1})),$
\begin{align*}
	&n \int_0^1 \frac{(f_0(x) - \widehat f_{4n}(x))^4}{\widehat f_{4n}(x)^3} dx \notag \\
	&\leq 2^9 n \int_0^1 \frac{(f_0(x) - E[\widetilde f_{4n}(x)|(Y_{1,t})_t])^4+(E[\widetilde f_{4n}(x)|(Y_{1,t})_t] - \widetilde f_{4n}(x))^4}{f_0(x)^3} dx \notag \\
	&\leq 2^{15}3^4 R^4 \sum_{i=1}^{\widetilde m} \frac{\widetilde \Delta_i^{4\beta+1}}{f_0(\widetilde x_{i-1})^3}
	+2^{15} \sum_{i=1}^{\widetilde m} \widetilde \omega_i^3
	+ 2^{21} \sum_{i=1}^{\widetilde m} \widetilde \omega_i \xi_i^4 
	+2^{15}  \sum_{i=1}^{\widetilde m} \widetilde \omega_i^3 (\xi_i^2-1)^4.
\end{align*}
The same argument as for \eqref{eq.sum_omega_i} gives $\sum_{i=1}^{\widetilde m} \widetilde \omega_i \lesssim n^{(1-2\beta)/(2\beta+1)} \int f_0(x)^{-(2\beta+3)/(2\beta+1)} dx.$ Moreover, since $\inf_{f_0\in \Theta} \inf_{x} f_0(x) \gg n^{-\beta/(\beta+1)},$ also $\max_i \widetilde \omega_i \rightarrow 0.$  Similar arguments as in \eqref{eq.glob_decomp_int} show
\begin{align*}
	&n \int_0^1 \frac{(f_0(x) - \widehat f_{4n}(x))^4}{\widehat f_{4n}(x)^3} dx \\
	&\lesssim 
	n^{\frac{1-2\beta}{2\beta+1}} \int_0^1 \widehat f_{4n}(x)^{-\frac{2\beta+3}{2\beta+1}} dx
	+  \sum_{i=1}^{\widetilde m} \widetilde \omega_i \big(\xi_i^4 -E\big[\xi_i^4\big]\big)
	+ \sum_{i=1}^{\widetilde m} \widetilde \omega_i^3 \big((\xi_i^2-1)^4 -E\big[(\xi_i^2-1)^4\big]\big).
\end{align*}
To control the second and third term, we apply Lemma \ref{lem.exp_ineq_glob} with $\epsilon_i = \xi_i$ and $\epsilon_i = \xi_i^2-1$ respectively. Notice that the moment condition in Lemma \ref{lem.exp_ineq_glob} is satisfied since  $E[(\xi_i^2-1)^r] \leq 2^r E[\xi_i^{2r}]+2^r = 2^r(2r)!/r!+2^r \leq 4^r r^r +2^r\leq 6^r r^r.$ Following exactly the same arguments as for $(I),$ we see that we can apply Lemma \ref{lem.exp_ineq_glob} and obtain in analogy with \eqref{eq.glob2_f0_in_ball} that
\begin{align}
	n \int_0^1 \frac{(f_0(x) - \widehat f_{4n}(x))^4}{\widehat f_{4n}(x)^3} dx 
	&\leq C_3 n^{\frac{1-2\beta}{2\beta+1}} \int_0^1 \widehat f_{4n}(x)^{-\frac{2\beta+3}{2\beta+1}} dx
	\label{eq.last_ineq}
\end{align}
holds with probability $1-O(1/n)$ for a constant $C_3$ that only depends on $\beta$ and $R.$ The final step is now to show that there is also an estimator $\widehat f_n$ which takes only finitely many values in $\Theta$ and also satisfies \eqref{eq.last_ineq} and $\tfrac 18 \widehat f_n \leq \widehat f_{4n} \leq 8\widehat f_n.$ The construction and analysis of this estimator is exactly the same as in the Poisson experiment considered in part $(I)$ and is therefore omitted. This completes the proof. 
\end{proof}

\section{Technical results}

\begin{lem}[Lemma 1 in \cite{RaySchmidt-Hieber2015c}]
\label{lem.fx_local_bd}
Suppose that $f \in \mH^\beta$ with $\beta>0$ and let $a = a(\beta)>0$ be any constant satisfying $(e^a-1) + a^\beta / (\lfloor \beta \rfloor!) \leq 1/2.$ Then for 
\begin{align*}
	|h| \leq a \left( \frac{|f(x)|}{\|f\|_{\mH^\beta}}\right)^{1/\beta},
\end{align*}
we have
\begin{equation*}
|f(x+h) - f(x)| \leq \frac{1}{2} |f(x)| ,
\end{equation*}
implying in particular, $|f(x)|/2 \leq |f(x+h)| \leq 3|f(x)|/2$.
\end{lem}

\begin{lem}
\label{lem.integral_comp}
%Let $\beta>0.$
\begin{itemize}
\item[(i)] If $(f_n)_n \subset \mH^\beta(R)$ is a sequence of functions such that $\inf_{x} f_n(x) \gg n^{-\beta/(\beta+1)}$ and $[x_{j_{1n}},x_{j_{2n}}]$ is as defined in \eqref{eq.xj1xj2_def}, then
\begin{align*}
	\int_{x_{j_{1n}}}^{x_{j_{2n}}} f_n(x)^{-\frac{3\beta+4}{2\beta+1}} dx \ll n^{\frac{1}{4\beta+2}}\Big( \int_{x_{j_{1n}}}^{x_{j_{2n}}}  f_n(x)^{-\frac{2\beta+3}{2\beta+1}} dx\Big)^{3/2}.
\end{align*}
\item[(ii)] If $(f_n)_n \subset \mH^\beta(R)$ is a sequence of functions such that $\inf_{x} f_n(x) \geq n^{-\beta/(\beta+1)}\log^3 n$ and $\beta\leq 1,$ then there is a constant $C$ that is independent of $(f_n)_n$ such that
\begin{align*}
	 \int_0^1 f_n(x)^{-\frac{4\beta+5}{2\beta+1}} dx \leq C  \frac{n^{\frac{1}{2\beta+1}}}{\log^2 n}\Big( \int_0^1  f_n(x)^{-\frac{2\beta+3}{2\beta+1}} dx\Big)^2.
\end{align*}
\end{itemize}
\end{lem}

\begin{proof}
{\it (i):} Set $m_n := \inf_{x} f_n(x)$ and $E= [x_{j_{1n}}, x_{j_{2n}}].$ Let $L_k$ be the Lebesgue measure of the set $\{x: 4^{k}m_n \leq f_n(x) < 4^{k+1}m_n\}\cap E$ and denote by $k^*$ the largest $k$ such that $L_k$ is positive. Then
\begin{align}
	\int_E f_n(x)^{-\frac{3\beta+4}{2\beta+1}} dx 
	\leq \sum_{k=0}^{k^*} L_k (4^k m_n)^{-\frac{3\beta+4}{2\beta+1}}.
	\label{eq.int_comp_1}
\end{align}
If $k=k^*,$ then $4^{k^*+1}m_n \geq 1,$ since by construction of $E,$ $\sup_{x\in E} f_n(x) \geq 1.$ Considering $L_k^* \lessgtr (4^{k^*}m_n)^{1/(2\beta+1)}n^{(\beta^*-1)/(2\beta+1)}$ with $\beta^*=\beta \wedge 1,$ gives
\begin{align}
	L_{k^*} (4^{k^*} m_n)^{-\frac{3\beta+4}{2\beta+1}}
	&\leq n^{\frac{\beta^*-1}{2\beta+1}} 4^{\frac{3\beta+3}{2\beta+1}}
	+ n^{\frac{1-\beta^*}{4\beta+2}} L_{k^*}^{3/2} (4^{k^*}m_n)^{-\frac{3\beta+9/2}{2\beta+1}} \notag \\
	&\leq n^{\frac{\beta^*-1}{2\beta+1}} 4^{\frac{3\beta+3}{2\beta+1}}
	+ 4^5 n^{\frac{1-\beta^*}{4\beta+2}} \Big( \int_E  f_n(x)^{-\frac{2\beta+3}{2\beta+1}} dx\Big)^{3/2} \notag \\
	&\ll n^{\frac{1}{4\beta+2}} \Big( \int_E  f_n(x)^{-\frac{2\beta+3}{2\beta+1}} dx\Big)^{3/2}.
	\label{eq.int_comp_2}
\end{align}
For the last step we used that  $\int_E  f_n(x)^{-\frac{2\beta+3}{2\beta+1}} dx \gtrsim 1\wedge n^{(2\beta-1)/(2\beta+1)},$ which follows from the definition of $E$ in \eqref{eq.xj1xj2_def} and the fact that $f_n$ is a density. If $k< k^*,$ then by continuity there is an $x\in E$ such that $f_n(x)= 2\cdot 4^k m_n$ and by Lemma \ref{lem.fx_local_bd}, $L_k\geq a (4^k m_n/R)^{1/\beta}.$ Since $\sum_i |a_i|^{3/2}\leq (\sum_i |a_i|)^{3/2},$
\begin{align}
	\sum_{k=0}^{k^*-1} L_k (4^k m_n)^{-\frac{3\beta+4}{2\beta+1}} 
	&\leq 
	\frac{R^{\frac{1}{2\beta}}}{\sqrt{a}} \sum_{k=0}^{k^*-1} L_k^{3/2} (4^k m_n)^{- \frac{1}{2\beta}-\frac{3\beta+4}{2\beta+1}} 
	\notag \\
	&\leq \frac{R^{\frac{1}{2\beta}}}{\sqrt{a}} m_n^{-\frac{\beta+1}{\beta(4\beta+2)}} \sum_{k=0}^{k^*-1} L_k^{3/2} (4^k m_n)^{-\frac{3\beta+9/2}{2\beta+1}} \notag \\
	&\ll n^{\frac{1}{4\beta+2}}\Big( \sum_{k=0}^{k^*-1} L_k (4^k m_n)^{-\frac{2\beta+3}{2\beta+1}}\Big)^{3/2} \notag \\
	&\leq 4^5 n^{\frac{1}{4\beta+2}}\Big( \int_E  f_n(x)^{-\frac{2\beta+3}{2\beta+1}} dx\Big)^{3/2}.
	\label{eq.int_comp_3}
\end{align}
Together with \eqref{eq.int_comp_1} and \eqref{eq.int_comp_2} this yields the assertion.

{\it (ii):} Applying the same argument as for \eqref{eq.int_comp_1} with $E=[0,1]$ gives $\int_0^1 f_n(x)^{-\frac{4\beta+5}{2\beta+1}}dx \leq \sum_{k=0}^{k^*} L_k (4^k m_n)^{-\frac{4\beta+5}{2\beta+1}}.$ If $k=k^*,$ it is enough to treat the two cases $L_k^* \lessgtr (4^{k^*}m_n)^{1/(2\beta+1)}$ and to argue as for \eqref{eq.int_comp_2} in order to find that 
\begin{align*}
	L_{k^*} (4^{k^*} m_n)^{-\frac{4\beta+5}{2\beta+1}}\lesssim 1+ \Big(\int_0^1 f_n(x)^{-\frac{2\beta+3}{2\beta+1}}dx\Big)^2 \lesssim  \frac{n^{\frac 1{2\beta+1}}}{\log^2 n}\Big(\int_0^1 f_n(x)^{-\frac{2\beta+3}{2\beta+1}}dx \Big)^2.
\end{align*}
Arguing as for \eqref{eq.int_comp_3} yields $\sum_{k=0}^{k^*-1} L_k (4^k m_n)^{-\frac{4\beta+5}{2\beta+1}} \lesssim m_n^{-\frac{\beta+1}{\beta(2\beta+1)}} ( \int_0^1 f_n(x)^{-\frac{2\beta+3}{2\beta+1}} dx)^2.$ Since $m_n \geq n^{-\beta/(\beta+1)}\log^3 n$ and $(\beta+1)/(\beta(2\beta+1))$ is monotone decreasing for $\beta>0,$ we find $m_n^{-\frac{\beta+1}{\beta(2\beta+1)}}\leq n^{1/(2\beta+1)}/\log^2 n$ and this completes the proof for $(ii).$
\end{proof}

\begin{lem}
\label{lem.Poisson_moments}
Let $N \sim Poi(\lambda).$ Then
\begin{itemize}
\item[(i)] For any integer $r>0,$ $E[|N-\lambda|^r]\leq r^r (1\vee \lambda)^{r/2}$ for all $\lambda>0,$
\item[(ii)] For $r>0,$ $E[N^{-r}\mathbf{1}(N>0)] = \lambda^{-r}+O(\lambda^{-r-1})$ as $\lambda \rightarrow \infty,$
\item[(iII)] For any $0\leq x \leq \lambda,$
\begin{align*}
	\P\big(|N-\lambda| > x) \leq 2 e^{-\frac{x^2}{2\lambda} + \frac{x^3}{2\lambda^2}}.
\end{align*}
\end{itemize}
\end{lem}

\begin{proof}
Part $(i)$ is proved by induction. The statement is clearly true for $r\leq 2.$ Now suppose it is true for $r\leq 2s.$ We want to show that it also holds for $r\leq 2(s+1).$ Consider first $r=2s+2.$ The $r$-th centralized moment satisfies the recurrence relation $E[(N-\lambda)^r]= \lambda \sum_{k=0}^{r-2}\binom{r-1}{k}E[(N-\lambda)^k]$ (cf. the proof of Lemma 3.1 in \cite{Privault2011}). Thus, $E[(N-\lambda)^{2s+2}]\leq (1\vee \lambda)^{s+1} \sum_{k=0}^{2s}\binom{2s+1}{k}(2s)^k\leq (1\vee \lambda)^{s+1} (2s+1)^{2s+2}.$ This shows that the statement also holds for $r=2s+2.$ For $r=2s+1,$ we apply Jensen's inequality and obtain $E[|N-\lambda |^{2s+1}] \leq [E(N-\lambda)^{2s+2}]^{(2s+1)/(2s+2)} \leq   (1\vee \lambda)^{s+1/2} (2s+1)^{2s+1},$ completing the proof of the induction step. Statement $(ii)$ is a consequence of Corollary 4 in \cite{Znidaric2005}. Let us now prove $(iii).$ Using exponential moments gives for any $t>0,$ $\P(N> \lambda + x) \leq e^{\lambda (e^t -1-t)-tx}.$ Optimizing over $t>0$ gives $t=\log ((\lambda+x)/\lambda)$ and using that $-\log(1+z) \leq -z+\tfrac 12 z^2$ for $z>0,$ yields $\P(N> \lambda + x) \leq e^{x-(x+\lambda)\log(\frac x{\lambda} +1) } \leq e^{-\frac{x^2}{2\lambda} + \frac{x^3}{2\lambda^2}}.$ Writing $P(N < \lambda-x) = P(-(N-\lambda)>x)$ and following the same steps as above gives $P(N-\lambda < -x) \leq e^{x-(x+\lambda)\log(1+\frac x{\lambda})}\leq e^{-\frac{x^2}{2\lambda} + \frac{x^3}{2\lambda^2}}.$ 
\end{proof}

\begin{lem}
\label{lem.exp_ineq_glob}
Let $m \geq 3$ and suppose that $\epsilon_i,$ $i=1,\ldots,m,$ are independent random variables satisfying $E[|\epsilon_i|^r]\leq A^r r^r$ for all $i=1,\ldots,m$ and all integers $r \geq 2.$ For positive weights $\omega_1, \ldots, \omega_m,$ integer $p\geq 1$ and any $t>0,$
\begin{align}
	P\Big( \sum_{i=1}^m \omega_i \big(\epsilon_i^p -E[\epsilon_i^p]\big)  \geq 2e(2Ap)^p \max(\|\omega\|_2, \|\omega \|_\infty t^p)  t \Big) 
	\leq m e^{2-t}.
	\label{eq.exp_ineq_for_weighted_sums}
\end{align}
\end{lem}

\begin{proof}
Let $q$ be an even, positive integer and $\xi_1, \ldots, \xi_m$ be independent, centered random variables with bounded $q$-th moment. Applying Lemma \ref{lem.Poisson_moments}(i) to bound the explicit constant in Rosenthal's inequality that is derived in Ibragimov and Sharkhmetov \cite{ibragi2002}, we have
\begin{align}
	E\big[ \big(\sum_{i=1}^m \xi_i\big)^q \big] 
	\leq q^q \max\Big( \sum_{i=1}^m E\big[ \xi_i^q\big] , \big(\sum_{i=1}^m E\big[\xi_i^2\big]\big)^{q/2}\Big).
	\label{eq.IS_moment_ineq}
\end{align}
We now apply this to show \eqref{eq.exp_ineq_for_weighted_sums}. There is nothing to prove in the case $t\leq 2.$ Thus it is enough to consider $t>2.$ Let $q$ be now the largest even integer smaller than $t$ and observe that in particular, $q\geq 2$ as well. The moment bound \eqref{eq.IS_moment_ineq} gives
\begin{align*}
	E\Big[ \Big( \sum_{i=1}^m \omega_i \big(\epsilon_i^p -E[\epsilon_i^p]\big) \Big)^q \Big]
	&\leq 
	q^q \max\Big( \sum_{i=1}^m (2\omega_i)^q (Apq)^{pq}, \big( \sum_{i=1}^m \omega_i^2 (2Ap)^{2p} \big)^{q/2} \Big) \notag \\
	&\leq 
	q^q 2^q (2Ap)^{pq} m \max\big( \|\omega\|_\infty q^p, \|\omega\|_2\big)^q.
\end{align*}
Taking both sides in the inequality to the power $q$ and applying Markov's inequality yields
\begin{align*}
	P\Big( \sum_{i=1}^m \omega_i \big(\epsilon_i^p -E[\epsilon_i^p]\big)  \geq 2e(2Ap)^p \max(\|\omega\|_2, \|\omega \|_\infty t^p)  t \Big) 
	\leq m e^{-q} \leq m e^{2-t}.
\end{align*}
\end{proof}

\begin{lem}
\label{lem.switch_relation}
Suppose that there are positive sequences $(a_n)_n,$ $(b_n)_n$ and $(r_n)_n$ such that for some $\beta>0$ and a positive constant $C,$ 
\begin{align*}
	|a_n -b_n |\leq C r_n^{\beta/(\beta+1)} + C (a_n r_n)^{\beta/(2\beta+1)}.
\end{align*}
Then there exists a finite constant $\widetilde C$ that only depends on $C$ and $\beta,$ such that
\begin{align*}
	|a_n -b_n |\leq \widetilde C r_n^{\beta/(\beta+1)} + \widetilde C (b_n r_n)^{\beta/(2\beta+1)}.
\end{align*}
\end{lem}

\begin{proof}
Without loss of generality, we can assume that $C\geq 1.$ If $a_n \geq (4C)^{(2\beta+1)/(\beta+1)} r_n^{\beta/(\beta+1)},$ then $|a_n -b_n |\leq C r_n^{\beta/(\beta+1)} + C (a_n r_n)^{\beta/(2\beta+1)}\leq a_n/4 +a_n/4 \leq a_n/2$ and therefore $a_n \leq 2b_n.$ In this case we thus obtain $|a_n -b_n |\leq C r_n^{\beta/(\beta+1)} + C (2b_n r_n)^{\beta/(2\beta+1)}.$ Otherwise, if $a_n \leq (4C)^{(2\beta+1)/(\beta+1)}  r_n^{\beta/(\beta+1)},$ then $|a_n -b_n | \leq C(1+ (4C)^{\beta/(\beta+1)}) r_n^{\beta/(\beta+1)}.$
\end{proof}

\section{Brief overview of the Le Cam deficiency}
\label{sec.LeCam}

We briefly recall some basic facts about the Le Cam deficiency. General treatments can be found in \cite{str,Torgersen1991,LeCamY2000,mariucci2016}.  

Following \cite{nussbaum1996}, Definition 9.1, we call a statistical experiment $\mE(\Theta)=(\Omega, \mA, (P_\theta: \theta \in \Theta))$ dominated if there exists a probability measure $\mu$ such that any $P_\theta$ is dominated by $\mu.$ Moreover, $\mE(\Theta)$ is said to be Polish if $\Omega$ is a Polish space and $\mA$ is the associated Borel $\sigma$-algebra. If $\mE(\Theta)=(\Omega, \mA, (P_\theta: \theta \in \Theta))$ and $\mF(\Theta)=(\Omega', \mA', (Q_\theta: \theta \in \Theta))$ are two Polish and dominated experiments, the Le Cam deficiency can be defined as 
\begin{align*}
	\delta\big(\mE(\Theta) , \mF(\Theta)\big) := \inf_M \sup_{\theta\in \Theta} \big\| MP_\theta^n - Q_\theta^n \big\|_{\TV},
\end{align*} 
where the infimum is taken over all Markov kernels from $(\Omega, \mA)$ to $(\Omega', \mA'),$ see (68) and Proposition 9.2 of \cite{nussbaum1996}. For any three statistical experiments with the same parameter space, the Le Cam deficiency satisfies the triangle inequality (cf. the proof of Lemma 59.2 in \cite{str}). The Le Cam distance 
\begin{align*}
	\Delta\big(\mE(\Theta) , \mF(\Theta)\big) := \delta\big(\mE(\Theta) , \mF(\Theta)\big) \vee \delta\big(\mF(\Theta) , \mE(\Theta)\big)
\end{align*}
thus defines a pseudo-distance on the space of all experiments with parameter space $\Theta.$ 

To derive bounds for the Le Cam deficiency, a common strategy is to construct intermediate experiments that embed both statistical models into a common probability space. Once the experiments are defined on the same measurable space, taking $M$ to be the identity yields (cf.\cite{tsybakov2009}, Lemmas 2.3 and 2.4)
\begin{align}
	\Delta\big(\mE(\Theta) , \mF(\Theta)\big) \leq 
	\sup_{\theta\in \Theta} \big\| P_\theta^n - Q_\theta^n \big\|_{\TV}
	\leq \sup_{\theta\in \Theta} H\big( P_\theta^n , Q_\theta^n\big)
	\leq \sup_{\theta\in \Theta} \sqrt{\KL\big( P_\theta^n , Q_\theta^n\big)},
	\label{eq.LeCam_ub_on_same_prob_space}
\end{align}
where $H$ and $\KL$ denote the Hellinger distance and the Kullback-Leibler divergence respectively. Bounding the Le Cam distance therefore often reduces to bounding information measures. In the next lemma we collect a number of facts that we use repeatedly in this article.

\begin{lem}
\phantomsection\label{lem.bds_of_info_distances}
\begin{itemize}
\item[(i)] Denote by $\overline P_{\Lambda}$ the distribution of the Poisson process with intensity measure $\Lambda.$ If $\nu$ is a measure that dominates $\Lambda_1$ and $\Lambda_2$ and $\lambda_j = d\Lambda_j/d\nu,$ then$$H^2(\overline P_{\Lambda_1},\overline P_{\Lambda_2})= \int (\sqrt{\lambda_1(x)}-\sqrt{\lambda_2(x)})^2 d\nu(x).$$
\item[(ii)] For a function $b$ and $\sigma>0,$ denote by $Q_{b, \sigma}$ the distribution of the path $(Y_t)_{t\in [0,1]}$ with $dY_t= b(t)dt+ \sigma dW_t,$ where $W$ is a Brownian motion. If $\Phi$ denotes the c.d.f. of the standard normal distribution, then$$\| Q_{b_1, \sigma}- Q_{b_2, \sigma}\|_{\TV} = 1-2\Phi(-\tfrac 1{2\sigma} \|b_1 -b_2\|_2),$$ $$H^2(Q_{b_1, \sigma}, Q_{b_2, \sigma}) = 2-2\exp(-\tfrac 1{8\sigma^{2}} \|b_1 -b_2\|_2^2),$$ $$\KL(Q_{b_1, \sigma}, Q_{b_2, \sigma}) = \tfrac 1{2 \sigma^{2}} \|b_1 -b_2\|_2^2.$$
\end{itemize}
\end{lem}

\begin{proof}
For a proof of $(i),$ see \cite{LeCamY2000}, p. 67 and \cite{meister2013}. Part $(ii)$ follows from Girsanov's formula $dQ_{b, \sigma}/dQ_{0, \sigma} = \exp(\sigma^{-1}\int b(t) dW_t - \tfrac 12 \sigma^{-2}\| b\|_2^2)$ together with $\|P-Q\|_{\TV}= 1-P(\tfrac{dQ}{dP}>1)-Q(\tfrac{dP}{dQ}\geq1)$ and $H^2(P,Q)=2-2\int(dPdQ)^{1/2}.$
\end{proof}

For upper bounds on the Le Cam distance, we use the localization technique described in Section 3 of \cite{nussbaum1996}, which we briefly recall here. A sequence of experiments $\mE_n(\Theta)=(\Omega_n, \mA_n, (P_\theta^n: \theta \in \Theta))$ is said to allow sample splitting if $P_\theta^n = P_\theta^{\lfloor n/2\rfloor} \otimes P_\theta^{\lceil n/2\rceil},$ that is if the sample can be split into two independent samples of size $\lfloor n/2\rfloor$ and $\lceil n/2\rceil.$ Moreover given $\mE_n(\Theta)$, define the sub-experiment $\mE_n(\Theta'):=(\Omega_n, \mA_n, (P_\theta^n: \theta \in \Theta'))$ for any $\Theta'\subset \Theta.$

\begin{lem}
\label{lem.localization_bd}
Suppose that for any $n\geq 2,$ $\mE_n(\Theta)=(\Omega_n, \mA_n, (P_\theta^n: \theta \in \Theta))$ and $\mF_n(\Theta)=(\Omega_n', \mA_n', (Q_\theta^n: \theta \in \Theta))$ are Polish experiments which are dominated and allow sample splitting. Let $\widehat \theta_{1,n}$ and $\widehat \theta_{2,n}$ be two estimators based on a sample from $P_{\theta}^{\lfloor n/2\rfloor}$ and $Q_{\theta}^{\lceil n/2\rceil}$ respectively and assume that $\widehat \theta_{1,n}$ and $\widehat \theta_{2,n}$ only take values in a finite subset of $\Theta.$ For any $\theta \in \Theta,$ denote by $U_n(\theta) \subset \Theta$ a neighbourhood of $\theta.$ Then, for $n\geq 4,$
\begin{align*}
	&\Delta\big( \mE_n(\Theta), \mF_n(\Theta) \big) \\
	&\leq 8\sup_{\theta\in \Theta} \Big(\max_{r\in \{\lfloor n/2\rfloor,  \lceil n/2\rceil\}}\Delta\big( \mE_r(U_n(\theta)), \mF_r(U_n(\theta)) \big) 
	+ P_{\theta}^{\lfloor n/2\rfloor}\big(\theta \notin U_n(\widehat \theta_{1,n})\big)
	+ Q_{\theta}^{\lceil n/2\rceil}\big(\theta \notin U_n(\widehat \theta_{2,n})\big)\Big).
\end{align*}
\end{lem}

\begin{proof}
We split the sample $P_\theta^n = P_\theta^{\lfloor n/2\rfloor} \otimes P_\theta^{\lceil n/2\rceil}$ and construct the estimator $\widehat \theta_{1,n}$ based on the sub-sample from $P_\theta^{\lfloor n/2\rfloor}.$ Define a new statistical experiment $\mG_n(\Theta)=(\Omega_{\lfloor n/2\rfloor}\times \Omega_{\lceil n/2\rceil}',\mA_{\lfloor n/2\rfloor}\otimes \mA_{\lceil n/2\rceil}', (P_\theta^{\lfloor n/2\rfloor} \otimes Q_\theta^{\lceil n/2\rceil}: \theta \in \Theta))$ and observe that $\mG_n(\Theta)$ is also Polish and dominated. By Lemma 9.3 in \cite{nussbaum1996} (last display on p. 2427), it follows that
\begin{align*}
	\Delta\big( \mE_n(\Theta), \mG_n(\Theta) \big)
	&\leq 4 \sup_{\theta\in \Theta} \Big(\Delta\big( \mE_{\lceil n/2\rceil}(U_n(\theta)), \mF_{\lceil n/2\rceil}(U_n(\theta)) \big) 
	+ P_{\theta}^{\lfloor n/2\rfloor}\big(\theta \notin U_n(\widehat \theta_{1,n})\big) \Big).
\end{align*}
With the same arguments, 
\begin{align*}
	\Delta\big( \mG_n(\Theta), \mF_n(\Theta) \big)\leq 4 \sup_{\theta\in \Theta} \Big(\Delta\big( \mE_{\lfloor n/2\rfloor}(U_n(\theta)), \mF_{\lfloor n/2\rfloor}(U_n(\theta)) \big) 
	+ Q_{\theta}^{\lceil n/2\rceil}\big(\theta \notin U_n(\widehat \theta_{2,n})\big) \Big)
\end{align*}
and since $\Delta$ is a pseudo-distance, the result follows.
\end{proof}

The previous lemma essentially says that if the statistical experiments allow sample splitting and if $\theta$ can be estimated in both models with rate $\epsilon_n,$ then it is sufficient to bound the Le Cam distance on a local parameter space consisting of an $\epsilon_n$-neighbourhood of some arbitrary $\theta_0.$ Bounding the Le Cam distance on a local parameter space is often much more convenient since we can use the fact that any parameter $\theta$ is $\epsilon_n$-close to $\theta_0.$ If the estimation rate $\epsilon_n$ can be obtained with probability $1-\delta_n,$ then by Lemma \ref{lem.localization_bd} this localization step adds $O(\delta_n)$ to the global Le Cam distance. In the experiments studied in this article, $\delta_n$ is much smaller than the Le Cam distance between the local parameter spaces and so does not contribute to the global Le Cam rate.

\begin{lem}
\label{lem.sample_splitting}
Let $\Theta \subset \mH^\beta(R)$ for some $\beta>0.$ The statistical experiments $\mE_n^D(\Theta), \mE_n^P(\Theta)$ and $\mE_n^G(\Theta)$ defined in Section \ref{sec.main} are Polish, dominated and allow sample splitting. 
\end{lem}

\begin{proof}
The proof of Theorem 3.2 in \cite{nussbaum1996} shows that the experiments are Polish. The experiments are also dominated since $\sup_{f\in \mH^\beta(R)} \|f\|_\infty<\infty.$ The sample splitting property is obvious for density estimation $\mE_n^D(\Theta).$ Consider now $\mE_n^P(\Theta).$ Given $N \sim \Poi(\lambda),$ let $N' \sim \Bin(N,p_n)$ with $p_n = \lfloor n/2\rfloor/n.$ Then $(X_1, \ldots, X_{N'})$ and $(X_{N'+1}, \ldots, X_N)$ are two independent samples of the same Poisson intensity estimation experiment with $n$ replaced by $\lfloor n/2\rfloor$ and $\lceil n/2\rceil$ respectively. In the Gaussian white noise experiment $\mE_n^G(\Theta),$ we can use that a Brownian motion $W$ can be written as $W_t = (n^{-1}\lfloor n/2\rfloor)^{1/2}W_t^{(1)}+ (n^{-1}\lceil n/2\rceil)^{1/2}W_t^{(2)},$ $t>0,$ for two independent Brownian motions $W^{(1)}$ and $W^{(2)}.$ By Girsanov's theorem,
\begin{align*}
	\frac{dQ_f^n}{dQ_0^n} = \exp\Big( 2\sqrt{n} \int_0^1 \sqrt{f(t)} dW_t - 2n \big\|\sqrt{f}\big\|_2^2\Big)
	= \frac{dQ_f^{\lfloor n/2\rfloor}}{dQ_0^{\lfloor n/2\rfloor}}  \frac{dQ_f^{\lceil n/2\rceil}}{dQ_0^{\lceil n/2\rceil}}
\end{align*}
and this completes the proof for $\mE_n^G(\Theta).$
\end{proof}

\bibliographystyle{acm}    %acm   % (uses file "plain.bst")
\bibliography{bibhd}           % expects file "refsPart1.bib"

\end{document}